\documentclass[a4paper,10.5pt]{article}
\usepackage{dsfont}
\usepackage{amssymb}
\usepackage{latexsym}
\usepackage{amsmath}
\usepackage{xcolor}
\usepackage{comment}
\usepackage{amsthm}
\usepackage{graphicx}
\usepackage{layout} 
\usepackage{ulem}
\usepackage{enumerate}
\usepackage{bm}
\usepackage{bbm} 
\usepackage[bbgreekl]{mathbbol} 
\usepackage{authblk}
\usepackage{setspace}
\usepackage{algorithmic}
\usepackage{algorithm}
\usepackage{{booktabs}}
\usepackage{multirow}
\usepackage{tikz}
\usetikzlibrary{arrows.meta, positioning, shapes.geometric}
\usepackage{subcaption}
\newif\ifrs
\rstrue
\ifrs \usepackage{mathrsfs} \fi  
\newtheorem{theorem}{Theorem}[section]

\newtheorem{lemma}[theorem]{Lemma}

\newtheorem{proposition}[theorem]{Proposition}

\newtheorem{remark}[theorem]{Remark}

\numberwithin{equation}{section}
\newtheorem{theorem*}{Theorem}
\newtheorem{ass*}[theorem*]{Assumption}
\newtheorem{note*}[theorem*]{Note}
\newtheorem{lemma*}[theorem*]{Lemma}
\newtheorem{definition*}[theorem*]{Definition}
\newtheorem{proposition*}[theorem*]{Proposition}
\newtheorem{corollary*}[theorem*]{Corollary}
\newtheorem{remark*}[theorem*]{Remark}
\newtheorem{example*}[theorem*]{Example}
\numberwithin{equation}{section}

\newcommand{\ol}{\overline}

\allowdisplaybreaks[1]

\usepackage{tikz}
\usetikzlibrary{trees,positioning}

\tikzset{
  intnode/.style={circle,draw,inner sep=1.5pt},      
  leafnode/.style={rectangle,draw,inner sep=1.5pt},  
  freenode/.style={circle,draw,double,inner sep=1.5pt}, 
  psiEdge/.style={thick},         
  deltaEdge/.style={dashed},       
  freeEdge/.style={thick},         
  fixedEdge/.style={thin},         
  dedEdge/.style={thick,densely dashed}
}
\def\bd{\begin{description}}
                \def\ed{\end{description}}

\def\D2{\bbD_{2,\infty-}}

\def\D{{\bf D}}

\def\cala{{\cal A}}
\def\calb{{\cal B}}
\def\calc{{\cal C}}

\def\calf{{\cal F}}
\def\calg{{\cal G}}
\def\calh{{\cal H}}

\def\call{{\cal L}}

\def\caln{{\cal N}}

\def\calp{{\cal P}}

\def\calt{{\cal T}}

\def\calx{{\cal X}}

\def\sfp{{\sf p}}

\def\halflineskip{\vspace*{3mm}}

\def\be{\begin{equation}}
                \def\ee{\end{equation}}
\def\bea{\begin{eqnarray}}
                \def\eea{\end{eqnarray}}
\def\beas{\begin{eqnarray*}}
                \def\eeas{\end{eqnarray*}}
\def\bi{\begin{itemize}}
                \def\ei{\end{itemize}}

\def\bd{\begin{description}}
                \def\ed{\end{description}}
\def\pt{\partial_\theta}

\def\wit{\widetilde}
\def\wih{\widehat}

\DeclareMathOperator*{\essinf}{ess\,inf}
\def\ptaui{\partial_{\tau_i}}

\newcommand{\bbC}{{\mathbb C}}
\newcommand{\bbD}{{\mathbb D}}

\newcommand{\bbH}{{\mathbb H}}

\newcommand{\bbL}{{\mathbb L}}

\newcommand{\bbR}{{\mathbb R}}

\newcommand{\bbY}{{\mathbb Y}}
\newcommand{\bbZ}{{\mathbb Z}}

\newcommand{\indep}{
        \begin{picture}(12,12)(0,0)
                \put(1,0){\line(1,0){8}}
                \put(4,0){\line(0,1){8}}
                \put(6,0){\line(0,1){8}}
        \end{picture}
}

\title{Statistical inference for highly correlated stationary point processes and noisy bivariate Neyman-Scott processes
\footnote{
  This work was in part supported by Japan Science and Technology Agency CREST JPMJCR2115; Japan Society for the Promotion of Science Grants-in-Aid for Scientific Research Nos. 17H01702, 23H03354 (Scientific Research); Grant-in-Aid for JSPS Fellows (25KJ0933); World-leading Innovative Graduate Study for Frontiers of Mathematical Sciences and Physics; and by a Cooperative Research Program of the Institute of Statistical Mathematics.
  Sincere gratitude goes to Mr. Yoshito Date (EMC Healthcare Co., Ltd.), Mr. Naokatsu Hasegawa (Yayoi Kogyo Co., Ltd.), and Mr. Masaki Nonaka (TAUNS Laboratories, Inc.) for their valuable support.
}
}

\author[1,2]{Takaaki Shiotani}

\author[1,2,3]{Nakahiro Yoshida}
\affil[1]{Graduate School of Mathematical Sciences, University of Tokyo
\footnote{Graduate School of Mathematical Sciences, University of Tokyo: 3-8-1 Komaba, Meguro-ku, Tokyo 153-8914, Japan. e-mail: nakahiro@ms.u-tokyo.ac.jp}
        }
\affil[2]{Japan Science and Technology Agency CREST
       }
\affil[3]{The Institute of Statistical Mathematics
       }

\begin{document}
\maketitle

\begin{abstract}
  Motivated by estimating the lead-lag relationships in high-frequency financial data,
  we propose noisy bivariate Neyman-Scott point processes with gamma kernels (NBNSP-G).
  NBNSP-G tolerates noises that are not necessarily Poissonian and has an intuitive interpretation. Our experiments suggest that NBNSP-G can explain the correlation of orders of two stocks well.
  A composite-type quasi-likelihood is employed to estimate the parameters of the model. However, when one tries to prove consistency and asymptotic normality, NBNSP-G breaks the boundedness assumption on the moment density functions commonly assumed in the literature. Therefore, under more relaxed conditions, we show consistency and asymptotic normality for bivariate point process models, which include NBNSP-G. Our numerical simulations also show that the estimator is indeed likely to converge. \\
  {\flushleft{{\bf Keywords:} Point process, Neyman-Scott process, composite likelihood, lead-lag effect, high-frequency data} }
\end{abstract}

\section{Introduction}
Estimating lead-lag relationships between multiple financial assets is a fundamental problem in financial engineering.
In recent years, statistical methods for lead-lag estimation have evolved as high-frequency data at the millisecond and microsecond levels have become available.
Among these, many studies focus on price movements.
For example, Hoffmann et al. \cite{hoffmann2013estimation} have introduced a continuous-time price process model. They use the Hayashi-Yoshida (HY) covariance estimator~\cite{hayashi2005covariance}, which can deal with asynchronous observations, and investigate the asymptotic properties of their method.
Huth \& Abergel~\cite{huth2014high} proposed another empirical method based on the HY covariance estimator.
However, Hayashi~\cite{hayashi2017statistical} pointed out that such methods based on price movements may produce unstable results, possibly due to the influence of the microstructure noise.
Instead, Hayashi~\cite{hayashi2017statistical} used Dobrev \& Schaumburg (DS) estimator~\cite{dobrev2017high}, which only focuses on the correlation of order timestamps, to estimate lead-lag relationships between the same stocks in multiple markets. Although Dobrev \& Schaumburg~\cite{dobreva2023high} discuss some asymptotic properties of the contrast function for the DS estimator under the independence null, its statistical properties, like consistency and asymptotic normality of the estimator itself, are still unclear.
To ensure statistical adequacy, employing some parametric point process model is a natural deal.
Da Fonseca \& Zaatour~\cite{da2017correlation} utilized multivariate Hawkes point processes with exponential kernels to model the lead-lag relationship between futures and stocks. However, their model requires specifying the whole process by the parametric structure, and using single exponential kernels may be restrictive.
Therefore, our goal is to develop a point process model of order timestamp sequences of multiple financial assets, especially useful for lead-lag estimation.

To this end, we propose a tailored version of the Neyman-Scott process (NSP).
The NSP has been used to model various phenomena, starting with Neyman \& Scott~\cite{neyman1958statistical}, and studied extensively.
This model has a structure that each point of the latent ``parent'' Poisson process produces independent ``children'' according to some distribution.
The multivariate version of NSP with causal kernels seems suitable for modeling order sequences of related financial assets
because we may interpret the parent process as the arrival of information common with multiple assets, and the children as actual order sequences triggered by each parent point.
Recently, Hong \& Shelton~\cite{hong2022deep} used multivariate and layered versions of Neyman-Scott processes equipped with gamma kernels to model a series of occurrences of multiple types of events, such as Twitter retweets from different groups of users, earthquakes of several magnitudes, and homicides in different districts.
We could use the original model and estimation procedure in  Hong \& Shelton~\cite{hong2022deep} directly, but there is some room for improvement, especially when it comes to high-frequency financial data.
We will discuss several points below.
In the following, we restrict ourselves to the one-layer and bivariate version of their model for simplicity.
\begin{enumerate}
  \item \textbf{Choice of kernels}.
        We employ gamma kernels as Hong \& Shelton~\cite{hong2022deep} did.
        In recent years, high-frequency trading has been prominently present in the market, and it may react very quickly to trading opportunities.
        Thus, we express those rapid responses by the (possibly) diverging gamma kernels.
        Indeed, our experiments using actual stock order data in Section 8 show that the gamma kernel fits better than exponential kernels, which one may consider first as a kind of causal kernel.
  \item \textbf{Adding noise}.
        Financial data are generally noisy, so it seems too optimistic to assume that we can explain all the orders by only NSP.
        In order to express this character, we extend bivariate NSP by adding stationary independent noise to each component as follows:
        \begin{equation}\label{eq_rep_our_model}
          N_i = \sum_{c\in C}\sum_{j=1}^{M_i(c)}\delta_{c+d_i(c, j)} + N^B_i, \quad i=1, 2,
        \end{equation}
        where $\calc$ is the parent homogeneous Poisson process, which is common for both components,
        $M_i(c)$ is the random number of children generated by a parent point $c\in\calc$,
        $d_i(c, j)\sim \mathrm{Gamma}(\alpha_i, l_i)$ is the duration between a parent point and an offspring point,
        $N_i^B$ is a stationary noise process not necessarily Poissonian,
        and all variables are generated independently.
        Here we are interested in estimating the parameters of gamma kernels $(\alpha_i, l_i), i=1, 2$
        and not interested in the structure of noise processes.
        We also emphasize that by adding noises, our model is more flexible than fully parametric models such as Hawkes processes~\cite{da2017correlation}.
  \item \textbf{Estimation procedure}.
        We use a quasi-maximum likelihood estimation (QMLE) to estimate the parameters of our model, which is different from Hong \& Shelton~\cite{hong2022deep}.
        Hong \& Shelton~\cite{hong2022deep} developed a likelihood-based estimation procedure.
        However, our model includes noise whose structure is unknown, making the likelihood function non-explicit. Thus, we cannot apply their procedure directly.
        The minimum contrast estimation using the cross-K function~\cite{zhu2022minimum} could be used,
        but methods based on the K-function are criticized for their dependency on the choice of hyperparameters (Guan~\cite{Guan2006}).
        To address this problem, quasi-likelihood-based estimation methods such as composite likelihood~\cite{Guan2006,Waage2007} and Palm likelihood~\cite{Tanaka2008,Palm2013} have been extensively studied.
        For multivariate cases, Jalilian et al. \cite{jalilian2015multivariate} used a second-order composite likelihood to another specific model.
        Therefore, we adopt a composite-type quasi-likelihood (see Section \ref{sec_general_theory} for the definition of our objective function and estimator).
\end{enumerate}
Consequently, we extend bivariate NSP with gamma kernels by adding unknown noises and estimate its parameters by using the quasi-maximum likelihood estimation (QMLE). In the following, we will sometimes abbreviate our model as NBNSP-G (noisy bivariate Neyman-Scott process with gamma kernels).

As expected, our experiments in Section \ref{sec_realdata} suggest that NBNSP-G represents the correlation between the order timestamp series of two stocks better than the Hawkes process with exponential kernels.
However, the use of gamma kernels that possibly diverge at the origin results in the divergence of the cross-correlation function, that is, the two point processes are highly correlated.
Then, our model breaks the boundedness assumption of the moment density function for the asymptotic properties of the estimator, as commonly assumed in the literature (e.g.~\cite{Guan2006,Palm2013}).
Thus, we relax the boundedness assumptions of the moment density function and prove consistency and asymptotic normality for general bivariate point process models, and then
we verify that our model indeed satisfies these relaxed assumptions.
In other words, we come to be able to treat highly correlated point process modes such as NBNSP-G.
Moreover, we conduct numerical simulations and show that our parameter estimation method works well under various settings.

This paper is organized as follows.
In Section \ref{sec_preliminaries}, we summarize the basic concepts of point processes and tools necessary for the asymptotic theory.
In Section \ref{sec_general_theory}, we construct the estimator for the parameters of the model.
In Section \ref{sec_general_theory2}, we provide the asymptotic theory for general bivariate point process models.
In Section \ref{sec_NS_model}, we define the noisy BNSP model and demonstrate that, under appropriate assumptions, the general theory in Section \ref{sec_general_theory2} can be applied.
In Section \ref{sec_specific_models}, we introduce specific models including NBNSP-G and show that the asymptotic theory from Section \ref{sec_NS_model} can be applied.
In Section \ref{sec_simulation}, we perform numerical simulations under various settings.
In Section \ref{sec_realdata}, we conduct experiments using real high-frequency financial data.
In Section~\ref{sec_extension}, we discuss alternative estimation procedures and possible extensions of the model to nonstationary and multivariate settings.
Section~\ref{sec_proof} contains the proofs of the theoretical results.

\section{Preliminaries}\label{sec_preliminaries}
\subsection{Notation}
For a topological space $\calx$, $\calb(\calx)$ is the Borel $\sigma$-algebra.
For $x\in \bbR$ and $r>0$, $B(x, r)$ is the open ball centered at $x$ with radius $r$.
For $W\subset \bbR$ and $r>0$, $W\ominus r$ denotes $\{x\in W; B(x, r)\subset W\}$
and $W\oplus r$ denotes $\{x\in \bbR; \exists y\in W \text{ s.t. } y\in B(x, r)\}$.
For a $\bbR$-valued random variable $X$ and $p\geq 1$, $||X||_p$ is the $L^p$-norm.
For $A\in \calb(\bbR)$, $\mathrm{Leb}(A)$ is the value of the Lebesgue measure on $\bbR$.
For a countable set $I$, $\#I$ is the cardinality of $I$.
The Euclidean norm is denoted by $|\cdot|$ for a real-valued matrix or a vector.

For $a\in \bbZ_{\geq 1}$, let $\mathfrak{S}_a$ denote the symmetric group of degree $a$.
For integers $1\leq a\leq b$, let
\(
\widetilde{\calp}_a^b
=
\{\pi:\{1,\ldots,b\}\to\{1,\ldots,a\};\ \pi \text{ is surjective}\}.
\)
We define an equivalence relation $\sim$ on $\widetilde{\calp}_a^b$ by
\(
\pi_1\sim\pi_2
\Leftrightarrow
\exists\sigma\in\mathfrak{S}_a\ \mathrm{s.t.}\ \pi_1 = \sigma\circ\pi_2.
\)
Let $\calp_a^b = \widetilde{\calp}_a^b / \sim$ be the set of equivalence classes.
Elements of $\calp_a^b$ are identified with partitions of $\{1,\ldots,b\}$ into $a$ nonempty subsets.
In the subsequent formulas, we often index the blocks of a partition using integers $1, \ldots, a$.
To this end, for a partition $\pi \in \calp_a^b$, we adopt the convention of choosing a representative surjection (still denoted by $\pi \in \widetilde{\calp}_a^b$) and writing
\( \pi^{-1}(l) = \{k \in \{1,\ldots,b\}; \pi(k) = l\} \)
for the $l$-th block of the partition, $1\leq l\leq a$. For a partition $\pi$, $|\pi^{-1}(l)|$ is the size of $l$-th block.
The choice of the representative does not affect the values of the sums where these notations appear.
For $I, n_1,\ldots,n_I, m\in\bbZ_{\geq 1}$ such that $m\leq n_1+\cdots+n_I$, let
\(
\mathcal{D} = \bigsqcup_{i=1}^I\{1,\ldots,n_i\}
=
\{(i,j);\ i=1,\ldots,I,\ j=1,\ldots,n_i\}
\)
and
\(
\widetilde{\calp}_m^{n_1\sqcup\cdots\sqcup n_I}
=
\{
\pi:\mathcal{D}\to\{1,\ldots,m\};\ \pi\ \text{is surjective}
\}.
\)
Similarly, we define the set of partitions of $\mathcal{D}$ into $m$ subsets as the quotient space
\(
\calp_m^{n_1\sqcup\cdots\sqcup n_I}
=
\widetilde{\calp}_m^{n_1\sqcup\cdots\sqcup n_I}/\sim
\)
under the action of $\mathfrak{S}_m$.
For a partition $\pi\in\calp_m^{n_1\sqcup\cdots\sqcup n_I}$, we again refer to a representative surjection $\pi$ to denote the $l$-th component (block) of the partition by $\pi^{-1}(l)$ for $1\leq l\leq m$.
Furthermore, we decompose each block $\pi^{-1}(l)$ into its intersections with the $i$-th component of the disjoint union.
Specifically, for $1\leq i\leq I$, we define
\(
\pi^{-1}(l)_i
=
\{(i,j) \in \mathcal{D};\ \pi((i,j)) = l\}.
\)
This set represents the subset of points in the $i$-th dimension that belong to the $l$-th block.

\subsection{Point processes on the real line}
Let $(\Omega, \calf, P)$ be a probability space.
Suppose that
$\caln_{\bbR}^{\#}$
is the class of locally finite $\bbZ_{\geq 0}\cup\{\infty\}$-valued measures on $\bbR$,
$\calb(\caln_{\bbR}^{\#})$ is the smallest $\sigma$-field on $\caln_{\bbR}^{\#}$ with respect to which the mappings
$\Phi_A: \caln_{\bbR}^{\#} \to \bbZ_{\geq 0},  \mu \mapsto \mu(A)$ are measurable for all bounded Borel set $A\in \calb(\bbR)$.
The mapping $N: \Omega \to \caln_{\bbR}^{\#}$ is a point process if $N$ is $\calf / \calb(\caln_{\bbR}^{\#})$-measurable.
Since point processes can be seen as a countable set of random points (possibly with repetitions),
we sometimes use abbreviated notations such as $x\in N$ for a point process $N$.
Let $\caln_{\bbR}^{\#*} = \{\mu\in \caln_{\bbR}^{\#}; \forall x\in\bbR: \mu(\{x\}) \leq 1\}$.
A point process $N$ is simple if $P[N\in \caln_{\bbR}^{\#*}] = 1$.
Let $N_i: \Omega \to \caln_{\bbR}^{\#}, i=1, \ldots, I, I\in\bbZ_{i\geq 1}$ be point processes.
A family $(N_i)_{i=1}^I$ is a multivariate point process if the sum $\sum_{i=1}^I N_i$ is simple.
Let $S_u: \caln_{\bbR}^{\#} \to \caln_{\bbR}^{\#}$ is the shift operator defined by $S_u(\mu)(A) = \mu(A+u)$ for $A\in\calb(\bbR)$, $u\in\bbR$.
N is stationary if the finite dimensional distributions of $N$ and $S_u N$ coincide for all $u\in \bbR$, i.e.
\[
  P\bigl(
  N(A_1) = k_1, \ldots, N(A_l) = k_l
  \bigr)
  = P\bigl(
  N(A_1+u) = k_1, \ldots, N(A_l+u) = k_l
  \bigr)
\]
holds for all bounded $A_1, \ldots, A_l \in \calb(\bbR)$, $k_1, \ldots, k_l\in \bbZ_{\geq 0}$, $l\in \bbZ_{\geq 1}$, $u\in\bbR$.
A multivariate point process $(N_i)_{i=1}^I$ is stationary if
$$
  P\bigl(
  N_{i_1}(A_1) = k_1, \ldots, N_{i_l}(A_l) = k_l
  \bigr)
  = P\bigl(
  N_{i_1}(A_1+u) = k_1, \ldots, N_{i_l}(A_l+u) = k_l
  \bigr)
$$
holds for all bounded $A_1, \ldots, A_l \in \calb(\bbR)$, $k_1, \ldots, k_l\in \bbZ_{\geq 0}$, $l\in \bbZ_{\geq 1}$, $i_m\in \{1, \ldots, I\}$, $u\in\bbR$.
These formulations are based on Daley \& Vere-Jones~\cite{DaleyandVereJones2}.

\subsection{Moment measures}
Let $N=(N_i)_{i=1}^I$ be a multivariate point process.
In the following, we say $N$ has a $k$-th moment if
$
  \forall i\in\{1, \ldots, I\},
  \forall A \in \calb(\bbR), \text{$A$ is bounded}:
  E[N_i(A)^k] < \infty.
$
Suppose that $n_i\in \bbZ_{\geq 0}, i = 1, \ldots, I$ and $N$ has sufficient orders of moments.
The moment measure of order $(n_1, \ldots, n_I)$ of $N$ is defined by
\begin{eqnarray*}
  M_{(n_1, \ldots, n_I)}
  \Bigl(\prod_{i=1}^I \prod_{j=1}^{n_i} A_{i,j}\Bigr)
  = E\Bigl[
    \prod_{i=1}^I
    \Bigl(
    \int_{\bbR^{n_i}}
    1_{\{x_{i,1}\in A_{i,1},\ldots, x_{i,n_i}\in A_{i,n_i}\}}
    (N_i\times\cdots \times N_i)(dx_{i,1} \times \cdots \times dx_{i,n_i})
    \Bigr)
    \Bigr]
\end{eqnarray*}
for bounded $A_{i, j}\in \calb(\bbR), j=1, \ldots, n_i$, and $i = 1, \ldots, I$.

The factorial moment measure of order $(n_1, \ldots, n_I)$ of $N$ is defined by
\begin{eqnarray*}
  M_{[n_1, \ldots, n_I]}
  \Bigl(\prod_{i=1}^I \prod_{j=1}^{n_i} A_{i,j}\Bigr)
  = E\Bigl[
    \prod_{i=1}^I
    \Bigl(
    \int_{\bbR^{n_i}}
    1_{\Bigl\{\substack{x_{i,1}\in A_{i,1},\ldots, x_{i,n_i}\in A_{i,n_i} \\
          x_{i, a}\neq x_{i, b}, a, b \in\{1, \ldots, n_i\}}\Bigr\}}
    (N_i\times\cdots \times N_i)(dx_{i,1} \times \cdots \times dx_{i,n_i})
    \Bigr)
    \Bigr]
\end{eqnarray*}
for bounded $A_{i, j}\in \calb(\bbR), j=1, \ldots, n_i$, and $i = 1, \ldots, I$.

We will give relationships between the moment and the factorial moment measures as a generalization of Exercise 5.4.5 of Daley \& Vere-Jones~\cite{daley2003introduction}.
\begin{proposition}
  Suppose that $N=(N_i)_{i=1}^I$ is a multivariate point process with sufficient orders of moments, $n_i\in \bbZ_{\geq 0}$, and $i = 1, \ldots, I$. Then
  \begin{align}\label{eq_factorial_and_normal_moment}
     & \quad M_{(n_1, \ldots, n_I)}
    \Bigl(\prod_{i=1}^I \prod_{j=1}^{n_i} dx_{i,j}\Bigr)                        \nonumber                                   \\
     & = \sum_{m_1=1}^{n_1} \sum_{\pi_1 \in \calp_{m_1}^{n_1}} \cdots \sum_{m_I=1}^{n_I} \sum_{\pi_I \in \calp_{m_I}^{n_I}}
    \prod_{i=1}^I \prod_{l=1}^{m_i} \prod_{j\in \pi_i^{-1}(l)}\delta(y_{i, l} - x_{i, j})
    M_{[m_1, \ldots, m_I]}
    \Bigl(\prod_{i=1}^I\prod_{l=1}^{m_i} dy_{i, l}\Bigr)        \nonumber                                                   \\
     & = \sum_{m_1=1}^{n_1} \sum_{\pi_1 \in \calp_{m_1}^{n_1}} \cdots \sum_{m_I=1}^{n_I} \sum_{\pi_I \in \calp_{m_I}^{n_I}}
    \int_{\bbR^{Im_i}}
    \prod_{i=1}^I \prod_{l=1}^{m_i} \prod_{j\in \pi_i^{-1}(l)}\delta_{y_{i, l}}(dx_{i, j})
    M_{[m_1, \ldots, m_I]}
    \Bigl(\prod_{i=1}^I\prod_{l=1}^{m_i} dy_{i, l}\Bigr)
  \end{align}
\end{proposition}
\begin{proof}
  We derive the result by dividing cases by the number of distinct points ($=m_i$) and which points are duplicates ($\pi_i \in \calp_{m_i}^{n_i}$) for each $i=1, \ldots, I$.
\end{proof}
\halflineskip
We recall some formulae for stationary multivariate point processes.
Let $N=(N_i)_{i=1}^I$ be a stationary multivariate point process on $\bbR$.
Suppose that $\lambda_i$ is the intensity of $N_i$, and $\lambda_{i,j}(\cdot)$ is the cross-intensity function of $N_i$ and $N_j$, for $i, j = 1, \ldots, I, i\neq j$.
Then, for any non-negative Borel measurable function $h: \bbR\to\bbR$ and $D\in \calb(\bbR)$, we have
\[
  E\Bigl[\sum_{x\in N_i}h(x)\Bigr] = \lambda_i\int_{\bbR} h(u)du
\]
and
\begin{align}\label{cross_intensity_formula}
  E\Bigl[\sum_{\substack{x\in N_i, y\in N_j \\ x\in D}} h(y-x)\Bigr]
  = \mathrm{Leb}(D) \int_{\bbR} h(u) \lambda_{i,j}(u) du.
\end{align}
We will refer to these equations as Campbell's formulae.
The cross-correlation function $g_{i,j}(\cdot)$ of $N_i$ and $N_j$, $i, j = 1, \ldots, I, i\neq j$ is defined as
\[
g_{i,j}(u) = \frac{\lambda_{i,j}(u)}{\lambda_i\lambda_j}.
\]

\subsection{Moments and cumulants}
Let $(N_i)_{i=1}^I$ be a multivariate point process with sufficient order of moments.
The probability generating functional (p.g.fl) of $(N_i)_{i=1}^I$ is
\(
G(h_1, \ldots, h_I) = E\Bigl[\exp\Bigl(\sum_{i=1}^{I} \int_{\bbR} \log{h_i}dN_i\Bigr)\Bigr]
\)
where $h_i: \bbR \to (0, 1]$ is measurable function such that the support of $(1-h_i)$ is bounded for $i=1,\ldots, I$.

Suppose $h_i = \sum_{j=1}^{n_i}t_{ij}1_{A_{i,j}} + 1_{(\cup_{j=1}^{n_i}A_{i,j})^c}$ where
$\{A_{i,j}\}_{j=1}^{n_i}$ is a family of Borel sets such that
(
\(
\forall j_1, j_2 \in \{1, \ldots, n_i\}:
A_{i, j_1} = A_{i, j_2}\ \text{or} \ A_{i, j_1} \cap A_{i, j_2} = \emptyset
\)
)
and $0< t_{ij} \leq 1, j=1, \ldots, n_i$ for each $i=1, \ldots, I$.
We derive the factorial moment measure by taking derivatives at
$(t_{ij})_{j=1}^{n_i} =: \bm{t_i} = \bm{1} := (1, \ldots, 1)$:
\[
  \Bigl(
  \prod_{i=1}^I \partial_{t_{i1}}\cdots \partial_{t_{in_i}}
  \Bigr)
  G(h_1, \ldots, h_I)(\bm{t_1}, \ldots, \bm{t_I}) \Big|_{\bm{t_1}, \ldots, \bm{t_n} = \bm{1}}
  =  M_{[n_1, \ldots, n_I]}
  \Bigl(\prod_{i=1}^I \prod_{j=1}^{n_i} A_{i,j}\Bigr),
\]
$G(h_1, \ldots, h_I)$ regarded as a function of $\bm{t_1}, \ldots, \bm{t_I}$.
The factorial cumulant measure $C_{[n_1, \ldots, n_I]}$ of order $(n_1, \ldots, n_I)$ is defined by
\begin{align}\label{fcumt2fmom}
  C_{[n_1, \ldots, n_I]}\Bigl(\prod_{i=1}^I \prod_{j=1}^{n_i} dx_{i,j}\Bigr)
  =
  \sum_{m=1}^{n_1+\cdots +n_I} (-1)^{m-1} (m-1)! \sum_{\pi\in\calp_m^{n_1\sqcup \cdots \sqcup n_I}}
  \prod_{l=1}^m
  M_{[|\pi^{-1}(l)_1|, \ldots, |\pi^{-1}(l)_I|]}
  \Bigl(\prod_{(i, j)\in \pi^{-1}(l)}dx_{i,j} \Bigr).
\end{align}
We can obtain the factorial cumulant measure $C_{[n_1, \ldots, n_I]}$ from the logarithm of the p.g.fl:
\begin{equation}\label{eq_cumulant_general_formula}
  \Bigl(
  \prod_{i=1}^I \partial_{t_{i1}}\cdots \partial_{t_{in_i}}
  \Bigr)
  \log(G(h_1, \ldots, h_I)(\bm{t_1}, \ldots, \bm{t_I}))\Big|_{\bm{t_1}, \ldots, \bm{t_n} = \bm{1}}
  =  C_{[n_1, \ldots, n_I]}
  \Bigl(\prod_{i=1}^I \prod_{j=1}^{n_i} A_{i,j}\Bigr).
\end{equation}
The factorial cumulant measures express the factorial moment measure as
\begin{align}\label{fmomt2fcum}
  M_{[n_1, \ldots, n_I]}\Bigl(\prod_{i=1}^I \prod_{j=1}^{n_i} dx_{i,j}\Bigr)
  =
  \sum_{m=1}^{n_1+\cdots +n_I} \sum_{\pi\in\calp_m^{n_1\sqcup \cdots \sqcup n_I}}
  \prod_{l=1}^m
  C_{[|\pi^{-1}(l)_1|, \ldots, |\pi^{-1}(l)_I|]}
  \Bigl(\prod_{(i, j)\in \pi^{-1}(l)}dx_{i,j} \Bigr).
\end{align}

\subsection{Mixing}
Suppose that $(\Omega, \calf, P)$ is a probability space and $\calg, \calh \subset \calf$ are $\sigma$-algebras.
The $\alpha$-mixing coefficient of $\calg$ and $\calh$ is
\(
\alpha(\calg, \calh) = \sup \{|P(C \cap D) - P(C)P(D)|; C \in \calg, D \in \calh \}.
\)
Let $\bbL$ be a countable subset of $\bbZ$.
The $\alpha$-mixing coefficient of a random field $Z = \{Z(l)\}_{l\in\bbL}$ is
\begin{align*}
  \tilde{\alpha}_{c_1, c_2}^Z(m; \bbL) =
  \sup \{
   & \alpha(\sigma(Z(l); l\in I_1), \sigma(Z(l); l\in I_2));       \\
   & I_1\subset\bbL, I_2\subset\bbL, \#I_1\leq c_1, \#I_2\leq c_2,
  d(I_1, I_2)\geq m
  \}, \quad m, c_1, c_2 \geq 0.
\end{align*}
We will write $\tilde{\alpha}_{c_1, c_2}^Z(m) = \tilde{\alpha}_{c_1, c_2}^Z(m; \bbZ)$.
Suppose that $N = (N_i)_{i=1}^I$ is a multivariate point process, $N\cap A = \{N_i\cap A\}_{i=1}^I = \{N_i(\cdot\cap A)\}_{i=1}^I$ is the restriction of $N$ to $A\in \calb(\bbR)$,
$C(l)$ is an interval of side length $1$ centered at $l\in\bbZ$, i.e.
$
  C(l) = (l-\frac{1}{2}, l+\frac{1}{2}],
$
and
$
  D(A) = \{l\in \bbZ;  C(l)\cap A\neq \emptyset\}
$
for any $A\subset \bbR$.
The $\alpha$-mixing coefficient of $N$ is
\begin{align*}
  \alpha_{c_1, c_2}^N(m; r) =
  \sup \Bigl\{
   & \alpha(\sigma(\{N_i\cap E_1\}_{i=1}^I), \sigma(\{N_i\cap E_2\}_{i=1}^I));   \\
   & E_1 = \bigcup_{l\in M_1}C(l)\oplus r, E_2 = \bigcup_{l\in M_2}C(l)\oplus r, \\
   & \#M_1\leq c_1, \#M_2\leq c_2,
  d(M_1, M_2)\geq m, M_1, M_2\subset \bbZ
  \Bigr\}, \quad m, c_1, c_2, r \geq 0.
\end{align*}
This definition is a generalization to the multivariate case of the one appearing in Prokešová et al. \cite{prokevsova2017two}, p. 528.
We note that, if a random field $\{X_l\}_{l\in\bbZ}$ satisfies $X_l\in\sigma(N\cap C(l)\oplus r)$ for each $l\in\bbZ$, then
\begin{equation}\label{eq_note_on_mixing1}
  \tilde{\alpha}_{c_1, c_2}^X (m)
  \leq \alpha_{c_1, c_2}^N (m-2r-2; r)
\end{equation}
for all $c_1, c_2, r\geq 0$ and $m\geq2r+2$.

\subsection{Moment inequalities}
We refer to a Rosenthal-type inequality for use in the proof of the asymptotic properties of the estimator.
\begin{theorem}[Moment inequality, Doukhan {\cite{Doukhan1994}} {pp. 25-26}]\label{mixing_moment_ineq}
  \rm
  Suppose that $X = \{X_t\}_{t \in \bbZ}$ is a family of random variables indexed by $\bbZ$, $T$ is a finite subset of $\bbZ$, and $L > 2$.
  If there exist $\epsilon > 0$ and an even integer $c$ larger than $L$ such that
  \begin{gather}
    \label{mixing_condition1}
    \forall u, v \in \bbZ_{\geq 2}, u+v \leq c:  \quad
    \sum_{k=1}^{\infty} (1+k)^{c-u}
    \tilde{\alpha}_{u, v}^X(k)^{\frac{\epsilon}{c + \epsilon}} < \infty, \\
    \forall t \in \bbZ: \quad
    ||X_t||_{L+\epsilon} < \infty , \quad E[X_t] = 0,
  \end{gather}
  then there is some constant $C$ only depending on $L$ and on the $\alpha$-mixing coefficient of $X$ such that
  \(
  E[|\sum_{t \in T} X_t|^L]
  \leq C \times \max\{M(L, \epsilon, T),  M(2, \epsilon, T)^{\frac{L}{2}}\}
  \)
  where
  \(
  M(L, \epsilon, T) = \sum_{t \in T} \|X_t\|_{L+\epsilon}^{L}.
  \)
\end{theorem}
\halflineskip
\section{Construction of the estimator}\label{sec_general_theory}
For the sake of parameter estimation of a point process model on $\bbR$, it is usual to use the conditional intensity function.
However, no explicit form of the conditional intensity is available for models of our interest, such as Neyman-Scott type models discussed in Section \ref{sec_NS_model}, thus we use the quasi-likelihood function based on the moment density functions as in the literature (e.g.~\cite{Guan2006,Tanaka2008}).

Let $N = (N_1, N_2)$ be a stationary bivariate point process on $\bbR$ with intensities $\lambda_i, i=1, 2$ and a parametric cross-correlation function $g_{1, 2}(\cdot; \theta)=g(\cdot; \theta)$, where $\theta$ is the parameter.
Then, the cross-intensity function is $\lambda_{1,2}(\cdot; \theta) = \lambda_1\lambda_2 g(\cdot; \theta)$.
Suppose that $W \in \calb(\bbR)$ is the bounded observation window.
In order to estimate the parameter $\theta$, we will maximize the following quasi-likelihood function $\wit{\bbH}(\theta; W)$ with respect to $\theta$.
\begin{align*}
  \wit{\bbH}(\theta; W)
   & = \sum_{\substack{x\in N_1, y\in N_2                              \\ |y-x|\leq r \\ x\in W\ominus r}}
  \log(\lambda_{1,2}(y-x; \theta))
  - \mathrm{Leb}(W\ominus r) \int_{|u|\leq r}\lambda_{1,2}(u; \theta)du \\
   & = \sum_{\substack{x\in N_1, y\in N_2                              \\ |y-x|\leq r \\ x\in W\ominus r}}
  \Bigl(
  \log(g(y-x; \theta)) + \log\lambda_1 + \log\lambda_2
  \Bigr)
  - \mathrm{Leb}(W\ominus r)
  \lambda_1\lambda_2\int_{|u|\leq r}g(u; \theta)du.
\end{align*}
We restrict the range of $x$ to $W\ominus r$ from $W$ because the inner edge bias correction guarantees the unbiasedness of the score function.
We do not generally know the true value of $\lambda_1$ and $\lambda_2$.
However, if we estimate them by some estimator $\wih{\lambda}_i$, we can use
\begin{align*}
  \bbH(\theta; W)
  = \sum_{\substack{x\in N_1, y\in N_2 \\ |y-x|\leq r \\ x\in W\ominus r}}
  \Bigl(
  \log(g(y-x; \theta)) + \log\wih{\lambda}_{1, n} + \log\wih{\lambda}_{2, n}
  \Bigr)
  - \mathrm{Leb}(W\ominus r)
  \wih{\lambda}_{1, n}\wih{\lambda}_{2, n}\int_{|u|\leq r}g(u; \theta)du
\end{align*}
instead of $\wit{\bbH}(\theta; W)$.
We call an estimator $\widehat{\theta}$ that maximizes $\bbH(\theta; W)$ a quasi-maximum likelihood estimator (QMLE).
Theoretical details will be given in the next section.
This construction of the quasi-log likelihood function is based on the idea of two-step estimation (e.g.~\cite{choiruddin2021two,prokevsova2017two}).

We will introduce a more concrete setting for the asymptotic theory.
Suppose that the parameter space $\Theta \subset \bbR^p$ is a bounded open set, $\{W_n\}_{n=1}^{\infty}$ is a sequence of increasing compact subset of $\bbR$.
Let us assume that the parameter of interest of our stationary bivariate point process model $(N_1, N_2)$ on $\bbR$ is the parameter $\theta\in\ol{\Theta}$ and that the cross-correlation function $g(\cdot; \theta)$ is parametrized by $\theta\in\ol{\Theta}$.
Moreover, we suppose that the true values of the intensities $\lambda_1$, $\lambda_2$, and parameter $\theta$ are $\lambda_1^*>0$, $\lambda_2^*>0$, and $\theta^*\in\Theta$, respectively.
Fix a user-specified parameter $r>0$.
Let us write $a_n = \mathrm{Leb}(W_n\ominus r)$, $D_n = D(W_n\ominus r)$, and
\(
\wih{\lambda}_{i,n} = \frac{1}{a_n}N_i(W_n\ominus r).
\)
Suppose
\(
F_2^N(h; V) = \sum_{\substack{x\in N_1, y\in N_2 \\ |y-x|\leq r, x\in V}}h(y-x)
\)
for $V\in\calb(\bbR)$ and a measurable function $h:\bbR\to\bbR^d$, $d\geq 1$.
We note that, if $h(\cdot)g(\cdot; \theta^*)\in \mathrm{L}^1([-r, r])$, then
\begin{equation*}
  E[F_2^N(h; V)]
  = \mathrm{Leb}(V) \lambda_1^*\lambda_2^*\int_{|u|\leq r} h(u)g(u; \theta^*)du,
\end{equation*}
because of Campbell's formula (\ref{cross_intensity_formula}).
We sometimes write
$F_{2, n}^N(h) = F_2^N(h; W_n\ominus r)$ for ease of notation.
Note that $F_2^N(h; \cdot)$ is $\sigma$-additive on $\calb(\bbR)$ and $F_2^N(\cdot; V)$ is linear on the space of all Borel measurable functions.
Then, the quasi-likelihood function is
\begin{align*}
  \bbH_n(\theta)
   & = \sum_{\substack{x\in N_1, y\in N_2                                      \\ |y-x|\leq r \\ x\in W_n\ominus r}}
  \Bigl(
  \log(g(y-x; \theta)) + \log\wih{\lambda}_{1, n} + \log\wih{\lambda}_{2, n}
  \Bigr)
  - a_n \wih{\lambda}_{1, n}\wih{\lambda}_{2, n}\int_{|u|\leq r}g(u; \theta)du \\
   & = F_{2, n}^N
  \Bigl(
  \log(g(y-x; \theta)) + \log\wih{\lambda}_{1, n} + \log\wih{\lambda}_{2, n}
  \Bigr)
  - a_n \wih{\lambda}_{1, n}\wih{\lambda}_{2, n}\int_{|u|\leq r}g(u; \theta)du.
\end{align*}
We call a measurable map $\wih{\theta}_n$ satisfying
\begin{equation}\label{eq_def_QMLE}
  \wih{\theta}_n \in \underset{\theta \in \ol{\Theta}}{\rm{argmax}} \ \bbH_n(\theta)
\end{equation}
a quasi-maximum likelihood estimator (QMLE).

\section{Asymptotic Theory}\label{sec_general_theory2}
In this section, we will show the consistency and the asymptotic normality of the QMLE.
Guan~\cite{Guan2006} and Prokešová \& Jensen~\cite{Palm2013} showed asymptotic properties of a similar kind of QMLE for general stationary univariate point processes, and their proof techniques could be used for our multivariate model.
However, the bounded conditions they imposed on the log derivatives of the moment densities are too strong for our concern, especially when we deal with the noisy bivariate Neyman-Scott process (NBNSP) with possibly diverging kernels, which will be introduced in Section \ref{sec_specific_models}. Instead, we consider certain integrability conditions that allow the divergence of the moment densities. (See also the discussion below.)

First, we recall that
$g(\cdot; \theta)$ is the cross-correlation function parametrized by $\theta\in\Theta\subset \mathbb{R}^p$,
$
  C(l) = (l-\frac{1}{2}, l+\frac{1}{2}],
$
and
$
  D(A) = \{l\in \bbZ;  C(l)\cap A\neq \emptyset\}, A\subset \bbR.
$
We will consider the following conditions:
\begin{description}
  \item{\bf{[WI]}}
        $
          W_1 \subset W_2 \subset \cdots \subset \bbR
        $,
        each $W_n$ is bounded,  $\mathrm{Leb}(\cup_{n=1}^{\infty} (W_n\ominus r)) = \infty$, and
        $\frac{\#D(W_n\ominus r)}{\mathrm{Leb}(W_n\ominus r)} = \frac{\#D_n}{a_n}\to 1$.
  \item{\bf{[PA]}}
        The parameter space $\Theta\subset \bbR^p$ is bounded, open, and convex.
  \item{\bf{[ID]}}
        For all $\theta\in\ol{\Theta}$,
        $g(\cdot; \theta) = g(\cdot; \theta^*)$ a.e. on $[-r, r]$ implies $\theta=\theta^*$.
  \item{\bf{[RE]}}
        \begin{description}
          \item{(i)}
                $ \displaystyle
                  \essinf_{|u|\leq r} \inf_{\theta\in\ol{\Theta}}g(u; \theta) > 0.
                $
          \item{(ii)}
                $\displaystyle
                  g(u; \cdot) \in C(\ol{\Theta}) \cap C^2(\Theta)
                $ for all $u\in [-r, r]$.
          \item{(iii)}
                There exists a measurable function $f_{B, 1}: [-r, r]\to \bbR_{\geq0}$ such that
                \begin{align*}
                  \max \{\sup_{\theta\in\Theta}|\pt^i g(\cdot; \theta)|\}_{i=0}^2
                  \leq f_{B, 1}(\cdot)
                \end{align*} on $[-r, r]$ and
                \begin{equation*}
                  \int_{|u|\leq r}f_{B, 1}(u)du < \infty.
                \end{equation*}
          \item{(iv)}
                There exists a measurable function $f_{B, 2}: [-r, r]\to \bbR_{\geq0}$ such that
                \begin{align*}
                  \max \{\sup_{\theta\in\Theta}|\pt^i \log g(\cdot; \theta)|\}_{i=0}^2
                  \leq f_{B, 2}(\cdot)
                \end{align*} on $[-r, r]$ and
                \begin{equation*}
                  \int_{|u|\leq r}f_{B, 2}(u)g(u; \theta^*)du < \infty.
                \end{equation*}
        \end{description}
  \item{\bf{[MI]}}
        There exists $\delta>0$ such that
        \begin{equation*}
          \sum_{m=1}^{\infty}\alpha_{2, \infty}^N(m; r)^{\frac{\delta}{2+\delta}} < \infty,
        \end{equation*}
        \begin{equation*}
          \|N_i(C(0))\|_{2+\delta} < \infty, \quad i=1, 2,
        \end{equation*}
        and
        \begin{equation*}
          \Bigl\|F_2^N\Bigl(|\pt^j\log g(\cdot; \theta)|; C(0)\Bigr)\Bigr\|_{2+\delta} < \infty, \quad j=0, 1, 2, \ \theta\in\Theta.
        \end{equation*}
\end{description}
\halflineskip

Here we discuss the assumptions.
The main advantages over the existing literature are the integrability requirements in [RE](iii)–(iv).
For comparison, following the approach of Prokešová \& Jensen~\cite{Palm2013} or Guan~\cite{Guan2006}, one would assume
\begin{equation}\label{eq_cond_prokesova1}
  \sup_{\theta\in\Theta,\ |u|\le r}\ \bigl|\pt \log g(u;\theta)\bigr|<\infty,
\end{equation}
and 
\begin{equation}\label{eq_cond_prokesova2}
  \lim_{\delta\to 0}\ \sup_{\substack{|u|\le r,\\ \theta_1,\theta_2\in\Theta,\\ |\theta_1-\theta_2|<\delta}}
  \bigl|\pt^2 \log g(u;\theta_1)-\pt^2 \log g(u;\theta_2)\bigr|<\infty.
\end{equation}
The assumption (\ref{eq_cond_prokesova1}) corresponds to a assumption used in Theorem~1 of Prokešová \& Jensen~\cite{Palm2013} and in Theorem~1 of Guan~\cite{Guan2006}, and
the assumption (\ref{eq_cond_prokesova2}) corresponds to the assumption~(14) for Theorem~3 in Prokešová \& Jensen~\cite{Palm2013} and the assumption~(11) for Theorem~2 in Guan~\cite{Guan2006}.
By contrast, very roughly,
$g(u;\theta)\asymp 1 + C|u|^{\alpha(\theta)}$ as $|u|\to 0$
with $\alpha(\theta)>-1$ and $C>0$
for the model NBNSP
as shown in the proof of Lemma~\ref{lem_NS_RE}.
Hence, when $0>\alpha(\theta)>-1$, we have
\(\pt \log g(u;\theta)\approx 
\pt g(u; \theta) /  g(u; \theta)
\approx (C\pt \alpha(\theta)|u|^{\alpha(\theta)}\log|u|) / (1 + C|u|^{\alpha(\theta)})
\simeq \log|u|
\) and similarly $\pt^2 \log g(u;\theta) \simeq (\log|u|)^2$
as $|u|\to 0$.
Since $u\mapsto \log|u|$ diverges at the origin, supremum-based bounds such as
\eqref{eq_cond_prokesova1} and \eqref{eq_cond_prokesova2} are not satisfied.
Thus, the integrability conditions in [RE](iii)–(iv) are tailored to allow such logarithmic divergences.
The mixing and moment conditions in [MI] stem from Bolthausen’s central limit theorem (Lemma \ref{thm_bolthausen_clt}). Although the required mixing rate would be stronger than that used in the blocking argument of Prokešová \& Jensen~\cite{Palm2013} (see Remark 5 therein), we adopt Bolthausen’s CLT to keep the proof simple.
In any case, the models considered in Section \ref{sec_specific_models} are geometrically mixing, so this stronger requirement is not restrictive for our purposes.
The condition [WI] is slightly weaker than $W_n$ are just expanding intervals $[0 ,T_n]$ or $[-T_n, T_n]$, $T_n\to\infty$. 
For instance, $W_n$ may consist of several disjoint subintervals separated by small gaps, as $W_n = \cup_{i=1}^n [iT_n+\delta, (i+1)T_n-\delta], \delta > 0, T_n\to\infty$, which still satisfies the assumptions.
The condition [PA] is standard for asymptotic theory, and the condition [ID] will be checked for specific models in Section \ref{sec_specific_models}.

Under these assumptions, the QMLE is consistent. 
\begin{theorem}\label{thm_consistency}
  Assume the conditions [WI], [PA], [ID], [RE], and [MI].
  Then the QMLE is consistent, i.e. $\wih{\theta}_n\to^p \theta^*$ as $n\to\infty$.
\end{theorem}
\halflineskip

For the asymptotic normality, we will restrict the shape of $W_n$ to ensure the convergence of the variance of the score function and impose positive definiteness on the limit of the observed information.
\begin{description}
  \item{\bf{[WI2]}}
        $W_n = [0, T_n]$.
  \item{\bf{[ID2]}}
        The matrix
        $\Gamma = \lambda_1^*\lambda_2^*\int_{|u|\leq r} \frac{\pt g(u; \theta^*)^{\otimes 2}}{g(u; \theta^*)}du$ is positive definite.
\end{description}
\halflineskip

Then, we have the asymptotic normality of the QMLE.
\begin{theorem}\label{thm_asymptotic_normality}
  Assume the conditions [WI2], [PA], [ID], [ID2], [RE], and [MI].
  Then, we have
  \[
    \sqrt{a_n}(\wih{\theta}_n - \theta^*) \to N(0, \Gamma^{-1}\Sigma\Gamma^{-1}),
  \]
  as $n\to\infty$
  for some nonnegative definite matrix $\Sigma$.
\end{theorem}
\halflineskip
The proofs will be given in Section \ref{subsec_prf_general}.

\begin{remark} \rm
    Note that the condition [WI2] is stronger than [WI].
  Also, even if the condition [WI2] is not satisfied, we can still derive the convergence of the distribution as
  \(
  \sqrt{a_n}\Sigma_n^{-\frac{1}{2}}\Gamma(\wih{\theta}_n - \theta^*) \to^d N(0, 1_p),
  \)
  when $\liminf_{n}\lambda_{\min}(\Sigma_n) > 0$. See Biscio \& Waagepetersen~\cite{biscio2019general} for more information.
  However, investigating $\Sigma_n$ is not easy when the observation window $W_n$ has a complicated shape.
\end{remark}
\section{Noisy bivariate Neyman-Scott process}\label{sec_NS_model}
In this section, we first introduce the bivariate Neyman-Scott point process and calculate its various characteristic quantities.
Next, we introduce the noisy bivariate Neyman-Scott point process (NBNSP) and identify the parametric structure.
Finally, under appropriate assumptions, we prove that the asymptotic theory from the previous sections can be applied to our model NBNSP.

\subsection{The bivariate Neyman-Scott process}
The bivariate Neyman-Scott process $(N_1^S, N_2^S)$ is constructed as follows.
First, suppose $\calc$ is a Poisson process with intensity $\lambda>0$, called parent below.
For $i=1, 2$, each parent points $c$ yields a random number $M_i(c)$ of offspring points which are realized independently and identically from the distribution of a $\bbZ_{\geq 0}$-valued random variable $M_i$ with the probability generating function $g_i(\cdot)$.
The offspring points from a parent point $c$ are independently and identically distributed around $c$ according to the probability density function $f_i(\cdot - c)$, where $f_i(\cdot)$ is a probability density function, which is sometimes called a dispersal kernel.
We denote the realization of the offspring point by $c + d_j(c, m)$.
Then the $i$-th component is given by
\begin{equation}\label{eq_our_model_signal}
  N_i^S =  \sum_{c \in \calc} \sum_{m = 1}^{M_i(c)} \delta_{c + d_i(c, m)}
  =: \sum_{c \in \calc} N_{i, c}^S.
\end{equation}
for $i=1, 2$.

In the context of lead-lag relationships in financial engineering, we try to model the ``trigger'' of co-occurrence of orders in the same direction for two assets by the parent $\calc$, and the statistical differences in the response speed to the ``trigger'' by the dispersal kernels $f_1$ and $f_2$.

Let $N^S=(N_i^S)_{i=1}^2$ be the bivariate Neyman-Scott process.
Suppose that $h_i: \bbR \to (0, 1]$ is measurable function such that $\mathrm{supp}(1-h_i)$ is bounded for $i=1, 2$. The probability generating functional of the bivariate Neyman-Scott process is given by
\begin{align*}
  G(h_1, h_2) =
   & \quad E\Bigl[\exp\Bigl(\sum_{i=1}^2 \int_{\bbR} \log{h_i}dN_i^S\Bigr)\Bigr] \\
   & = E\Bigl[\exp\Bigl(
    \sum_{c\in\calc}\sum_{i=1}^2 \int_{\bbR} \log{h_i}dN_{i, c}^S
  \Bigr)\Bigr]                                                                   \\
   & = E\Bigl[\prod_{c\in\calc} \prod_{i=1}^2
    E[\exp\Bigl(\int_{\bbR} \log{h_i}dN_{i, c}^S\Bigr) | \calc]
  \Bigr] \quad (\because \mathrm{conditional\ independence})                     \\
   & = E\Biggl[\prod_{c\in\calc} \prod_{i=1}^2
    E\Bigl[\prod_{m=1}^{M_i(c)}
      E[h_i(c+d_i(c, m)) | M_i(c), \calc]
      \Big|\calc\Bigr]
  \Biggr]                                                                        \\
   & =E\Biggl[\prod_{c\in\calc} \prod_{i=1}^2
    E\Bigl[
      \Bigl(\int_{\bbR}h_i(c+u_i)f_i(u_i)du_i \Bigr)^{M_i(c)}
      \Big|\calc\Bigr]
  \Biggr]                                                                        \\
   & =E\Bigl[\prod_{c\in\calc}
    \prod_{i=1}^2 g_i\Bigl(
    \int_{\bbR}h_i(c+u_i)f_i(u_i)du_i
    \Bigr)
  \Bigr]                                                                         \\
   & =\exp\Biggl(\int_{\bbR}
  \lambda \Bigl\{\prod_{i=1}^2 g_i\Bigl(
  \int_{\bbR}h_i(c+u_i)f_i(u_i)du_i
  \Bigr)
  -1
  \Bigr\}dc
  \Biggr).
\end{align*}
Especially we find
\[
  E\Bigl[\exp\Bigl(\sum_{i=1}^2 \int_{\bbR} \log{h_i(x)}N_i^S(dx)\Bigr)\Bigr]
  = E\Bigl[\exp\Bigl(\sum_{i=1}^2 \int_{\bbR} \log{h_i(x+u)}N_i^S(dx)\Bigr)\Bigr]
\]
for all $u\in\bbR$.
This equation implies the stationarity of the multivariate Neyman-Scott process.

We can verify the density of the factorial cumulant measure (with respect to the Lebesgue measure) is
\begin{equation}\label{eq_cumulant_NS}
  \gamma_{[n_1, n_2]}(x_{1,1}, \ldots, x_{1,n_1}, x_{2,1}, \ldots, x_{2,n_2})
  = \lambda \int_{\bbR} \prod_{i=1}^2 \Bigl(g_i^{(n_i)}(1-) f_i(x_{i,j}-c)\Bigr)dc,
\end{equation}
using the relation (\ref{eq_cumulant_general_formula}) as in Jolivet~\cite{jolivet1981central}
with assuming the existence of $g_i^{(n_i)}(1-)$, which is the $n$-th factorial moment of $M_i$. 
(The existence of the factorial moment measure itself will be dealt with in Lemma \ref{lem_ns_existence_of_moment}.)
We can also calculate densities of the factorial moment measures by using the relation (\ref{fmomt2fcum}) as follows.

\begin{description}
  \item[First-order moment]
        By stationarity, the density of the first-order (factorial) moment measure of $N_i^S$ (called intensity hereafter) is given by constant $\lambda \sigma_i$, where $\sigma_i = g_i^{(1)}(1-)$ is the first-order moment of the number of the offsprings for each parent point, $i=1, 2$.
  \item[Second-order moment]
        The density of the $(1, 1)$-th order factorial moment measure $M_{[1, 1]}^S$ of $(N_1^S, N_2^S)$ (called cross-intensity hereafter) is given by
        \[
          \lambda_{[1, 1]}^S(x, y)
          =\lambda^2 \sigma_1\sigma_2
          + \lambda \sigma_1\sigma_2
          \int_{\bbR}f_1(x-c)f_2(y-c)dc.
        \]
        Thus, the cross-intensity function is
        \[
          \lambda_{[1, 1]}^S(u)
          =\lambda^2 \sigma_1\sigma_2
          + \lambda \sigma_1\sigma_2
          \int_{\bbR}f_1(s)f_2(u+s)ds.
        \]
\end{description}

\subsection{Noisy bivariate Neyman-Scott process}
We finalize our model by adding independent noises whose structures are unknown.
Let $N^S=(N_1^S, N_2^S)$ be a bivariate Neyman-Scott process introduced in the previous subsection.
In addition, suppose $N^B = (N_1^B, N_2^B)$ is a bivariate stationary point process such that $N_1^B, N_2^B$, and $N^S$ are independent.
The intensity of $N_i$ is denoted by $\lambda_i^B\geq 0$, $i=1, 2$.
The noisy bivariate Neyman-Scott process is constructed by superposing $N^B$ to $N^S$ as noise, i.e.
$N = (N_1, N_2) = (N_1^S+N_1^B, N_2^S+N_2^B)$.
The intensity of $N$ is
\[
  \lambda_i = E[N_i([0, 1])] = \lambda\sigma_i + \lambda_i^B, \quad i=1, 2.
\]
The cross-intensity of the stationary process $N=(N_1, N_2)$ is
\[
  \lambda_{[1, 1]}(u) = (\lambda\sigma_1 + \lambda_1^B)(\lambda\sigma_2 + \lambda_2^B)
  + \lambda\sigma_1\sigma_2\int_{\bbR}f_1(s)f_2(s+u)ds
\]
because, for a bounded Borel function $f$, we have
\begin{align*}
   & \quad E\Bigl[\sum_{x\in N_1, y\in N_2} f(x, y)\Bigr] \\
   & = E\Bigl[
    \sum_{x\in N_1^S, y\in N_2^S} f(x, y)
    + \sum_{x\in N_1^S, y\in N_2^B} f(x, y)
    + \sum_{x\in N_1^B, y\in N_2^S} f(x, y)
    + \sum_{x\in N_1^B, y\in N_2^B} f(x, y)
  \Bigr]                                                  \\
   & = \int_{\bbR^2}f(x, y)\lambda_{[1, 1]}^S(x, y)dxdy
  + (\lambda\sigma_1\lambda_2^B + \lambda\sigma_2\lambda_1^B + \lambda_1^B\lambda_2^B)
  \int_{\bbR^2}f(x, y)dxdy                                \\
   & = \int_{\bbR^2}\Bigl\{
  (\lambda\sigma_1 + \lambda_1^B)(\lambda\sigma_2 + \lambda_2^B)
  + \lambda\sigma_1\sigma_2\int_{\bbR}f_1(x-c)f_2(y-c)dc
  \Bigr\}
  f(x, y)dxdy.
\end{align*}

Hereafter, we parametrize the dispersal kernels by the parameters $(\tau_1, \tau_2)$ as $f_i(\cdot; \tau_i), i=1, 2$.
Therefore, the cross-correlation function of $N$ is
\begin{equation}\label{eq_ns_cross_corr}
  g(u; \theta)
  = 1 + a \int_{\bbR} f_1(s; \tau_1)f_2(u+s; \tau_2)ds,
\end{equation}
where $a = \frac{\lambda\sigma_1\sigma_2}{(\lambda\sigma_1 + \lambda_1^B)(\lambda\sigma_2 + \lambda_2^B)}$.
We parametrize the entire model with $\theta = (a, \tau_1, \tau_2) \in \ol{\Theta}$, where $\Theta =  \cala\times \calt_1\times \calt_2$, $\cala\subset(0, \infty)$, $\calt_1\subset \bbR^{\sfp_1}$, $\calt_2\subset \bbR^{\sfp_2}$, and $\sfp_1, \sfp_2\in\bbZ_{\geq 1}$.
In the following, we call this model the noisy bivariate Neyman-Scott process (NBNSP).

\subsection{Statistical inference for noisy bivariate Neyman-Scott processes}
Let $N = (N_1, N_2) = (N_1^S+N_1^B, N_2^S+N_2^B)$ be the noisy bivariate Neyman-Scott process model introduced in the previous subsection. We discuss the asymptotic properties of the model based on the general theory in Section \ref{sec_general_theory}.
We will impose some conditions below.
\begin{description}
  \item{\bf{[NS]}}
        \begin{description}
          \item{(i)}
                $\cala$, $\calt_1$, and $\calt_2$ are bounded, open, and convex.
          \item{(ii)}
                For $i=1, 2$, the dispersal kernel $f_i$ has a form
                \begin{equation}\label{eq_ns_cond_1}
                  f_i(u; \tau_i) =
                  h_{i, 1}(u; \tau_i)u^{h_{i, 2}(\tau_i) - 1}1_{(0, 1)}(u)
                  + h_{i, 3}(u; \tau_i)1_{[1, \infty)}(u), \ i=1, 2,
                \end{equation}
                where, for $i=1, 2$,
                $h_{i, 1}$ is bounded measurable function on $(0, 1)\times \ol{\calt_i}$,
                $h_{i, 1}(u; \cdot)\in C(\ol{\calt_i}) \cap C^3(\calt_i)$ for all $u\in(0, 1)$,
                \(
                \inf_{|u|<1, \tau_i\in\ol{\calt_i}}h_{i, 1}(u; \tau_i) > 0,
                \)
                $\partial_{\tau_i}^k h_{i, 1}$ is bounded on $(0, 1) \times \calt_i$ for $k=1, 2, 3$,
                $h_{i, 2}\in C^3(\ol{\calt_i})$,
                \(\inf_{\tau_i\in\ol{\calt_i}}h_{i, 2}(\tau_i) > 0\),
                $h_{i, 3}$ is bounded measurable function on $[1, \infty)\times \ol{\calt_i}$,
                $h_{i, 3}(u; \cdot)\in C(\ol{\calt_i}) \cap C^3(\calt_i)$ for all $u\in [1, \infty)$,
                and there exists bounded $\tilde{f} \in L^1([1, \infty))$ (common to $i=1, 2$) such that
                \begin{equation}\label{eq_ns_assumption1}
                  \sup_{\tau_i}|\ptaui^k h_{i, 3}(\cdot; \tau_i)| \leq \tilde{f}
                \end{equation}
                for $k=0, 1, 2, 3$, where the supremum is taken on $\ol{\calt_i}$ for $k=0$ and on $\calt_i$ for $k=1, 2, 3$.
          \item{(iii)}
                There exist some $\delta>0$ and $\epsilon>0$ such that
                \begin{gather}
                  \|M_i\|_{\lceil 2+\delta \rceil} < \infty, \quad i=1, 2 , \label{gat_ns_cond_1}
                \end{gather}
                \begin{equation}\label{eq_ns_cond_2}
                  \int_{|u|>R}f_i(u; \tau_i)du
                  = O\!\left(R^{-\frac{2+\delta}{\delta(2+\epsilon)}}\right) \text{ as $R\to\infty$}, \quad \tau_i\in \calt_i, i=1, 2.
                \end{equation}
          \item{(iv)}
                For $\delta>0$ appearing in [NS](iii),
                \begin{equation}
                  \|N_i^B([0, 1])\|_{\lceil 2+\delta \rceil} < \infty, \quad i=1, 2,\label{gat_ns_cond_2}
                \end{equation}
                and
                \begin{equation}\label{eq_ns_cond_3}
                  \sum_{m=1}^{\infty}\alpha_{2(1+r), \infty}^{N_i^B}(m)^{\frac{\delta}{2+\delta}} < \infty, \quad i=1, 2.
                \end{equation}
                Moreover, $N_i^B$ has locally bounded factorial cumulant densities up to
                $\lceil2+\delta\rceil$-th order.
        \end{description}
\end{description}

\begin{remark}\label{rem_mom_NS}\rm
  Under [NS](ii), $f_i(\cdot; \tau_i)$ is in $L^q(\bbR)$ for some $q>1$.
\end{remark}
\halflineskip

We now state our main result for this section.
\begin{theorem}\label{thm_noisy_NS_asymptotics}
  Suppose that the bivariate noisy Neyman-Scott process model satisfies the condition [NS].
  Then, [PA], [MI], and [RE] hold. If we further assume [WI2], [ID], and [ID2], then the assumptions in Theorem \ref{thm_consistency} and Theorem \ref{thm_asymptotic_normality} hold so that the QMLE defined in (\ref{eq_def_QMLE}) has consistency and asymptotic normality.
\end{theorem}
\halflineskip

The proof of Theorem \ref{thm_noisy_NS_asymptotics} will be given in Section \ref{subsec_proof_thm_noisy_symptotics}.
Here we discuss the conditions.
The condition [ID] will be discussed for specific models in Section \ref{sec_specific_models}.
The condition [NS](ii) looks a bit complicated, but, roughly speaking, it requires that $f_i$ is $O(u^{\alpha-1}), u\to0$ for some $\alpha>0$ and the parameter derivatives of $f_i$ have suitable integrability.
A lot of popular parametric classes of positive distributions can be written in the form of (\ref{eq_ns_cond_1}).
For example, the exponential kernel
\(
f_i(u; l_i) = l_i e^{-l_iu}1_{(0, \infty)}(u), l_i > 0
\)
and gamma kernel
\(
f_i(u; \alpha_i, l_i) = \frac{l_i^{\alpha_i}}{\Gamma(\alpha_i)}u^{\alpha_i-1}e^{-l_i u} 1_{(0, \infty)}(u), l_i, \alpha_i > 0
\)
satisfy these conditions. Of course, if one does not need the divergence at the origin, just take $h_{i, 2} \equiv 1$. It is also possible that $f_i$ has some singular points besides the origin, but we do not pursue it here.
The condition (\ref{gat_ns_cond_1}) is the existence of higher-order moments of the number of offspring of the Neyman-Scott process.
We note that the distribution of $M_i$ is a nuisance parameter and only the mean $\sigma_i$ of $M_i$ is related to the parameter $a = \frac{\lambda\sigma_1\sigma_2}{(\lambda\sigma_1 + \lambda_1^B)(\lambda\sigma_2 + \lambda_2^B)}$ in our setting. If one wants to estimate $\sigma_i$ itself, one could assume some parametric structures for the noise processes and estimate all the parameters using additional methods.
For example, one may apply the adaptive estimation procedure based on the nearest-neighbor distance property, as in Tanaka \& Ogata~\cite{tanaka2014identification}, if the noise processes are also univariate Neyman-Scott processes.
The condition (\ref{eq_ns_cond_2}) requires the fast decay of the tail of the dispersal kernel, which ensures the fast decay of the $\alpha$-mixing rate of the Neyman-Scott process. In particular, the kernels that have exponential decay, such as the exponential and gamma kernels, satisfy (\ref{eq_ns_cond_2}).
Regarding (\ref{gat_ns_cond_2}) and (\ref{eq_ns_cond_3}), it is just the existence of higher-order locally finite moment measures and the fast decay of the $\alpha$-mixing rate of the noise processes that hold for many point processes.
In practice, the choice of the (stationary) Poisson noise may be reasonable to some extent because the law of independent superposition of many point processes converges to the law of the Poisson process.
For another example, the stationary Hawkes process with an exponential kernel, a basic model for high-frequency financial data (see, for example, \cite{bacry2015hawkes}), satisfies the assumption for the noise processes.
Indeed, we have the mixing condition (\ref{eq_ns_cond_3}) by Theorem 1 in Cheysson \& Lang \cite{cheysson2022spectral}.
Moreover, one can prove that the stationary Hawkes process with a bounded kernel has bounded factorial cumulant densities for all orders, thanks to the expression in Jovanovi{\'c} et al. \cite{jovanovic2015cumulants}. For details, see Appendix \ref{sec:appendix_a}. It also guarantees the existence of the moments themselves. Therefore, we have all of the conditions in [NS](iv) for the Hawkes process.


\section{Specific models}\label{sec_specific_models}
In this section, we will consider NBNSP with gamma kernels (NBNSP-G) and exponential kernels (NBNSP-E) and give the asymptotic theory for these models based on the results in Section \ref{sec_NS_model}.


\subsection{Noisy bivariate Neyman-Scott process with gamma kernels (NBNSP-G)}\label{subsec_nbnsp-g}
For the sake of modeling lead-lag relationships in high-frequency financial data,
we suggest gamma kernels
because the divergence at the origin represents rapid responses of algorithm trades to the common ``trigger'' (modeled by the parent process $\calc$) between two assets.

Let
\[
  f_i(u; \alpha_i, l_i) = \frac{l_i^{\alpha_i}}{\Gamma(\alpha_i)}u^{\alpha_i-1}e^{-l_i u} 1_{(0, \infty)}(u),
  \quad l_i, \alpha_i > 0, i=1, 2
\]
in NBNSP.
The parameters to be estimated are $(a, \alpha_1, \alpha_2, l_1, l_2) \in \Theta = \prod_{k=1}^5(p_{k,1}, p_{k, 2})$ where
$0<p_{k,1}<p_{k, 2}<\infty$ for $k=1, \ldots, 5$.
As shown in (\ref{eq_ns_cross_corr}), The cross-correlation function is
\begin{equation}\label{eq_gamma_ns_ccf}
  g(u; \theta) = 1 + a p(u; \alpha_1, \alpha_2, l_1, l_2)
\end{equation}
where
\[
  p(u; \alpha_1, \alpha_2, l_1, l_2) = \int_{\bbR} f_1(s; \alpha_1, l_1)f_2(u+s; \alpha_2, l_2)ds.
\]
For clarity, we summarize the estimation procedure.
Suppose the data is observed as counting measures $N = (N_1, N_2)$ on $[0, T]$.
The quasi-log likelihood function is
\begin{equation}\label{eq_summarize_estimation_gamma_ns}
  \bbH(\theta)
  = \sum_{\substack{x\in N_1, y\in N_2 \\ |y-x|\leq r \\ x\in [r, T-r]}}
  \Bigl(
  \log(g(y-x; \theta)) + \log\wih{\lambda}_1 + \log\wih{\lambda}_2
  \Bigr)
  - (T-2r)
  \wih{\lambda}_1\wih{\lambda}_2\int_{|u|\leq r}g(u; \theta)du,
\end{equation}
where
$\wih{\lambda}_i = T^{-1}N_i([0, T]), i=1, 2$.
Then we derive the estimator (QMLE) $\wih{\theta}_T$ by maximizing $\bbH(\theta)$.

In fact, $p(u; \alpha_1, \alpha_2, l_1, l_2)$ in (\ref{eq_gamma_ns_ccf}) is the probability density function of the bilateral gamma distribution introduced by K{\"u}chler \& Tappe~\cite{kuchler2008bilateral}.
By the definition, we have
\begin{equation}\label{eq_bigamma_integral_expression}
  p(u; \alpha_1, \alpha_2, l_1, l_2)
  =\frac{l_1^{\alpha_1} l_2^{\alpha_2}}
  {(l_1+l_2)^{\alpha_1}\Gamma(\alpha_1)\Gamma(\alpha_2)}
  e^{-l_1u}
  \int_0^{\infty} v^{\alpha_2-1}\Bigl(
  u + \frac{v}{l_1+l_2}
  \Bigr)^{\alpha_1-1}
  e^{-v}dv, \quad u>0.
\end{equation}
According to K{\"u}chler \& Tappe~\cite{kuchler2008shapes},
it can also be expressed using the confluent hypergeometric function $\Phi(\gamma, \delta; z)$:
\begin{align*}
  p(u; \alpha_1, \alpha_2, l_1, l_2)
   & = \frac{l_1^{\alpha_1} l_2^{\alpha_2} \Gamma(1-(\alpha_1+\alpha_2))}{\Gamma(\alpha_1)\Gamma(1-\alpha_1)} \\
   & \quad \times u^{(\alpha_1+\alpha_2 - 1)}
  e^{-l_1u}
  \Phi\bigl(\alpha_2, \alpha_1+\alpha_2, (l_1+l_2)u\bigr)                                                     \\
   & \quad +
  \frac{l_1^{\alpha_1} l_2^{\alpha_2} \Gamma(\alpha_1+\alpha_2-1)}{\Gamma(\alpha_1)\Gamma(\alpha_2)}          \\
   & \quad \times
  (l_1+l_2)^{1-(\alpha_1+\alpha_2)}
  e^{-l_1u}\Phi\bigl(1-\alpha_1, 2-(\alpha_1+\alpha_2), (l_1+l_2)u\bigr)
\end{align*}
for $u>0$
where
\[
  \Phi(\gamma, \delta; z)
  = 1 + \frac{\gamma}{\delta} \frac{z}{1!}
  + \frac{\gamma(\gamma + 1)}{\delta(\delta + 1)} \frac{z^2}{2!}
  + \frac{\gamma(\gamma + 1)(\gamma + 2)}{\delta(\delta + 1)(\delta + 2)} \frac{z^3}{3!}
  + \cdots.
\]
For $u<0$, we can obtain similar formulae using the relation
\begin{equation}\label{eq_bigamma_symmetry}
  p(u; \alpha_1, \alpha_2, l_1, l_2)
  = p(-u; \alpha_2, \alpha_1, l_2, l_1).
\end{equation}
Using these analytical representations, we can avoid calculating numerical integration so many times in the first term of the RHS of (\ref{eq_summarize_estimation_gamma_ns}).

We can derive the asymptotic property of the QMLE for NBNSP-G using Theorem \ref{thm_noisy_NS_asymptotics}.
\begin{theorem}\label{thm_NBNSP_G}
  Assume that the conditions [NS](i), (iii)-(\ref{gat_ns_cond_1}), (iv), and [WI2] hold.
  Let $\theta^*$ be the true value of the parameter.
  Then the QMLE for the NBNSP-G has the consistency and the asymptotic normality,  i.e. we have
  \[
    \wih{\theta}_{T_n} \to^p \theta^*
  \]
  and
  \[
    \sqrt{T_n}(\wih{\theta}_{T_n} - \theta^*) \to^d N(0, \Gamma^{-1}\Sigma\Gamma^{-1})
  \]
  for some nonnegative matrix $\Sigma$.
\end{theorem}
\halflineskip
The proof will be given in Section \ref{subsec_prf_NBNSP}.

\subsection{Noisy bivariate Neyman-Scott process with exponential kernels (NBNSP-E)}
Let
\[
  f_i(u; l_i) = l_i e^{-l_i u} 1_{(0, \infty)}(u),
  \quad l_i > 0, i=1, 2.
\]
Then the cross-correlation function is
\[
  g(u; \theta) = 1 + a q(u; l_1, l_2)
\]
where
\[
  q(u; l_1, l_2)
  = \frac{l_1l_2}{l_1+l_2} \Bigl(
  1_{(0, \infty)} (u) e^{-l_2u} + 1_{(-\infty, 0)} (u) e^{-l_1u}
  \Bigr).
\]
The parameters to be estimated are $\theta = (a, l_1, l_2)$.

\begin{theorem}\label{thm_NBNSP_E}
  Assume that the conditions [NS](i), (iii)-(\ref{gat_ns_cond_1}), (iv), and [WI2] hold.
  Then the QMLE for the exponential kernel model has consistency and asymptotic normality.
\end{theorem}

The proof will also be given in Section \ref{subsec_prf_NBNSP}.
\section{Simulation studies}\label{sec_simulation}
The performance of the QMLE for the NBNSP-G is investigated by simulations in various settings.
The number of replications in each Monte Carlo simulation is 500.
The optimization is conducted by the Nelder-Mead algorithm of the Python package Scipy.
We report the mean and the standard deviation (std) of each component of the QMLE.
Table \ref{table_expr1} shows the consistency of the QMLE in the absence of any noise.
Table \ref{table_expr2} shows the consistency of the QMLE when the data is contaminated by homogenous Poisson noise with the same intensities as the signal process $N_i^S$.
We observe that the parameter standard deviations are larger than in the noiseless version.
Table \ref{table_expr3} shows that as the amount of noise increases, both the bias and standard deviation of the QMLE also increase.
Table \ref{table_expr4} shows that a bigger $r$ decreases the standard deviation of the estimator, but it demands more computational time.
Specifically, a single evaluation of the quasi‑likelihood costs the order of $T\lambda_1\lambda_2 \int_{|u|\leq r}g(u)du$.
We have $g\geq 1$ for the NBNSP models, so that the cost grows at least linearly with
$r$, i.e., $T\lambda_1\lambda_2 \int_{|u|\leq r}g(u)du\geq 2T\lambda_1\lambda_2r$.

\begin{table}[h]
  \centering
  \begin{tabular}{llrrrrr}
    \toprule
                              &      & $a$   & $\alpha_1$ & $\alpha_2$ & $l_1$ & $l_2$ \\
    $T$                       &      &       &            &            &       &       \\
    \midrule
    \multirow[t]{2}{*}{2500}  & mean & 10.3  & 0.305      & 0.402      & 1.06  & 1.03  \\
                              & std  & 1.28  & 0.0329     & 0.0326     & 0.389 & 0.297 \\
    \cline{1-7}
    \multirow[t]{2}{*}{5000}  & mean & 10.2  & 0.301      & 0.402      & 1.01  & 1.02  \\
                              & std  & 0.852 & 0.0219     & 0.0227     & 0.267 & 0.211 \\
    \cline{1-7}
    \multirow[t]{2}{*}{10000} & mean & 10.1  & 0.301      & 0.4        & 1.01  & 1.01  \\
                              & std  & 0.564 & 0.0169     & 0.0164     & 0.191 & 0.15  \\
    \cline{1-7}
    true                      &      & 10    & 0.3        & 0.4        & 1     & 1     \\
    \cline{1-7}
    \bottomrule
  \end{tabular}
  \caption{Means and standard deviations of the estimator for the parameters of NBNSP-G. Settings: $r=1.0, M_1\sim \mathrm{Poi}(2), M_2\sim \mathrm{Poi}(4), \lambda=0.1$, without noise.}
  \label{table_expr1}
\end{table}

\begin{table}[h]
  \centering
  \begin{tabular}{llrrrrr}
    \toprule
                              &      & $a$   & $\alpha_1$ & $\alpha_2$ & $l_1$ & $l_2$ \\
    T                         &      &       &            &            &       &       \\
    \midrule
    \multirow[t]{2}{*}{2500}  & mean & 2.58  & 0.305      & 0.403      & 1.07  & 1.03  \\
                              & std  & 0.439 & 0.035      & 0.0357     & 0.462 & 0.369 \\
    \cline{1-7}
    \multirow[t]{2}{*}{5000}  & mean & 2.54  & 0.301      & 0.402      & 1.01  & 1.03  \\
                              & std  & 0.227 & 0.023      & 0.0244     & 0.308 & 0.249 \\
    \cline{1-7}
    \multirow[t]{2}{*}{10000} & mean & 2.51  & 0.301      & 0.401      & 1.01  & 1.02  \\
                              & std  & 0.158 & 0.0179     & 0.018      & 0.226 & 0.177 \\
    \cline{1-7}
    true                      &      & 2.5   & 0.3        & 0.4        & 1     & 1     \\
    \cline{1-7}
    \bottomrule
  \end{tabular}
  \caption{Means and standard deviations of the estimator for the parameters of NBNSP-G. Settings: $r=1.0, M_1\sim \mathrm{Poi}(2), M_2\sim \mathrm{Poi}(4), \lambda=0.1$, $N_i^B$ is a homogenous Poisson process, $E[N_i^B([0, 1])] = E[N_i^S([0, 1])] \times 1.0, i=1, 2$.}
  \label{table_expr2}
\end{table}

\begin{table}[h]
  \centering
  \begin{tabular}{llrrrrr}
    \toprule
                           &      & $a$    & $\alpha_1$ & $\alpha_2$ & $l_1$ & $l_2$ \\
    SN Coef                &      &        &            &            &       &       \\
    \midrule
    \multirow[t]{3}{*}{0}  & mean & 10.2   & 0.301      & 0.402      & 1.01  & 1.02  \\
                           & std  & 0.852  & 0.0219     & 0.0227     & 0.267 & 0.211 \\
                           & true & 10     & 0.3        & 0.4        & 1     & 1     \\
    \cline{1-7}
    \multirow[t]{3}{*}{5}  & mean & 0.295  & 0.304      & 0.402      & 1.08  & 1.03  \\
                           & std  & 0.0599 & 0.0325     & 0.035      & 0.541 & 0.407 \\
                           & true & 0.278  & 0.3        & 0.4        & 1     & 1     \\
    \cline{1-7}
    \multirow[t]{3}{*}{10} & mean & 0.109  & 0.305      & 0.407      & 1.15  & 1.12  \\
                           & std  & 0.0721 & 0.0455     & 0.0482     & 0.932 & 0.654 \\
                           & true & 0.0826 & 0.3        & 0.4        & 1     & 1     \\
    \cline{1-7}
    \bottomrule
  \end{tabular}
  \caption{Means and standard deviations of the estimator for the parameters of NBNSP-G. Settings: $T=5000$, $r=1.0, M_1\sim \mathrm{Poi}(2), M_2\sim \mathrm{Poi}(4), \lambda=0.1$, $N_i^B$ is a homogenous Poisson process, $E[N_i^B([0, 1])] = E[N_i^S([0, 1])] \times (\mathrm{SN \ Coef}), i=1, 2$.}
  \label{table_expr3}
\end{table}

\begin{table}[h]
  \centering
  \begin{tabular}{llrrrrr}
    \toprule
                            &      & $a$   & $\alpha_1$ & $\alpha_2$ & $l_1$ & $l_2$ \\
    $r$                     &      &       &            &            &       &       \\
    \midrule
    \multirow[t]{2}{*}{0.5} & mean & 2.69  & 0.301      & 0.401      & 1.02  & 1.02  \\
                            & std  & 0.6   & 0.027      & 0.0269     & 0.464 & 0.383 \\
    \cline{1-7}
    \multirow[t]{2}{*}{1}   & mean & 2.54  & 0.301      & 0.402      & 1.01  & 1.03  \\
                            & std  & 0.227 & 0.023      & 0.0244     & 0.308 & 0.249 \\
    \cline{1-7}
    \multirow[t]{2}{*}{2}   & mean & 2.52  & 0.303      & 0.402      & 1.03  & 1.03  \\
                            & std  & 0.207 & 0.0227     & 0.0232     & 0.272 & 0.212 \\
    \cline{1-7}
    true                    &      & 2.5   & 0.3        & 0.4        & 1     & 1     \\
    \cline{1-7}
    \bottomrule
  \end{tabular}
  \caption{Means and standard deviations of the estimator for the parameters of NBNSP-G. Settings: $T=5000$, $M_1\sim \mathrm{Poi}(2), M_2\sim \mathrm{Poi}(4), \lambda=0.1$, $N_i^B$ is a homogenous Poisson process, $E[N_i^B([0, 1])] = E[N_i^S([0, 1])] \times 1.0, i=1, 2$.}
  \label{table_expr4}
\end{table}

\newpage
\section{Application to real-world data}\label{sec_realdata}
In this section, using real high-frequency financial transaction data, we compare NBNSP-G, NBNSP-E, and a bivariate Hawkes process model with exponential kernels (BHP-E).


We get individual stock tick data from Nikkei NEEDS (Nikkei Economic Electronic Databank System) traded on the Tokyo Stock Exchange in August 2019.
Three pairs of individual stocks (see Table \ref{table_individual_stocks} below) from the same industry are used.
We extract timestamps of executed sell and buy orders in the afternoon session and remove the first and last 15 minutes to avoid the influence of the auctions.
Then, for each stock, all of the timestamp data are combined sequentially with 2-second intervals between each date.
We also report the number of sell and buy executed orders in the processed dataset in Table \ref{table_individual_stocks}.

\begin{table}
  \centering
  \begin{tabular}{lllrr}
    \toprule
    code & type                     & company                              & \#buy & \#sell \\
    \midrule
    7201 & Transportation Equipment & NISSAN MOTOR CO., LTD.               & 37661 & 35593  \\
    7203 & Transportation Equipment & TOYOTA MOTOR CORPORATION             & 27297 & 26313  \\
    8306 & Banks                    & Mitsubishi UFJ Financial Group, Inc. & 34106 & 35112  \\
    8411 & Banks                    & Mizuho Financial Group, Inc.         & 14405 & 16108  \\
    8031 & Wholesale Trade          & MITSUI \& CO.,LTD.                   & 12349 & 13937  \\
    8058 & Wholesale Trade          & Mitsubishi Corporation               & 20506 & 22906  \\
    \bottomrule
  \end{tabular}
  \caption{The three pairs of stocks used in the experiments and the number of executed orders in the processed data.}
  \label{table_individual_stocks}
\end{table}

We compare the cross-correlation of NBNSP-G, NBNSP-E, and BHP-E.
The NBNSPs are introduced in Section \ref{sec_specific_models} and their parameter are estimated by the QMLE.
The hyperparameter $r$ for the QMLE is set to $1.0$ for all of the experiments.
The conditional intensity functions of the BHP-E are
\[
  \lambda_1^H(t) = \mu_1 + \int_{-\infty}^t \alpha_{11} e^{-\beta_{1}(t-s)} d N_1(s) + \int_{-\infty}^t \alpha_{12} e^{-\beta_{1}(t-s)} d N_2(s),
\]

\[
  \lambda_2^H(t) = \mu_2 + \int_{-\infty}^t \alpha_{21} e^{-\beta_{2}(t-s)} d N_1(s) + \int_{-\infty}^t \alpha_{22} e^{-\beta_{2}(t-s)} d N_2(s),
\]
which are similar to Da Fonseca \& Zaatour~\cite{da2017correlation}.
Therefore, the parameters of BHP-E are
\[
  (\mu_1, \mu_2, \alpha_{11}, \alpha_{12}, \alpha_{21}, \alpha_{22}, \beta_1, \beta_2),
\]
and estimated by MLE using an R package ``emhawkes''.

The empirical kernel estimator of the cross-correlation function is
\begin{equation}\label{eq_kernel_method_ccf}
  \wih{g}(u)
  = \frac{1}{\wih{\lambda}_1\wih{\lambda}_2}\sum_{\substack{x\in N_1, y\in N_2    \\ x\in [r, T-r]}}
  \frac{k_h\bigl((y-x) - u\bigr)}{T - |y-x|}, \quad u\in\bbR,
\end{equation}
where $k_h(z) = (2h)^{-1} 1_{[-h, h]}(z)$ is the uniform kernel with bandwidth $h$.
We set $h = 0.001$ for all of the experiments.

The estimated parameters are reported in Tables \ref{table_est_params_real_start}-\ref{table_est_params_real_end} in Appendix.
Figures \ref{fig_buy} and \ref{fig_sell} show the empirically estimated cross-correlation function using the kernel method for each pair of stocks, as well as the theoretical cross-correlation function for each model based on the estimated parameters.
The theoretical curves for the NBNSPs are calculated from the theoretical formulae in Section \ref{sec_specific_models}.
For BHP-E, 100,000 points are sampled from the estimated model, and then the theoretical curve is estimated by the kernel estimator.

NBNSP-E does not seem to be able to explain the strong correlation near the origin.
BHP-E can explain the correlation near the origin to some extent, but it tends to have poor fits in little away from the origin.
NBNSP-G appears to explain both the correlation near the origin and the tail decay well.

\begin{figure}[htbp]
  \begin{minipage}[b]{0.49\linewidth}
    \centering
    \includegraphics[height=8cm, width=8cm]{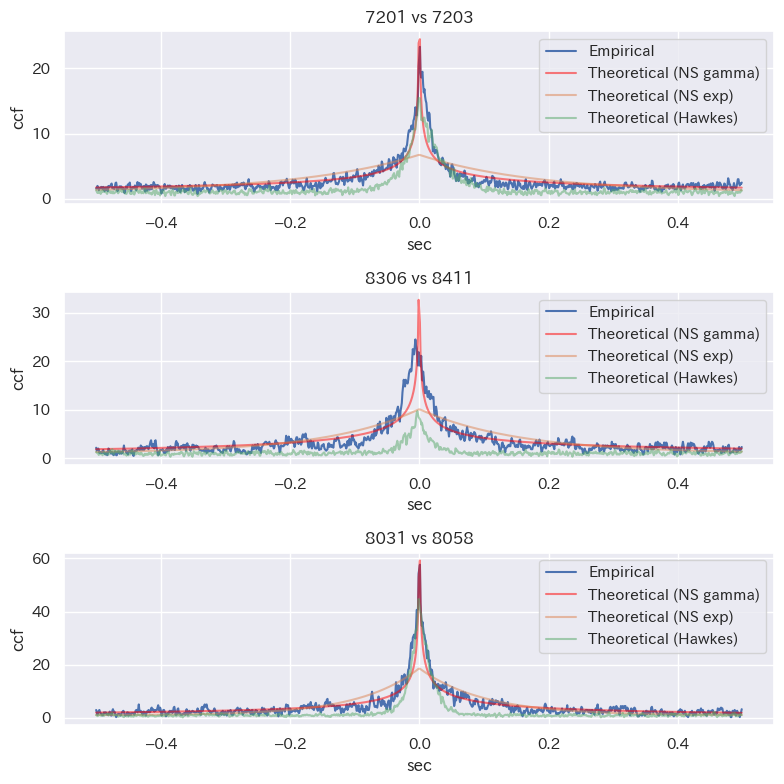}
    \caption{The empirical CCFs estimated by the kernel estimator (\ref{eq_kernel_method_ccf}) and theoretical CCFs for each model using estimated parameters for the buy orders data.}
    \label{fig_buy}
  \end{minipage}
  \begin{minipage}[b]{0.49\linewidth}
    \centering
    \includegraphics[height=8cm, width=8cm]{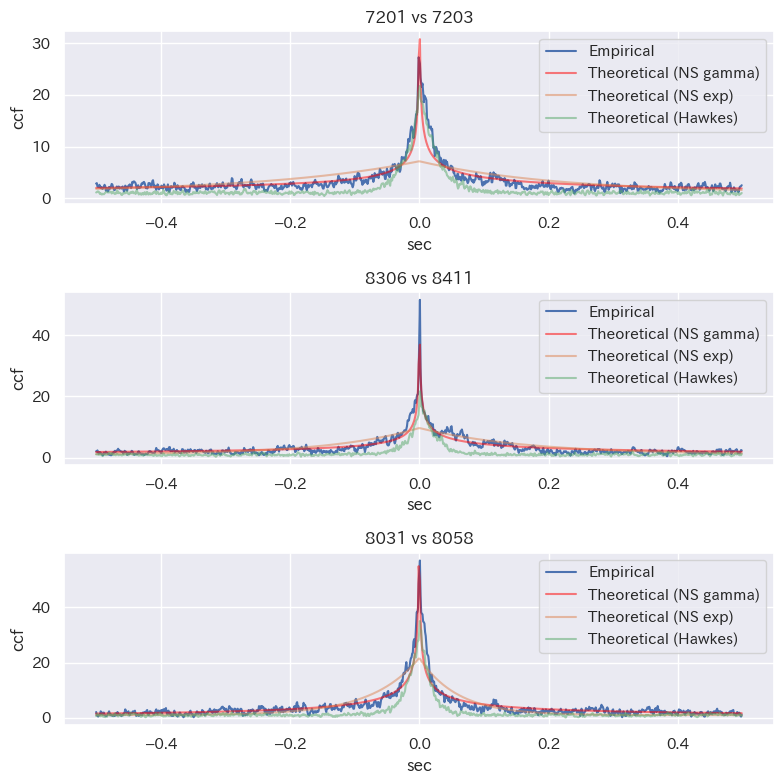}
    \caption{The empirical CCFs estimated by the kernel estimator (\ref{eq_kernel_method_ccf}) and theoretical CCFs for each model using estimated parameters for the sell orders data.}
    \label{fig_sell}
  \end{minipage}
\end{figure}

We discuss how our proposed model NBNSP will be used for lead-lag estimation.
Our model NBNSP of (\ref{eq_rep_our_model}) consists of the noise part $N^B$ and the signal part $N^S$.
In the signal part $N^S$ of (\ref{eq_our_model_signal}), the dispersal kernels $f_1$ and $f_2$ (=the laws of $d_1$ and $d_2$) show how quickly the two stocks react to the arrival of common new information. Therefore, by comparing how much the estimated kernels $f_1(\cdot, \wih{\tau_1})$ and $f_2(\cdot, \wih{\tau_2})$ are concentrated at the origin, we can estimate the lead-lag effect between two stocks.
For instance, one could simply compare the means $m_1$ and $m_2$ of the kernels and say, ``Stock 1 leads Stock 2 by $m_2-m_1$ seconds on average.''
More specifically, in the case of the gamma kernel model NBNSP-G, we can also compare the shape parameters $\alpha_1$ and $\alpha_2$, which indicate the divergence speeds of the kernels at the origin.
Moreover, it is possible to test whether these differences are significant by using asymptotic normality for hypothesis testing if one knows the asymptotic variance.
The development of the estimation theory of the asymptotic variance of QMLE for NBNSP and the construction of the test for the existence of the lead-lag relationship will be the subjects of forthcoming works.
For instance, the subsampling method (e.g.~\cite{biscio2019general}) could be used to estimate the asymptotic variance.
\section{Extensions and alternatives}\label{sec_extension}
Modern empirical applications often demand more flexibility than a bivariate, stationary specification can offer. In practice, one may wish to 
avoid imposing stationarity on financial data with time‑varying trading activity and quantify lead-lag relationships among many assets.
To address these needs, this section develops inhomogeneous and multivariate extensions of our model.
We also present an alternative estimation procedure based on a spectral (frequency-domain) method.

\subsection{Inhomogeneous extension}
We outline below an extension of our model and estimation method to inhomogeneous settings.
Let \(N=(N_1,N_2)\) be a bivariate point process on \(\mathbb{R}\), which is not necessarily stationary.
Recall the intensity function $\lambda_i(\cdot)$ of $N$ is defined as
\[
  E\Bigl[\sum_{x\in N_i}h(x)\Bigr] = \int_{\mathbb{R}} h(u)\lambda_i(u)du, \quad i=1, 2
\]
where $h: \mathbb{R}\to\mathbb{R}$ is non-negative measurable function, and 
the cross intensity function $\lambda_{1, 2}(u, v)$ of $N$ is defined as
\[
  E\Bigl[\sum_{\substack{x\in N_1, y\in N_2}} h(x, y)\Bigr]
  = \int_{\mathbb{R}} h(u, v) \lambda_{1, 2}(u, v)dudv,
\]
where $h: \mathbb{R}^2\to\mathbb{R}$ is a non-negative measurable function.
The cross-correlation function $g(u, v)$ of $N$ is also defined as
\begin{equation}\label{eq_ccf_inhom_def}
    g(u, v) = \frac{\lambda_{1, 2}(u, v)}{\lambda_1(u)\lambda_2(v)}, \quad u, v\in \mathbb{R}.
\end{equation}

\subsubsection{General Formulation}
The key idea is to work under a transition-invariant cross-correlation function (see, e.g., Shaw et al. \cite{shaw2021globally}), which is a multivariate extension of second-order intensity-reweighted stationarity (SOIRS)~\cite{baddeley2000non}.
Suppose we have a transition-invariant cross-correlation function $g$ of $N$, that is, $g(u, v)$ in (\ref{eq_ccf_inhom_def}) depends on only $v-u$.
Then, we have
\[
  E\Bigl[\sum_{\substack{x\in N_1, y\in N_2}} h(x, y)\Bigr]
  = \int_{\mathbb{R}} h(u, v) \lambda_1(u)\lambda_2(v)g(v-u)dudv
\]
for all non-negative measurable function $h: \mathbb{R}^2\to\mathbb{R}$,
with slightly abused notation on $g$.

If the first order intensities parametrized by $\vartheta_i$ as $\lambda_i(\cdot; \vartheta_i), i=1, 2$ and the transition-invariant cross-correlation function parametrized by $\theta$ as $g(\cdot; \theta)$, then one may consider following two-step composite-type QMLE
\begin{gather*}
    \widehat{\vartheta}_i \in \mathop{\arg\max}_{\vartheta_i} 
    \sum_{x\in N_i, x\in [0, T]}\log \lambda_i(x; \vartheta_i) - \int_0^T\lambda_i(u; \vartheta_i)du, \quad i=1, 2,  \\
    \widehat{\theta} \in \mathop{\arg\max}_{\theta}
    \sum_{\substack{x\in N_1, y\in N_2, \\ x, y\in[0, T] , \\ |y-x|\leq r}}
    \Biggl[
    \log g(y-x; \theta) 
    - \log\Biggl(
    \int_0^T \int_0^T \lambda_1(u; \widehat{\vartheta}_1)\lambda_2(v; \widehat{\vartheta}_2)
     g(v-u; \theta) 1_{\{|v-u|\leq r\}}dudv
    \Biggr)
    \Biggr],
\end{gather*}
as considered in the literature (e.g.~\cite{jalilian2013decomposition,prokevsova2017two}) in some univariate cases.
In the second step, one could also utilize minimum contrast estimation based on the cross-K function~\cite{zhu2022minimum} as Waagepetersen \& Guan~\cite{waagepetersen2009two} did in inhomogeneous univariate cases.

\subsubsection{An inhomogeneous bivariate Neyman-Scott process}\label{subsec_inhom_NS}
A inhomogeneous version of bivariate Neyman-Scott process $N^S = (N^S_1, N^S_2)$ is defined as follows:
Let $\mathcal{C}$ be a stationary Poisson process on $\mathbb{R}$ with parent intensity $\lambda>0$.
For $i=1,2$, given $\mathcal{C}$ , define the offspring (cluster) process of component $i$ as an inhomogeneous Poisson process $N_{i,c}$ with intensity
\[
  \Lambda_{i,c}(x;\beta_i,\tau_i)
  \;=\;
  \sigma_i\,f_i(x-c;\tau_i)\,\rho_i(x;\beta_i),
  \qquad x\in\mathbb{R},\ c\in C,
\]
where $\sigma_i>0$ is a scale (mean offspring per parent up to the weight $\rho_i$), $f_i(\cdot;\tau_i)$ is a probability density on $\mathbb{R}$ parametrized by $\tau_i$, and $\rho_i(x;\beta_i)>0$ is a weight function (e.g.\ $\rho_i(x;\beta_i)=\exp\{z_i(x)^\top\beta_i\}$ where $z_i(x)$ is a covariate vector as in Waagepetersen~\cite{Waage2007}). Here, all of the offspring processes are independent.
Then, the inhomogeneous bivariate Neyman-Scott process is defined by $N_i^S=\sum_{c\in C}N_{i,c}, i=1,2$.
If $\rho_i(x;\beta_i) \equiv 1$, this model corresponds to the stationary bivariate Neyman-Scott process presented in Section \ref{sec_NS_model} with specification of the distribution of the number of offspring to $\mathrm{Poi}(\sigma_i), i=1, 2$.
The first-order intensities are
\[
  \lambda_i^S(u)
  \;=\;
  \lambda\,\sigma_i\,\rho_i(u;\beta_i),
  \qquad u\in\mathbb{R},\ i=1,2.
\]
The cross intensity function is
\[
  \lambda_{1,2}^S(u,v)
  \;=\;
  \lambda_1^S(u)\lambda_2^S(v)
  \Biggl(
  1
  \;+\;
  \frac{1}{\lambda}
  \int_{\mathbb{R}} f_1(s;\tau_1)f_2\bigl((v-u)+s;\tau_2\bigr)\,ds
  \Biggr)
\]
so that $N^S$ has the transition invariant cross-correlation function
\[
  g^S(v-u)
  \;=\;
  \frac{\lambda_{1,2}^S(u,v)}{\lambda_1^S(u)\lambda_2^S(v)}
  \;=\;
    1
  \;+\;
  \frac{1}{\lambda}
  \int_{\mathbb{R}} f_1(s;\tau_1)f_2\bigl((v-u)+s;\tau_2\bigr)\,ds.
\]

\subsubsection{An inhomogeneous noisy bivariate Neyman-Scott process}
Consider the inhomogeneous bivariate Neyman-Scott process $N^S$ in Section \ref{subsec_inhom_NS} as the signal,
and an inhomogeneous bivariate point process $N^B=(N_1^B, N_2^B)$ as the noise. 
We assume that: (i) \(N^{B}_1\perp N^{B}_2\) and \(N^B\perp N^S\); (ii) the intensity $\lambda_i^B$ of $N^B_i$ is proportional to $\rho_i(\cdot;\beta_i)$, i.e. 
\[\lambda_i^B(u) = b_i\rho_i(u;\beta_i), \quad b_i>0, i=1, 2. \]
Now we define the inhomogeneous NBNSP as $N = N^S + N^B$. 
Using the Campbell's formula, the intensity function of $N$ is
\[
    \lambda^N_i(u) =  \lambda_i^S(u) + \lambda_i^B(u) = (\lambda\sigma_i + b_i) \rho_i(u;\beta_i),\quad u\in\mathbb{R}, i=1, 2 
\] and the cross-intensity function of $N$ is 
\begin{align*}
    \lambda_{1, 2}^N(u, v) 
    &=  \lambda_1^N(u)\lambda_2^N(v) \,+\, \lambda^{S}_1(u)\lambda^{S}_2(v)\{g^{S}(v-u)-1\}  \\
    &= \lambda_1^N(u)\lambda_2^N(v) \Bigl(
            1 + \frac{\lambda\sigma_1\sigma_2}{(\lambda\sigma_1 + b_1)(\lambda\sigma_2 + b_2)}
            \int_{\mathbb{R}} f_1(s;\tau_1)f_2\bigl((v-u)+s;\tau_2\bigr)\,ds
    \Bigr),
\end{align*}
and hence, the cross-correlation function of $N$ is 
\[
    g^N(v-u; \theta) = 1 + a\int_{\mathbb{R}} f_1(s;\tau_1)f_2\bigl((v-u)+s;\tau_2\bigr)\,ds,\quad v, u\in\mathbb{R}
\]
where $\theta = (a, \tau_1, \tau_2), a=\frac{\lambda\sigma_1\sigma_2}{(\lambda\sigma_1 + b_1)(\lambda\sigma_2 + b_2)}.$

We could estimate the parameters of this model by the two-step composite-type QMLE.
This extended model could explain time-varying trading volume/order flow in real financial markets.

\subsection{Multivariate extension}
We briefly show how our framework extends to dimensions $I>2$ and how inference can be carried out. Throughout this section, $r>0$ is fixed and $W_n$ is the observation window.

Let $\mathcal{C}$ be a homogeneous Poisson process with intensity $\lambda>0$. For each component $i\in\{1,\ldots,I\}$, let $M_i(c)$ be the number of offspring triggered by $c\in \mathcal{C}$ with mean $\sigma_i\in(0,\infty)$, and let $d_i(c,m)$ be i.i.d.\ delays with density $f_i(\cdot;\tau_i)$ on $(0,\infty)$, where $\tau_i\in\mathcal{T}_i\subset\mathbb{R}^{p_i}$. Define the signal process
\[
N_i^S \;=\; \sum_{c\in \mathcal{C}}\sum_{m=1}^{M_i(c)} \delta_{\,c+d_i(c,m)},
\qquad i=1,\ldots,I,
\]
and the observed process $N_i = N_i^S + N_i^B$, where $\{N_i^B\}_{i=1}^I$ are stationary noises, independent of $\mathcal{C}$ and of $\{M_i,d_i\}$. Let $\lambda_i^B=\mathbb{E}[N_i^B([0,1])]$ and $\lambda_i=\lambda\sigma_i+\lambda_i^B$.

For $i\neq j$, the cross-intensity of $(N_i,N_j)$ satisfies
\[
\lambda^{[1,1]}_{i,j}(u)
= \lambda_i\lambda_j + \lambda\sigma_i\sigma_j\int_{\mathbb{R}} f_i(s;\tau_i) f_j(u+s;\tau_j)\,ds,
\]
hence the pairwise cross-correlation takes the same form as in the bivariate case:
\begin{equation}\label{eq:gij}
g_{i,j}(u;\theta_{i,j})
= 1 + a_{i,j}\, \int_{\mathbb{R}} f_i(s;\tau_i) f_j(u+s;\tau_j)\,ds
,\quad
a_{i,j}:=\frac{\lambda\sigma_i\sigma_j}{(\lambda\sigma_i+\lambda_i^B)(\lambda\sigma_j+\lambda_j^B)},
\end{equation}
with parameter $\theta_{i,j}=(a_{i,j},\tau_i,\tau_j)\in \Theta_{i,j}=\mathcal{A}_{i,j}\times\mathcal{T}_i\times\mathcal{T}_j$.

For a pair $(i,j)$ and window $W_n$, let
\[
\bbH_{i,j,n}(\theta_{i,j})
=
\sum_{\substack{x\in N_i,\;y\in N_j\\ x\in W_n\ominus r,\;|y-x|\le r}}
\Big\{\log g_{i,j}(y-x;\theta_{i,j})+\log\widehat\lambda_{i,n}+\log\widehat\lambda_{j,n}\Big\}
- a_n \widehat\lambda_{i,n}\widehat\lambda_{j,n}\int_{|u|\le r} g_{i,j}(u;\theta_{i,j})\,du,
\]
where $a_n=\mathrm{Leb}(W_n\ominus r)$ and $\widehat\lambda_{i,n}=a_n^{-1}N_i(W_n\ominus r)$.
Define the pairwise QMLE
\[
\widehat\theta_{i,j,n}\in\mathop{\arg\max}_{\theta_{i,j}\in\Theta_{i,j}} \bbH_{i,j,n}(\theta_{i,j}).
\]
When some parameters are shared across multiple pairs (e.g., $\tau_i$ belongs to every $\theta_{i,j}$ with $j\ne i$), we combine their multiple pairwise estimates by the averaging estimator
\begin{equation}\label{eq:tau-avg}
\widetilde\tau_{i,n} \;=\; \frac{1}{I-1}\sum_{j\ne i} \big(\widehat\tau_{i\mid j,n}\big),
\end{equation}
where $\widehat\tau_{i\mid j,n}$ denotes the $\tau_i$-block of $\widehat\theta_{i,j,n}$. We can derive the consistency and asymptotic normality of these estimators under the same regularity conditions for each pair $(i, j), i\neq j$ as in the bivariate case.

An alternative is to maximize the global composite-type objective
\[
H^{\mathrm{comp}}_n(\theta)\;=\;\sum_{1\le i<j\le I} \bbH_{i,j,n}\big(\theta_{i,j}\big),
\qquad
\theta=\big(\tau_1,\ldots,\tau_I,\ (a_{i,j})_{i<j}\big),
\]
subject to the sharing constraints on $\{\tau_i\}$. Under the same regularity, the resulting estimator could be consistent and asymptotically normal by the same Z-estimation argument. In practice, however, the dimension of $\theta$ grows as $O(I^2)$. The former pairwise-and-average method avoids such a high-dimensional search while making use of shared information across components.

\subsection{Spectral approach}
Recent work has proposed estimation and inference for noisy point processes, including noisy Hawkes models, by working in the frequency domain via the Bartlett spectrum and Whittle-type likelihoods; see, e.g., Bonnet et al.\ \cite{bonnet2025spectral} and the general asymptotic theory for spectral likelihoods in Yang \& Guan~\cite{yang2024fourier}. A key advantage is that, for a superposition $N=X+Y$ of two independent stationary point processes, the Bartlett spectral density adds: $f^N(\nu)=f^X(\nu)+f^Y(\nu)$ for all $\nu\in\mathbb{R}$ (Proposition~2.2 in~\cite{bonnet2025spectral}). Hence, when the second-order structure of the noise process is \emph{known}, one may write down $f^N$ explicitly and maximize a (matrix) Whittle likelihood over Fourier frequencies $\nu_k=k/T$ using the periodogram or cross-periodogram computed from the observed events; see Bonnet et al. \cite{bonnet2025spectral} for more details.

Our primary interest, however, is the semi-parametric case where the noise structure is \emph{unknown}. Consider the NBNSP $N=N^S+N^B$ with a ``signal'' $N^S$ and an independent background/noise $N^B$, and assume the background components are mutually independent ($N^B_1\indep N^B_2$) and independent of $N^S$. Then the off-diagonal entry of the spectral density, i.e.\ the cross-spectral density, is unaffected by the additive noise (up to a constant):
\[
f^{N}_{12}(\nu;\theta,\lambda_1,\lambda_2)\;=\;\lambda_1\lambda_2\int_{\mathbb{R}}e^{-2\pi i\nu x}\,\{g(u;\theta)-1\}\,du,
\]
where $g(\cdot;\theta)$ is the cross-correlation function of the model and $\lambda_i=\mathbb{E}[N_i([0,1])]$ are the mean intensities. In particular, in our NBNSP construction where $\{g(\cdot;\theta)-1\}/a$ is the convolution of two kernels $f_1(-\cdot;\tau_1)$ and $f_2(\cdot;\tau_2)$, the convolution theorem yields the factorization
\begin{equation}\label{eq_cross_spectrum}
f^{N}_{12}(\nu;\theta,\lambda_1,\lambda_2)\;=\;a\,\lambda_1\lambda_2\,\mathcal{F}\{f_1(\cdot;\tau_1)\}(-\nu)\,\mathcal{F}\{f_2(\cdot;\tau_2)\}(\nu),\quad \theta=(a, \tau_1, \tau_2).
\end{equation}
Here $\mathcal{F}\{h\}(\nu)=\int_{\mathbb{R}}e^{-2\pi i\nu x}h(x)\,dx$.

Let $\hat\lambda_i=T^{-1}N_i([0,T])$ and write $\hat f^{N}_{12}(\nu;\theta)=f^{N}_{12}(\nu;\theta,\hat\lambda_1,\hat\lambda_2)$. Denote the empirical cross-periodogram
\[
I^{N}_{12}(\nu)\;=\;\frac{1}{T}\sum_{\substack{x\in N_1,\;y\in N_2\\ x,y\in[0,T]}}e^{-2\pi i\nu(x-y)},\qquad \nu\in\mathbb{R}.
\]
Then, focusing on the off-diagonal information, one may consider the following adaptive spectral objective over a grid $\nu_k=k/T$:
\[
\ell_{12}(\theta)\;=\;-\frac{1}{T}\sum_{k=1}^{m}\left\{\log \hat f^{N}_{12}(\nu_k;\theta)\;+\;\frac{I^{N}_{12}(\nu_k)}{\hat f^{N}_{12}(\nu_k;\theta)}\right\},\qquad k=1,2,\ldots,m.
\]
Maximizing $\ell_{12}(\theta)$ provides a frequency-domain estimator of the structural parameters $\theta$ that does not require specifying the noise spectrum $f^{N^B}$. This may complement our time-domain approach, i.e., in models where the Fourier transform of $g(\cdot;\theta)-1$ is tractable, $f^{N}_{12}(\nu;\theta)$ has an explicit representation (\ref{eq_cross_spectrum}), which could avoid the numerical approximations for the integral of $g$.
\section{Proofs}\label{sec_proof}
\subsection{Proof of Theorems \ref{thm_consistency} and \ref{thm_asymptotic_normality}
}\label{subsec_prf_general}


We introduce some notations:
\begin{gather*}
  \bbY_n(\theta) = \frac{1}{a_n}(\bbH_n(\theta) - \bbH_n(\theta^*)), \\
  \bbY(\theta) = \lambda_1^*\lambda_2^*
  \int_{|u|\leq r} g(u; \theta^*) \Bigl[
    \log\Bigl(
    \frac{g(u; \theta)}{g(u; \theta^*)}
    \Bigr)
    - \frac{g(u; \theta)}{g(u; \theta^*)} + 1
    \Bigr] du,
\end{gather*}
where $\theta\in\ol{\Theta}$ and $n\in\bbZ_{\geq 1}$.
Note that $\bbY(\cdot)\in C(\ol{\Theta})\cap C^2(\Theta)$ and $\bbH_n(\cdot)\in C(\ol{\Theta})\cap C^2(\Theta)$ for any fixed realization of the point process $N$, thanks to the assumptions [RE]. (Consider interchanging the integral and the differentiation.)

\begin{lemma}\label{lem_consistency_intensity}
  Under [MI] and [WI],
  we have
  $
    \wih{\lambda}_{i, n} \to^p \lambda_i^*
  $
  as $n\to\infty$ for $i=1, 2$.
\end{lemma}
\begin{proof}
  Let $i=1, 2$ and
  \(
  \wit{\lambda}_{i, n} = \frac{1}{a_n} N_i\Bigl(\bigsqcup_{l\in D_n}C(l)\Bigr)
  = \frac{1}{a_n}\sum_{l\in D_n} N_i(C(l)).
  \)
  The identically distributed random field $X = \{N_i(C(l))\}_{l\in\bbZ}$ satisfies
  $\sum_{m=0}^\infty\tilde{\alpha}_{1, 1}^X(m)^{\frac{\delta}{2+\delta}} < \infty$ and $\sup_{l\in\bbZ}\|N_i(C(l))\|_{2+\delta} = \|N_i(C(0))\|_{2+\delta} <\infty$ thanks to the assumptions [MI].
  By the covariance inequality, we have
  \begin{align*}
    \mathop{\text{Var}}[\wit{\lambda}_{i, n}]
     & \leq \frac{1}{a_n^2} \sum_{l_1, l_2\in D_n} \mathrm{Cov}[N_i(C(l_1)), N_i(C(l_2))]
    \lesssim \frac{1}{a_n^2} \sum_{l_1, l_2\in D_n} \tilde{\alpha}_{1, 1}^X(|l_1-l_2|)^{\frac{\delta}{2+\delta}}                             \\
     & \leq \frac{2}{a_n^2} (\#D_n) \sum_{m=0}^\infty\tilde{\alpha}_{1, 1}^X(m)^{\frac{\delta}{2+\delta}} \to 0.\quad (\because \text{[WI]})
  \end{align*}
  On the other hand,
  \(
  E[\wit{\lambda}_{i, n}] = \frac{\#D_n}{a_n}\lambda_i^* \to \lambda_i^*.
  \)
  Thus, we have
  \(
  \wit{\lambda}_{i, n} \to^p \lambda_i^*.
  \)
  Besides,
  \begin{align*}
    E[|\wit{\lambda}_{i, n} - \wih{\lambda}_{i, n}|]
    = \frac{1}{a_n}E\Bigl[N_i\Bigl(
      \Bigl(\bigsqcup_{l\in D_n}C(l)\Bigr) \setminus W_n\ominus r
      \Bigr)\Bigr]
    = \frac{\lambda_i^*}{a_n}  (
    \#D_n - a_n
    ) \to 0. \quad (\because \text{[WI]})
  \end{align*}
  Therefore, we obtain the result.
\end{proof}
\halflineskip

\begin{lemma}\label{lem_ergodic_second}
  Suppose a Borel measurable function $h: \bbR\to \bbR$ satisfies
  $\|F_2^N(|h|; C(0))\|_2 < \infty$.
  Then, under [MI] and [WI], we have
  \begin{equation*}
    \frac{1}{a_n}F_{2, n}^N(h) \to^p \lambda_1^*\lambda_2^* \int_{|u|\leq r} h(u)g(u; \theta^*)du
  \end{equation*}
  as $n\to \infty$.
  Especially, we have
  \begin{equation*}
    \frac{1}{a_n}F_{2, n}^N\Bigl(\pt^i\log g(\cdot; \theta)\Bigr) \to^p \lambda_1^*\lambda_2^* \int_{|u|\leq r} \pt^i\log g(u; \theta)g(u; \theta^*)du
  \end{equation*}
  as $n\to \infty$
  for all $i\in\{0, 1, 2, 3\}$ and $\theta\in\Theta$ if we further assume [RE].
\end{lemma}
\begin{proof}
  Thanks to the assumption [MI], the identically distributed random field $X = \{F_2^N(h; C(l))\}_{l\in\bbZ}$ satisfies
  $\sum_{m=0}^\infty\tilde{\alpha}_{1, 1}^X(m)^{\frac{\delta}{2+\delta}} < \infty$ by (\ref{eq_note_on_mixing1}) and
  $\sup_{l\in\bbZ}\|F_2^N(h; C(l))\|_{2+\delta} = \|F_2^N(h; C(0))\|_{2+\delta} <\infty$.
  Using the covariance inequality, we have
  \begin{align*}
    \mathop{\text{Var}}\Bigl[\frac{1}{a_n}\sum_{l\in D_n}F_2^N(h; C(l))\Bigr]
     & \leq \frac{1}{a_n^2} \sum_{l_1, l_2\in D_n}
    \mathrm{Cov}[F_2^N(h; C(l_1)), F_2^N(h; C(l_2))]
    \lesssim \frac{1}{a_n^2} \sum_{l_1, l_2\in D_n} \tilde{\alpha}_{1, 1}^X(|l_1-l_2|)^{\frac{\delta}{2+\delta}}                              \\
     & \leq \frac{1}{a_n^2} (\#D_n) \sum_{m=0}^\infty\tilde{\alpha}_{1, 1}^X(m)^{\frac{\delta}{2+\delta}} \to 0. \quad (\because \text{[WI]})
  \end{align*}
  On the other hand,
  \begin{align*}
    E\Bigl[\frac{1}{a_n}\sum_{l\in D_n}F_2^N(h; C(l))\Bigr]
    = \frac{\#D_n}{a_n}\lambda_1^*\lambda_2^* \int_{|u|\leq r} h(u)g(u; \theta^*)du
    \to\lambda_1^*\lambda_2^* \int_{|u|\leq r} h(u)g(u; \theta^*)du.
  \end{align*}
  Thus, we have
  \(
  \frac{1}{a_n}\sum_{l\in D_n}F_2^N(h; C(l)) \to^p \lambda_1^*\lambda_2^* \int_{|u|\leq r} h(u)g(u; \theta^*)du.
  \)
  Besides,
  \begin{align*}
     & \quad E\Bigl[\Bigl|\frac{1}{a_n}F_{2, n}^N(h) - \frac{1}{a_n}F_2^N\Bigl(h; \bigsqcup_{l\in D_n}C(l)\Bigr)\Bigr|\Bigr]
    = \frac{1}{a_n}E\Bigl[F_2^N\Bigl(h;\Bigl(
      \Bigl(
      \bigsqcup_{l\in D_n}C(l)
      \Bigr) \setminus W_n\ominus r
    \Bigr)\Bigr)\Bigr]                                                                                                       \\
     & = \frac{1}{a_n} \mathrm{Leb}\Bigl(
    \Bigl(
    \bigsqcup_{l\in D_n}C(l)
    \Bigr) \setminus W_n\ominus r
    \Bigr) \lambda_1^*\lambda_2^* \int_{|u|\leq r} h(u)g(u; \theta^*)du
    = \frac{1}{a_n}  (
    \#D_n - a_n
    ) \lambda_1^*\lambda_2^* \int_{|u|\leq r} h(u)g(u; \theta^*)du
    \to 0
  \end{align*}
  because of the assumption [WI].
  Therefore, we obtain the result.
\end{proof}
\halflineskip

\begin{lemma}\label{lem_pointwise_lln_y}
  Assume [WI], [MI], and [RE].
  Then, we have
  $
    \bbY_n(\theta) \to^p \bbY(\theta)
  $
  as $n\to\infty$
  for all $\theta\in\Theta$.
\end{lemma}
\begin{proof}
  By definition,
  \begin{align}\label{al_pointwise_consistency1}
    \bbY_n(\theta)
     & = \frac{1}{a_n}(\bbH_n(\theta) - \bbH_n(\theta^*)) \nonumber                           \\
     & = \frac{1}{a_n}F_{2, n}^N\Bigl(\log(g(\cdot; \theta)) - \log(g(\cdot; \theta^*))\Bigr)
    - \wih{\lambda}_{1, n}\wih{\lambda}_{2, n}
    \int_{|u|\leq r}(g(u; \theta) - g(u; \theta^*))du.
  \end{align}
  Now, we have
  \begin{equation}\label{eq_pointwise_consistency1}
    \frac{1}{a_n}F_{2, n}^N(\log(g(\cdot; \theta)))
    \to^p
    \lambda_1^*\lambda_2^*\int_{|u|\leq r}\log(g(u; \theta))g(u; \theta^*)du
  \end{equation}
  by Lemma \ref{lem_ergodic_second}.
  Here, $\log(g(\cdot; \theta))g(\cdot; \theta^*)$ is indeed integrable on $[-r, r]$ thanks to the assumption [RE](iii).
  Also, we have
  \begin{equation}\label{eq_pointwise_consistency2}
    \wih{\lambda}_{1, n}\wih{\lambda}_{2, n}
    \to^p \lambda_1^*\lambda_2^*
  \end{equation}
  by Lemma \ref{lem_consistency_intensity}.
  Thus, combining (\ref{al_pointwise_consistency1}), (\ref{eq_pointwise_consistency1}), and (\ref{eq_pointwise_consistency2}), we obtain
  \begin{align*}
    \bbY_n(\theta)
     & = \frac{1}{a_n}(\bbH_n(\theta) - \bbH_n(\theta^*)) \\
     & \to^p \lambda_1^*\lambda_2^*\int_{|u|\leq r}
    \Biggl(
    \log(g(u; \theta))g(u; \theta^*)
    - g(u; \theta)
    - \log(g(u; \theta^*))g(u; \theta^*)
    + g(u; \theta^*)
    \Biggr)du                                             \\
     & = \lambda_1^*\lambda_2^*\int_{|u|\leq r}
    \Biggl(
    \log\frac{g(u; \theta)}{g(u; \theta^*)}
    - \frac{g(u; \theta)}{g(u; \theta^*)} + 1
    \Biggr)g(u; \theta^*)du
    = \bbY(\theta).
  \end{align*}
  Note that dividing by $g(u; \theta^*)$ is allowed because of the assumption [RE](i).
\end{proof}
\halflineskip

\begin{lemma}\label{lem_bound_first_diff_y}
  Under [WI], [MI], and [RE], we have
  \(
  \displaystyle
  \sup_{\theta\in\Theta}|\pt\bbY_n(\theta)| = O_p(1)
  \)
  as $n \to \infty$
  and
  \(
  \sup_{\theta\in\Theta}|\pt\bbY(\theta)| < \infty.
  \)
\end{lemma}
\begin{proof}
  First, we mention that the assumptions [RE](iii) and [RE](iv) allow us the termwise differentiation as
  \begin{equation}\label{eq_interchange_g}
    \pt \int_{|u|\leq r} g(u; \theta)du = \int_{|u|\leq r} \pt g(u; \theta)du
  \end{equation}
  and
  \begin{equation}\label{eq_interchange_log_g}
    \pt \int_{|u|\leq r} \log (g(u; \theta))g(u; \theta^*)du = \int_{|u|\leq r} \pt\log (g(u; \theta)) g(u; \theta^*)du.
  \end{equation}

  By the assumption [RE](iii) and [RE](iv), we have
  \begin{align*}
     & \quad \sup_{\theta\in\Theta}|\pt\bbY_n(\theta)|      \\
     & =
    \sup_{\theta\in\Theta}\Bigl|\frac{1}{a_n} F_{2, n}^N(\pt\log g(\cdot; \theta))
    - \wih{\lambda}_{1, n}\wih{\lambda}_{2, n}
    \pt\int_{|u|\leq r} g(u; \theta)du\Bigr|                \\
     & \leq \frac{1}{a_n}
    \sup_{\theta\in\Theta}\Bigl|F_{2, n}^N(|\pt\log g(\cdot; \theta)|)\Bigr|
    + \wih{\lambda}_{1, n}\wih{\lambda}_{2, n}
    \sup_{\theta\in\Theta} \Bigl|\int_{|u|\leq r} \pt g(u; \theta)du\Bigr|
    \quad (\because \text{(\ref{eq_interchange_g})})        \\
     & \leq \frac{1}{a_n}F_{2, n}^N(f_{B, 2}) +
    \wih{\lambda}_{1, n}\wih{\lambda}_{2, n}\int_{|u|\leq r} f_{B, 1}(u)du
    \quad (\because \text{[RE](iii) \& [RE](iv)})           \\
     & \to^p \lambda_1^*\lambda_2^* \int_{|u|\leq r} \Bigl[
      f_{B, 2}(u)g(u; \theta^*) + f_{B, 1}(u)
      \Bigr]du < \infty.
    \quad (\because \text{Lemmas \ref{lem_consistency_intensity} and \ref{lem_ergodic_second}})
  \end{align*}
  Thus, we have $\sup_{\theta\in\Theta}|\pt\bbY_n(\theta)| = O_p(1)$.
  Besides,
  \begin{align*}
    \sup_{\theta\in\Theta}|\pt\bbY(\theta)|
     & = \lambda_1^*\lambda_2^*  \sup_{\theta\in\Theta}
    \Bigl| \pt \int_{|u|\leq r} g(u; \theta^*) \Bigl[
      \log\Bigl(
      \frac{g(u; \theta)}{g(u; \theta^*)}
      \Bigr)
      - \frac{g(u; \theta)}{g(u; \theta^*)} + 1
      \Bigr] du
    \Bigr|                                              \\
     & = \lambda_1^*\lambda_2^*  \sup_{\theta\in\Theta}
    \Bigl|\int_{|u|\leq r}  \Bigl[
      \pt\log (g(u; \theta)) g(u; \theta^*)
      - \pt g(u; \theta)
      \Bigr] du
    \Bigr|                                              \\
     & \leq \lambda_1^*\lambda_2^*
    \int_{|u|\leq r}  \Bigl[
      f_{B, 2}(u)g(u; \theta^*)
      + f_{B, 1}(u)
      \Bigr] du < \infty.
  \end{align*}
  Therefore, we obtain the result.
\end{proof}
\halflineskip

\begin{lemma}\label{lem_id_y}
  Under [RE] and [ID], we have
  $\bbY(\theta^*) = 0$ and
  $\bbY(\theta) < 0$ for all $\theta\in\ol{\Theta}\setminus\{\theta^*\}$.
\end{lemma}
\begin{proof}
  The first assertion follows from a direct calculation.
  Let us show the second assertion.
  By the inequality
  $\log x - x + 1\leq 0, \ x>0$, we have
  \(
  \bbY(\theta)\leq 0
  \) for all $\theta\in\ol{\Theta}$.
  Suppose $\bbY(\theta) = 0$ for some $\theta\in\ol{\Theta}$.
  Since the equality in $\log x - x + 1\leq 0$ holds only if $x=1$,
  \begin{equation*}
    \bbY(\theta) = \int_{|u|\leq r} g(u; \theta^*) \Bigl[
      \log\Bigl(
      \frac{g(u; \theta)}{g(u; \theta^*)}
      \Bigr)
      - \frac{g(u; \theta)}{g(u; \theta^*)} + 1
      \Bigr] du = 0
  \end{equation*}
  implies $g(\cdot; \theta) = g(\cdot; \theta^*)$ a.e. on $[-r, r]$, due to the assumption [RE](i).
  Therefore, we have $\theta=\theta^*$ by the assumption [ID].
\end{proof}
\halflineskip

\begin{proof}[Proof of Theorem \ref{thm_consistency}]
  By Lemmas \ref{lem_pointwise_lln_y} and \ref{lem_bound_first_diff_y},
  we have
  \(
  \sup_{\theta\in\Theta}|\bbY_n(\theta) - \bbY(\theta)| \to^p 0,
  \)
  thanks to the convexity and boundedness of $\Theta$ assumed in [PA].
  By the continuity of $\bbY_n(\cdot) - \bbY(\cdot)$ on $\ol{\Theta}$, we derive
  \(
  \sup_{\theta\in\ol{\Theta}}|\bbY_n(\theta) - \bbY(\theta)| \to^p 0.
  \)
  Together with Lemma \ref{lem_id_y}, we obtain the result.
\end{proof}
\halflineskip

Theorem \ref{thm_asymptotic_normality} is proven by a usual argument on the Z-estimation thanks to Theorem \ref{thm_consistency}, Lemmas \ref{lem_observed_information}, \ref{lem_residua_2}, and \ref{lem_score_convergence} below.

\begin{proof}[Proof of Theorem~4.2]
    By the consistency of the estimator sequence $\wih{\theta}_n$, 
    there exists a sequence of positive numbers $\delta_n>0$ converging to zero such that 
    $B(\theta^*, \delta_n)\subset \Theta$,
    $\lim_{n\to \infty} P[\wih{\theta}_n \in B(\theta^*, \delta_n)] = 1$.
    When $\wih{\theta}_n \in B(\theta^*, \delta_n)$, we have
    \[
        0 = \frac{1}{\sqrt{a_n}} \pt H_n(\wih{\theta}_n) 
        = \frac{1}{\sqrt{a_n}}\pt H_n(\theta^*) + \int_0^1 \frac{1}{a_n}\pt^2 H_n(t\theta^* + (1-t)(\wih{\theta}_n - \theta^*))dt
        \times 
        \sqrt{a_n}(\wih{\theta}_n - \theta^*)
    \]
    by Taylor's theorem.
    Then, setting
    \[
    \overline\Gamma_n
    := -\frac{1}{a_n}\int_0^1
       \pt^2\bbH_n\bigl(\theta^*+t(\widehat\theta_n-\theta^*)\bigr)\,dt,
    \]
    we obtain the exact identity
    \begin{equation}\label{eq:z-expansion}
    \overline\Gamma_n\,\sqrt{a_n}(\widehat\theta_n-\theta^*)
    = \frac{1}{\sqrt{a_n}}\pt\bbH_n(\theta^*).
    \end{equation}

    We show $\overline\Gamma_n\to^p\Gamma$.
    By Lemma~\ref{lem_observed_information},
    $-(1/a_n)\pt^2\bbH_n(\theta^*)\to^p\Gamma$.
    Moreover, by Lemma~\ref{lem_residua_2},
    \[
    \sup_{t\in[0, 1]}\Bigl|
      \frac{1}{a_n}\pt^2\bbH_n\bigl(\theta^*+t(\widehat\theta_n-\theta^*)\bigr)
      -\frac{1}{a_n}\pt^2\bbH_n(\theta^*)
    \Bigr|
    \leq \frac{1}{a_n}\sup_{\theta\in B(\theta^*, \delta_n)}
        \Bigl|
      \pt^2\bbH_n(\theta)
      -\pt^2\bbH_n(\theta^*)
        \Bigr|
    = o_p(1).
    \]
    Hence $\overline\Gamma_n=-(1/a_n)\pt^2\bbH_n(\theta^*)+o_p(1)\to^p\Gamma$.
    By [ID2], $\Gamma$ is positive definite, so
    $\overline\Gamma_n$ is invertible with probability tending to one and
    $\overline\Gamma_n^{-1}\to^p\Gamma^{-1}$.
    
    By Lemma~\ref{lem_score_convergence},
    $\frac{1}{\sqrt{a_n}}\pt\bbH_n(\theta^*)\to^d N(0,\Sigma)$ for some
    nonnegative definite matrix $\Sigma$.
    Therefore, from \eqref{eq:z-expansion} and Slutsky's theorem,
    \[
    \sqrt{a_n}(\widehat\theta_n-\theta^*)
    \;\;\to^d\;\;
    \Gamma^{-1}Z
    \sim N\!\bigl(0,\,\Gamma^{-1}\Sigma\Gamma^{-1}\bigr).
    \] 
    This proves the claim.
\end{proof}

\begin{lemma}\label{lem_observed_information}
  Under [WI2], [RE], and [MI], we have
  \(
  \frac{1}{a_n}\pt^2\bbH_n(\theta^*) \to^p -\Gamma
  \)
  as $n\to\infty$.
\end{lemma}
\begin{proof}
  First, we mention that the assumptions [RE](iii) and [RE](iv) allow us interchanging the differentiation and integral as
  \begin{equation}\label{eq_interchange_g_2}
    \pt^2 \int_{|u|\leq r} g(u; \theta)du = \int_{|u|\leq r} \pt^2 g(u; \theta)du.
  \end{equation}
  Then, we observe that
  \begin{align*}
     & \quad \frac{1}{a_n}\pt^2\bbH_n(\theta^*)                                                          \\
     & = \frac{1}{a_n}F_{2, n}^N(\pt^2\log g(\cdot; \theta^*))
    - \wih{\lambda}_{1, n}\wih{\lambda}_{2, n}
    \pt^2\int_{|u|\leq r}g(u; \theta^*)du                                                                \\
     & = \frac{1}{a_n}F_{2, n}^N(\pt^2\log g(\cdot; \theta^*))
    - \wih{\lambda}_{1, n}\wih{\lambda}_{2, n}
    \int_{|u|\leq r}\pt^2g(u; \theta^*)du
    \quad (\because \text{(\ref{eq_interchange_g_2})})                                                   \\
     & = \lambda_1^*\lambda_2^*
    \int_{|u|\leq r} (\pt^2\log g(u; \theta^*))g(u; \theta^*)du
    - \lambda_1^*\lambda_2^*
    \int_{|u|\leq r}\pt^2g(u; \theta^*)du
    + o_p(1)
    \quad (\because \text{Lemmas \ref{lem_consistency_intensity} and \ref{lem_ergodic_second}})          \\
     & = \lambda_1^*\lambda_2^*
    \int_{|u|\leq r} \Bigl[(\pt^2\log g(u; \theta^*))g(u; \theta^*)-\pt^2g(u; \theta^*)\Bigr]du + o_p(1) \\
     & = -\lambda_1^*\lambda_2^*
    \int_{|u|\leq r} \frac{\pt g(u; \theta^*)^{\otimes 2}}{g(u; \theta^*)}du + o_p(1) = -\Gamma + o_p(1).
  \end{align*}
  Therefore, we obtain the result.
\end{proof}
\halflineskip

\begin{lemma}\label{lem_residua_2}
  Assume the conditions [WI2], [MI] and [RE].
  Let $V_n\subset \Theta$ be a sequence of neighborhood of $\theta^*$ shrinking to $\{\theta^*\}$.
  Then, we have
    \[
        \frac{1}{a_n}\sup_{\theta\in V_n}
            \Bigl|
          \pt^2\bbH_n(\theta)
          -\pt^2\bbH_n(\theta^*)
            \Bigr| = o_p(1)
    \]
  as $n\to\infty$.
\end{lemma}

\begin{proof}
By [RE](iii)–(iv) and dominated convergence, differentiation and integration can be interchanged:
\[
\pt^2\!\int_{|u|\le r} g(u;\theta)\,du \;=\; \int_{|u|\le r}\pt^2 g(u;\theta)\,du.
\]
Hence
\[
\frac{1}{a_n}\pt^2\bbH_n(\theta)
=
\frac{1}{a_n}F_{2,n}^N\!\bigl(\pt^2\log g(\cdot;\theta)\bigr)
\;-\;
\wih\lambda_{1,n}\,\wih\lambda_{2,n}\int_{|u|\le r}\pt^2 g(u;\theta)\,du.
\]
Let $V_n\subset\Theta$ be neighborhoods shrinking to $\{\theta^*\}$. Then
\begin{align*}
\frac{1}{a_n}\sup_{\theta\in V_n}\Bigl|\pt^2\bbH_n(\theta)-\pt^2\bbH_n(\theta^*)\Bigr|
&\le
\underbrace{\sup_{\theta\in V_n}\Bigl|
\frac{1}{a_n}F_{2,n}^N\!\bigl(\pt^2\log g(\cdot;\theta)-\pt^2\log g(\cdot;\theta^*)\bigr)
\Bigr|}_{=:E_n}\\
&\quad+
\underbrace{\bigl|\wih\lambda_{1,n}\wih\lambda_{2,n}\bigr|\,
\sup_{\theta\in V_n}\Bigl|
\int_{|u|\le r}\bigl\{\pt^2 g(u;\theta)-\pt^2 g(u;\theta^*)\bigr\}\,du
\Bigr|}_{=:R_n}.
\end{align*}

First, we will evaluate $R_n$.
Define
\[
k_n(u):=\sup_{\theta\in V_n}\bigl|\pt^2 g(u;\theta)-\pt^2 g(u;\theta^*)\bigr|,\quad u\in[-r,r].
\]
For each fixed $u$, the map $\theta\mapsto\pt^2 g(u;\theta)$ is continuous on $\Theta$ by [RE](ii); since $V_n\downarrow\{\theta^*\}$, we have $k_n(u)\to 0$ pointwise in $u$. Moreover, by [RE](iii),
\(
0\le k_n(u)\le 2\sup_{\theta\in\Theta}|\pt^2 g(u;\theta)|\le 2f_{B,1}(u)
\)
with $\int_{|u|\le r}f_{B,1}(u)\,du<\infty$.
Thus, by dominated convergence,
\[
    K_n := \int_{|u|\leq r}k_n(u)\,du \to 0.
\]
By Lemma \ref{lem_consistency_intensity}, $\wih\lambda_{1,n}\wih\lambda_{2,n}=O_p(1)$ holds, 
so that we have 
\[R_n \leq K_n \times \wih\lambda_{1,n}\wih\lambda_{2,n} = o_p(1)\times O_p(1)=o_p(1).\]

Next, we will deal with $E_n$.
Set
\[
l_n(u):=\sup_{\theta\in V_n}\bigl|\pt^2\log g(u;\theta)-\pt^2\log g(u;\theta^*)\bigr|,
\quad u\in[-r,r].
\]
By [RE](i) and (ii), $l_n(u)\to0$ pointwise in $u$, and by [RE](iv),
\(
0\le l_n(u)\le 2\sup_{\theta\in\Theta}|\pt^2\log g(u;\theta)|\le 2f_{B,2}(u),
\)
with $\int_{|u|\le r}f_{B,2}(u)g(u;\theta^*)\,du<\infty$.
Using linearity, positivity, and the triangle inequality for $F_2^N(\cdot;W)$, we have
\[
E_n \;\le\; \frac{1}{a_n}F_{2,n}^N\!\bigl(l_n\bigr) =: S_n\geq 0.
\]
By Campbell's formula \eqref{cross_intensity_formula} and dominated convergence with 
$l_n(u)\,g(u;\theta^*)\leq 2f_{B,2}(u)g(u;\theta^*), u\in [-r, r]$, 
we have
\[
E[S_n]
= \lambda_1^*\lambda_2^* \int_{|u|\le r} l_n(u)\,g(u;\theta^*)\,du
\to 0.
\]
Therefore $S_n\to^p0$, and hence $E_n=o_p(1)$.

\medskip
Finally, combining the bounds for $E_n$ and $R_n$,
\[
\frac{1}{a_n}\sup_{\theta\in V_n}\Bigl|\pt^2\bbH_n(\theta)-\pt^2\bbH_n(\theta^*)\Bigr|
\;=\;E_n+R_n
\;=\;o_p(1).
\]
\end{proof}



\begin{lemma}\label{thm_bolthausen_clt}
  Suppose $X_l=(X_{l, 1}, \cdots, X_{l, d}), l\in\bbZ$ is a stationary $\bbR^d$-valued random field on $\bbZ$,
  $E[X_1] = 0$, and $S_n = \sum_{l=1}^n X_l$.
  Assume that there exists some $\delta>0$ such that
  $\|X_0\|_{2+\delta}< \infty$ and $\sum_{m=1}^{\infty}\tilde{\alpha}_{2, \infty}^X(m)^{\frac{\delta}{2+\delta}}<\infty$.
  Then, we have the convergence
  \(
  \frac{1}{n}\mathop{\text{Var}}[S_n] \to \Sigma
  \)
  as $n\to\infty$
  for some non-negative definite matrix $\Sigma$.
  Moreover, we have
  \(
  \frac{1}{\sqrt{n}}S_n \to^d N(0, \Sigma)
  \)
  as $n\to\infty$.
\end{lemma}
\begin{proof}
  Let $S_{n, i} = \sum_{l=1}^n X_{l, i}$, $i=1, \ldots, d$.
  For $i, j=1, \ldots, d$, we have
  \begin{align*}
    \frac{1}{n}\mathop{\text{Cov}}[S_{n, i}, S_{n, j}]
     & = \frac{1}{n}\sum_{l_1=1}^n \sum_{l_2=1}^n \mathop{\text{Cov}}[X_{l_1, i}, X_{l_2, j}] \\
     & = \mathrm{Cov}[X_{1, i}, X_{1, j}]
    + 2\sum_{k=1}^{n-1} \Bigl(1-\frac{k}{n}\Bigr)\mathop{\text{Cov}}[X_{1, i}, X_{k, j}],
  \end{align*}
  using the stationarity.
  By the mixing and moment assumptions, we have
  \[
    \sum_{k=1}^{n-1}\Bigl(1-\frac{k}{n}\Bigr)|\mathop{\text{Cov}}[X_{1, i}, X_{k, j}]|
    \lesssim \sum_{k=1}^{\infty} \tilde{\alpha}_{1, 1}^X(k)^{\frac{\delta}{2+\delta}} < \infty.
  \]
  Thus, $\sum_{k=1}^{\infty} (1-\frac{k}{n})\mathop{\text{Cov}}[X_{1, i}, X_{k, j}]$ converges, and we have
  \(
  \frac{1}{n}\mathop{\text{Cov}}[S_{n, i}, S_{n, j}] \to \sigma_{i, j}
  \)
  as $n\to\infty$ for some $\sigma_{ij} \in \bbR$.
  Since the limit $\Sigma = (\sigma_{i, j})$ inherits the non-negative definiteness of $\frac{1}{n}\mathop{\text{Var}}[S_n]$, the first assertion follows.

  Next, we will deal with the second assertion.
  By the Cramér-Wold device, it is sufficient to show
  \begin{equation}\label{eq_bolthausen_clt_cr_device}
    \forall a\in\bbR^d: \quad
    \frac{1}{\sqrt{n}} a'S_n \to^d N(0, a'\Sigma a)
  \end{equation}
  as $n\to\infty$.
  From the first assertion, we have
  \(
  \frac{1}{n}\mathop{\text{Var}}[a' S_n] = a' \Bigl(\frac{1}{n}\mathop{\text{Var}}[S_n]\Bigr) a
  \to a'\Sigma a
  \) as $n\to\infty$.
  If $a'\Sigma a > 0$, we can apply the theorem in Bolthausen~\cite{bolthausen1982central}
  so that we obtain (\ref{eq_bolthausen_clt_cr_device}).
  If $a'\Sigma a = 0$, we also have (\ref{eq_bolthausen_clt_cr_device})
  because
  $E[a'S_n] = 0$ and $\mathop{\text{Var}}[\frac{1}{\sqrt{n}}a' S_n] \to a'\Sigma a = 0$ as $n\to\infty$.
  Therefore, we obtain the result.
\end{proof}
\halflineskip

\begin{lemma}\label{lem_score_convergence}
  Under [WI2], [MI] and [RE],
  we have the convergence
  \(
  \frac{1}{\sqrt{a_n}}\pt\bbH_n(\theta^*) \to^d N(0, \Sigma)
  \)
  as $n\to\infty$
  for some non-negative definite symmetric matrix $\Sigma$.
\end{lemma}
\begin{proof}
  Firstly,
  $\pt\bbH_n(\theta^*)$ can be decomposed as
  \begin{align}\label{al_general_clt_proof_1}
    \frac{1}{\sqrt{a_n}}\pt\bbH_n(\theta^*)
     & = \frac{1}{\sqrt{a_n}}\Bigl(\sum_{l\in D_n} Z_l\Bigr)
    - \sqrt{a_n}(\wih{\lambda}_{1, n}\wih{\lambda}_{2, n} - \lambda_1^*\lambda_2^*)
    \int_{|u|\leq r}\pt g(u; \theta^*)du      \nonumber                                                               \\
     & \quad - \frac{1}{\sqrt{a_n}}
    F_2^N\Bigl(\pt\log g(\cdot; \theta^*); \Bigl(\bigsqcup_{l\in D_n}C(l)\Bigr) \setminus W_n\ominus r\Bigr)\nonumber \\
     & \quad + \frac{1}{\sqrt{a_n}}(\#D_n - a_n)\lambda_1^*\lambda_2^*
    \int_{|u|\leq r}\pt g(u; \theta^*)du
  \end{align}
  where
  \begin{align*}
    Z_l
     & = F_2^N\Bigl(\pt\log g(\cdot; \theta^*); C(l)\Bigr)
    - \lambda_1^*\lambda_2^*\int_{|u|\leq r}\pt g(u; \theta^*)du \\
     & = F_2^N\Bigl(\pt\log g(\cdot; \theta^*); C(l)\Bigr)
    - E\Bigl[F_2^N\Bigl(\pt\log g(\cdot; \theta^*); C(l)\Bigr)\Bigr], \quad l\in\bbZ.
  \end{align*}
  Here, the third and fourth terms in the RHS of (\ref{al_general_clt_proof_1}),
  are $o_p(1)$ because of [WI2]. (Regarding the third term, consider taking expectation and use the assumption [RE](iv).)
  Thus, we have
  \begin{equation}\label{eq_score_modification}
    \frac{1}{\sqrt{a_n}}\pt\bbH_n(\theta^*) = \frac{1}{\sqrt{a_n}}\sum_{l\in D_n} Z_l - \sqrt{a_n}(\wih{\lambda}_{1, n}\wih{\lambda}_{2, n} - \lambda_1^*\lambda_2^*)
    + o_p(1).
  \end{equation}
  Besides, for $i=1, 2$, we have
  \begin{align}\label{al_error_intensity}
    \sqrt{a_n}(\wih{\lambda}_{i, n} - \lambda_i^*)
     & = \frac{1}{\sqrt{a_n}}\sum_{l\in D_n} \Bigl(N_i(C(l)) - \lambda_i^*\Bigr)
    + \frac{1}{\sqrt{a_n}} N_i\Bigl(\Bigl(\bigsqcup_{l\in D_n}C(l)\Bigr) \setminus W_n\ominus r\Bigr)
    + \frac{\lambda_i^*}{\sqrt{a_n}}(\#D_n - a_n)  \nonumber                              \\
     & =\frac{1}{\sqrt{a_n}}\sum_{l\in D_n} \Bigl(N_i(C(l)) - \lambda_i^*\Bigr) + o_p(1).
  \end{align}
  Now, we would like to apply Lemma \ref{thm_bolthausen_clt} to the $\bbR^{p+2}$-valued random field
  \begin{equation*}
    X_l =
    \left(
    \begin{array}{c}
      N_1(C(l)) - \lambda_1^* \\
      N_2(C(l)) - \lambda_2^* \\
      Z_l
    \end{array}
    \right), \quad l\in\bbZ,
  \end{equation*}
  so we will check the assumptions.
  First, since $N$ is stationary, $X=\{X_l\}_{l\in\bbZ}$ is also stationary.
  Second, because $N_i(C(l))$ and $Z_l$ only depend on $N\cap{(C(l)\oplus r)}$, we have
  $\sum_{m=1}^{\infty}\alpha_{2, \infty}^X(m)^{\frac{\delta}{2+\delta}}<\infty$ by (\ref{eq_note_on_mixing1}) and the assumption [MI].
  Finally, the moment condition $\|X_0\|_{2+\delta}<\infty$ also follows from the assumption [MI].
  Then, the additional assumption [WI2] enable us to apply Lemma \ref{thm_bolthausen_clt}, and then we have
  \begin{equation*}
    \frac{1}{\sqrt{a_n}}\sum_{l\in D_n}
    \left(
    \begin{array}{c}
        N_1(C(l)) - \lambda_1^* \\
        N_2(C(l)) - \lambda_2^* \\
        Z_l
      \end{array}
    \right)
    =
    \left(
    \begin{array}{c}
        \sqrt{a_n}(\wih{\lambda}_{1, n} - \lambda_1^*) \\
        \sqrt{a_n}(\wih{\lambda}_{2, n} - \lambda_2^*) \\
        \frac{1}{\sqrt{a_n}}\sum_{l\in D_n} Z_l
      \end{array}
    \right)
    \to^d
    \left(
    \begin{array}{c}
        \Delta_1 \\
        \Delta_2 \\
        \Delta^*
      \end{array}
    \right)
    \sim N(0, \Sigma^*),
  \end{equation*}
  where $\Sigma^*$ is some nonnegative definite matrix.
  Using the delta method, we have
  \begin{equation*}
    \left(
    \begin{array}{c}
      \sqrt{a_n}(\wih{\lambda}_{1, n}\wih{\lambda}_{2, n} - \lambda_1^*\lambda_2^*) \\
      \frac{1}{\sqrt{a_n}}\sum_{l\in D_n} Z_l
    \end{array}
    \right)
    \to^d
    \left(
    \begin{array}{c}
      \lambda_2^*\Delta_1 + \lambda_1^*\Delta_2 \\
      \Delta^*
    \end{array}
    \right).
  \end{equation*}
  Since
  \begin{equation*}
    \frac{1}{\sqrt{a_n}}\pt\bbH_n(\theta^*)
    =  \frac{1}{\sqrt{a_n}}\sum_{l\in D_n} Z_l
    - \sqrt{a_n}(\wih{\lambda}_{1, n}\wih{\lambda}_{2, n} - \lambda_1^*\lambda_2^*)
    \int_{|u|\leq r}\pt g(u; \theta^*)du + o_p(1)
  \end{equation*}
  from (\ref{al_error_intensity}) and (\ref{eq_score_modification}),
  we obtain
  \begin{equation*}
    \frac{1}{\sqrt{a_n}}\pt\bbH_n(\theta^*) \to^d \Delta^* - (\lambda_2^*\Delta_1 + \lambda_1^*\Delta_2)
    \int_{|u|\leq r}\pt g(u; \theta^*)du \sim N(0, \Sigma).
  \end{equation*}
\end{proof}
\halflineskip

\subsection{Proof of Theorem \ref{thm_noisy_NS_asymptotics}}\label{subsec_proof_thm_noisy_symptotics}
First, we present a brief sketch of the proof.
We reduce the consistency and asymptotic normality of the QMLE for the NBNSP to the general results in Section~\ref{sec_general_theory2} (Theorems~\ref{thm_consistency} and~\ref{thm_asymptotic_normality}). To this end, we verify the conditions [PA], [RE], and [MI] for the NBNSP. The conditions [ID] and [ID2] will be verified for specific models in Section \ref{sec_specific_models}. The verification is organized as follows.

\medskip
\noindent\textbf{1) Parameter space.}
From [NS](i), the parameter domain is a product of bounded, open, convex sets; hence [PA] holds (Lemma~\ref{lem_NS_PA}).

\medskip
\noindent\textbf{2) Regularity under possibly diverging kernels.}
The cross-correlation function has the convolution form
\[
g(u;\theta)=1+a\int f_1(s;\tau_1)f_2(u+s;\tau_2)\,ds.
\]
Unlike standard settings, \(g\) and its log-derivatives may diverge at \(u=0\). Lemma~\ref{lem_NS_RE} ensures the existence of integrable envelopes \(f_{B,1},f_{B,2}\) on \([-r,r]\) such that
\[
\sup_{\theta}\bigl|\partial_{\theta}^{i}g(u;\theta)\bigr|\le f_{B,1}(u),\quad
\sup_{\theta}\bigl|\partial_{\theta}^{i}\log g(u;\theta)\bigr|\le f_{B,2}(u),\quad i=0,1,2,
\]
with \(\int f_{B,1}<\infty\) and \(\int f_{B,2}(\cdot)\,g(\cdot;\theta^\ast)<\infty\), which are required from the conditions in [RE]. The key step is a precise near-origin analysis of the kernel convolution. 

\medskip
\noindent\textbf{3) Mixing and high-order moments.}
For the signal part (a bivariate Neyman–Scott process), Lemma~\ref{lem_mixing_NS} bounds the \(\alpha\)-mixing rate by the tail probabilities of the dispersal kernels. Lemma~\ref{lem_ns_existence_of_moment} establishes the existence of higher-order moments. For the noise part, [NS](iv) imposes analogous moment/mixing conditions. 
Then, the bound required in [MI], that is, 
\[
\bigl\|F_N^2\!\bigl(|\partial_{\theta}^{k}\log g|;\,C(0)\bigr)\bigr\|_{2+\delta}<\infty\quad(k=0,1,2,3)
\]
is obtained thanks to Lemmas~\ref{lem_log_ccf_derivative_estimate_parameter_wise}, \ref{lem_high_moment_NS}, \ref{lem_high_moment_NS_2}, and \ref{lem_high_moment_NS_3} by expanding moments with factorial cumulant measures and by using H\"older’s inequality repeatedly together with the fact that \(f_i\) are probability densities; this eliminates integrals over parents and yields integrability even if the kernel $f_i$ is not bounded. These ingredients imply [MI] (Proposition~\ref{lem_NS_MI}).

\halflineskip
We now proceed to the proof.

\begin{lemma}\label{lem_NS_PA}
  Under [NS](i), the condition [PA] holds.
\end{lemma}
\begin{proof}
  The product of open, convex, and bounded subsets from Euclidean space is also open, convex, and bounded.
\end{proof}
\halflineskip

\begin{lemma}\label{lem_NS_RE}
  Under [NS](ii), our model satisfies the conditions [RE].
\end{lemma}
\begin{proof}
  In this proof, the supremum about $\tau_i$ is taken on $\ol{\calt_i}$ for $k=0$ and on $\calt_i$ for $k=1, 2, 3$ so long as there is no risk of confusion. Also, the notation $X \lesssim Y$ means that there is a constant $C > 0$ such that $X \leq CY$, where $C$ depends on neither $u$ nor the parameters $\tau_1$, $\tau_2$, and $a$.
  Since $g(u; \theta)\geq 1$, we have [RE](i).
  Next, we will deal with [RE](ii) and (iii).
  Let $i=1, 2$ and $k=0, 1, 2, 3$.
  Using the Leibniz rule, we have
  \begin{equation}\label{eq_ns_kernel_diff}
    \ptaui^k f_i(u; \tau_i)=
    \left\{
    \begin{array}{c}
      \sum_{j=0}^k \binom{k}{j}
      \Bigl(\ptaui^j h_{i, 1}(u; \tau_i)\Bigr)
      \Bigl(\ptaui^{k-j} h_{2, i}(\tau_i)\Bigr)
      (\log{u})^{k-j} u^{h_{i, 2}(\tau_i) - 1}, \quad 0<u<1, \\
      \ptaui^k h_{i, 3}(u; \tau_i), \quad u>1,               \\
      0, \quad u\leq 0.
    \end{array}
    \right.
  \end{equation}
  Then, by the assumption for $h_{i, 1}$ and $h_{i, 2}$, there exist $1/2 >\beta>0$ such that
  \begin{equation}\label{eq_ns_proof_assumption1}
    \sup_{\tau_i}|\ptaui^k f_i(u; \tau_i)|
    \lesssim u^{\beta-1}, \quad 0<u<1
  \end{equation}
  because $|\log u|\leq u^{-\epsilon}, 0<u<1$ for an arbitrary small $\epsilon >0$.
  Combining (\ref{eq_ns_assumption1}) and (\ref{eq_ns_proof_assumption1}), we have
  \begin{equation}\label{eq_ns_proof_assumption2}
    \sup_{\tau_i}|\ptaui^k f_i(u; \tau_i)|
    \lesssim u^{\beta-1}1_{(0, 1)}(u) + \tilde{f}(u)1_{[1, \infty)}
    =: \wit{f}_0(u), \ u\in\bbR
  \end{equation}
  for some $1/2>\beta>0$ and a nonnegative bounded $L^1([1, \infty))$ function $\tilde{f}$.
  Thus, we have
  \begin{equation}\label{eq_ns_proof_assumption3}
    \sup_{\tau_1, \tau_2}|\partial_{\tau_1}^{k_1}f_1(s; \tau_1)\partial_{\tau_2}^{k_2}f_2(u+s; \tau_2)|
    \leq \wit{f}_0(s)\wit{f}_0(u+s), \quad s\in\bbR
  \end{equation}
  for all $u\in\bbR$ and $k_1, k_2 \in\{0, 1, 2, 3\}$ such that $k_1+k_2\leq 3$.
  We are going to examine the integral of the RHS.
  We observe
  \begin{align*}
    \int_{\bbR}\wit{f}_0(s)\wit{f}_0(u+s)ds
     & \leq \int_{\bbR}s^{\beta-1}(s+u)^{\beta-1}1_{\{0<s<1, 0<s+u<1\}}ds
    + \int_{\bbR}s^{\beta-1}\tilde{f}(s+u)1_{\{0<s<1, s+u\geq 1\}}ds \nonumber  \\
     & \quad + \int_{\bbR}\tilde{f}(s)(s+u)^{\beta-1}1_{\{s\geq 1, 0<s+u<1\}}ds
    + \int_{\bbR}\tilde{f}(s)\tilde{f}(s+u)1_{\{s\geq 1, s+u\geq 1\}}ds.
  \end{align*}
  for all $u\neq 0$.
  Since $\tilde{f}$ is bounded, the second and third integrals are bounded by some constant not depending on $u$. The fourth integral is a convolution of bounded and integrable functions on $\bbR$ so that it is continuous on $\bbR$ as a function of $u\in\bbR$. In particular, it is bounded on $[-r, r]$.
  Thus, there exists some constant $C>0$ depending on neither $u$ nor the parameters such that
  \begin{equation*}
    \int_{\bbR}\wit{f}_0(s)\wit{f}_0(u+s)ds
    \leq \int_{\bbR}s^{\beta-1}(s+u)^{\beta-1}1_{\{0<s<1, 0<s+u<1\}}ds + C, \quad 0<|u|\leq r.
  \end{equation*}
  Here, we have
  \begin{align*}
     & \quad \int_{\bbR}s^{\beta-1}(s+u)^{\beta-1}1_{\{0<s<1, 0<s+u<1\}}ds \\
     & = \left\{
    \begin{array}{c}\displaystyle
      u^{2\beta-1}\int_0^{\frac{1}{u}-1} t^{\beta-1}(t+1)^{\beta-1}dt \quad (0<u<1) \\
      \displaystyle
      |u|^{2\beta-1}\int_0^{\frac{1}{|u|}-1} (t+1)^{\beta-1}t^{\beta-1}dt \quad (-1<u<0)
    \end{array}
    \right.                                                                \\
     & \leq |u|^{2\beta-1}\int_0^{\infty} t^{\beta-1}(t+1)^{\beta-1}dt
    \lesssim |u|^{2\beta-1}, \quad 0<|u|<1.
  \end{align*}
  by the change of variable $s=ut$ [resp. $s=-u(t+1)$] for $0<u<1$ [resp. $-1<u<0$].
  Thus, we obtain
  \begin{equation}\label{eq_ns_proof_assumption8}
    \int_{\bbR}\wit{f}_0(s)\wit{f}_0(u+s)ds
    \lesssim \Bigl(|u|^{-(1-2\beta)} 1_{\{0<|u|<1\}}(u) + 1\Bigr), \quad 0<|u|\leq r.
  \end{equation}
  Therefore, by (\ref{eq_ns_proof_assumption3}) and (\ref{eq_ns_proof_assumption8}), the dominated convergence theorem implies that
  \[
    p(u; \tau_1, \tau_2) := \int_{\bbR}f_1(s; \tau_1)f_2(u+s; \tau_2)ds
  \]
  is in $C(\ol{\calt_1\times\calt_2})\cap C^3(\calt_1\times\calt_2)$ for all fixed $u\in[-r, r]\setminus \{0\}$.
  Consequently, the cross-correlation function
  \[g(u; \theta) = 1 + 1_{\{\bbR\setminus \{0\}\}}(u) \times a\times p(u; \tau_1, \tau_2), \quad \theta = (a, \tau_1, \tau_2)\]
  given in (\ref{eq_ns_cross_corr}) satisfies [RE](ii).
  Moreover, we can interchange the differentiations and the integrals as
  \begin{equation}\label{eq_ns_re_interchange_diff_int}
    \partial_{\tau_1}^{k_1}\partial_{\tau_2}^{k_2}
    p(u; \tau_1, \tau_2) = \int_{\bbR}\partial_{\tau_1}^{k_1}f_1(s; \tau_1)\partial_{\tau_2}^{k_2}f_2(u+s; \tau_2)ds, \quad \tau_i\in\calt_i, i=1, 2
  \end{equation}
  for $u\in[-r, r]\setminus \{0\}$ and $k_1, k_2 \in\{0, 1, 2, 3\}$ such that $k_1+k_2\leq 3$.
  Considering (\ref{eq_ns_re_interchange_diff_int}) with (\ref{eq_ns_proof_assumption3}) and (\ref{eq_ns_proof_assumption8}), we can take $f_{B, 1}$ in [RE](iii) as
  \[
    f_{B,1}(u) = C'\Bigl(|u|^{-(1-2\beta)} + 1\Bigr)
  \]
  for some large constant $C'>0$.

  To find $f_{B, 2}$ in [RE](iv), we need more sophisticated evaluations.
  Suppose a constant $\epsilon_0>0$ satisfies $-(1-2\beta+\epsilon_0)>-1$.
  Because $g(u; \theta^*)$ is bounded from above by $f_{B,1}(u)=C'(|u|^{-(1-2\beta)} 1_{\{0<|u|<1\}}(u) + 1)$,
  it is sufficient to show that
  \begin{equation}\label{eq_obj_fB2}
    |\pt^k (\log g(u; \theta))| \lesssim |u|^{-\epsilon_0} 1_{\{0<|u|\leq 2^{-1}\}} + 1, \quad 0<|u|\leq r
  \end{equation} for $k=0, 1, 2, 3$ to obtain [RE](iv).
  For $k=0$, (\ref{eq_obj_fB2}) is obvious because $1\leq g(u; \theta)\lesssim f_{B,1}(u)$. We will deal with $k=1, 2, 3$.
  Now each component of
  \(
  \pt^k (\log g(u; \theta)) = \pt^k (\log (1 + ap(u; \tau_1, \tau_2)))
  \)
  is in the linear span of
  \[
    \Biggl\{
    \frac{
    a^{k_0} \prod_{l=1}^{L} \partial_{\tau_1}^{k_{l_1}}\partial_{\tau_2}^{k_{l_2}}
    p(u; \tau_1, \tau_2)
    }{
    (1 + ap(u; \tau_1, \tau_2))^{K}
    }
    ; 0\leq k_{l_1}+k_{l_2} \leq 3, k_{l_1}, k_{l_2}\in \bbZ_{\geq 0},
    l=1, \ldots, L,
    \ k_0, L\leq K\leq 3,
    \ k_0, K, L\in\bbZ_{\geq 0}
    \Biggr\}.
  \]
  in the space of functions of $u$.
  Also, as the way obtaining (\ref{eq_ns_proof_assumption8}), we have
  \[
    |\partial_{\tau_1}^{k_{l_1}}\partial_{\tau_2}^{k_{l_2}}p(u; \tau_1, \tau_2)|
    \lesssim
    1 + 1_{\{0<|u|< 1\}}
    \sum_{0\leq k_1+k_2\leq 3}
    \int_0^1 s^{h_{2, 1}(\tau_1)-1}|\log s|^{k_1}
    (s+u)^{h_{2, 2}(\tau_2)-1} |\log(u+s)|^{k_2} 1_{\{0<u+s<1\}}ds
  \]
  by the expression (\ref{eq_ns_proof_assumption3}) and (\ref{eq_ns_re_interchange_diff_int}).
  Therefore, to obtain (\ref{eq_obj_fB2}), it is sufficient to show that
  \begin{equation}\label{eq_obj_fB2_2}
    \frac{I(u; \tau_1, \tau_2) }{1 + ap(u; \tau_1, \tau_2)}
    \lesssim
    |u|^{-\frac{1}{3}\epsilon_0} 1_{\{0<|u|\leq 2^{-1}\}} + 1,
    \quad 0<|u|<1
  \end{equation}
  for $0\leq k_1+k_2 \leq 3, k_1, k_2\in \bbZ_{\geq 0}$,
  where
  \[
    I(u; \tau_1, \tau_2)
    := \int_0^1 s^{h_{2, 1}(\tau_1)-1}|\log s|^{k_1}
    (s+u)^{h_{2, 2}(\tau_2)-1} |\log(u+s)|^{k_2} 1_{\{0<u+s<1\}}ds.
  \]
  Thus, we will evaluate $I(u; \tau_1, \tau_2)$ from above and
  $1 + ap(u; \tau_1, \tau_2)$ from below.
  In the following, we only consider the case $0<u<1$ because the other case $-1<u<0$ goes similarly.

  First, we will consider the case $h_{2, 1}(\tau_1) + h_{2, 2}(\tau_2) \leq 1$.
  Let $0 < \epsilon < \frac{1}{9}\min\{\inf_{\tau_1}h_{2, 1}(\tau_1), \inf_{\tau_2}h_{2, 2}(\tau_2), \epsilon_0\}$ and fix $\theta = (a, \tau_1, \tau_2)$.
  We observe
  \begin{align}\label{al_log_derivatives_nbnsp_ineq_singular}
     & \quad I(u; \tau_1, \tau_2)         \nonumber                                                                                                              \\
     & \leq \int_0^1 s^{h_{2, 1}(\tau_1)-1-\epsilon}
    (s+u)^{h_{2, 2}(\tau_2)-1-\epsilon} 1_{\{0<u+s<1\}}ds              \nonumber                                                                                 \\
     & = u^{h_{2, 1}(\tau_1)+h_{2, 2}(\tau_2)-1-2\epsilon}
    \int_0^{\frac{1}{u}-1}
    t^{h_{2, 1}(\tau_1)-1-\epsilon}
    (t+1)^{h_{2, 2}(\tau_2)-1-\epsilon}dt
    \quad (\because \text{change of the variable as } s=ut)        \nonumber                                                                                     \\
     & \leq u^{h_{2, 1}(\tau_1)+h_{2, 2}(\tau_2)-1-2\epsilon}
    \Bigl(\int_0^1 + \int_1^{(\frac{1}{u}-1)\vee 1}\Bigr)
    t^{h_{2, 1}(\tau_1)-1-\epsilon}
    (t+1)^{h_{2, 2}(\tau_2)-1-\epsilon}dt           \nonumber                                                                                                    \\
     & \leq u^{h_{2, 1}(\tau_1)+h_{2, 2}(\tau_2)-1-2\epsilon}
    \Bigl\{\int_0^1 t^{h_{2, 1}(\tau_1)-1-\epsilon}dt
    + \int_1^{(\frac{1}{u}-1)\vee 1} t^{h_{2, 1}(\tau_1) + h_{2, 2}(\tau_2)-2-2\epsilon} dt \Bigr\}  \quad (\because h_{22}(\tau_2) - 1 - \epsilon < 0)\nonumber \\
     & \leq u^{h_{2, 1}(\tau_1)+h_{2, 2}(\tau_2)-1-2\epsilon}
    \Bigl\{\int_0^1 t^{h_{2, 1}(\tau_1)-1-\epsilon}dt
    + \int_1^{(\frac{1}{u}-1)\vee 1} t^{-1} dt \Bigr\}  \quad (\because h_{2, 1}(\tau_1) + h_{2, 2}(\tau_2) - 2 - 2\epsilon < -1)    \nonumber                   \\
     & \leq u^{h_{2, 1}(\tau_1)+h_{2, 2}(\tau_2)-1-2\epsilon}
    \Bigl\{
    (h_{2, 1}(\tau_1) - \epsilon)^{-1} + |\log((u^{-1}-1)\vee 1)|
    \Bigr\}       \nonumber                                                                                                                                      \\
     & \lesssim u^{h_{2, 1}(\tau_1)+h_{2, 2}(\tau_2)-1-3\epsilon}
    \quad 0<u<1.
    \quad (\because h_{2, 1}(\tau_1) - \epsilon > 2^{-1} \inf_{\tau_1}h_{2, 1}(\tau_1)>0)
  \end{align}
  Also, we have
  \begin{align*}
    p(u; \tau_1, \tau_2)
     & \gtrsim \int_0^1 s^{h_{2, 1}(\tau_1) - 1} (s+u)^{h_{2, 2}(\tau_2) - 1} 1_{\{0<u+s<1\}} ds \\
     & = u^{h_{2, 1}(\tau_1)+h_{2, 2}(\tau_2)-1}
    \int_0^{\frac{1}{u}-1} t^{h_{2, 1}(\tau_1)-1}(t+1)^{h_{2, 2}(\tau_2)-1}dt                    \\
     & \geq u^{h_{2, 1}(\tau_1)+h_{2, 2}(\tau_2)-1}
    \int_0^1 t^{h_{2, 1}(\tau_1)-1}(t+1)^{h_{2, 2}(\tau_2)-1}dt
    \quad (\because 0<u<2^{-1})                                                                  \\
     & \geq u^{h_{2, 1}(\tau_1)+h_{2, 2}(\tau_2)-1}
    \int_0^1 (t+1)^{h_{2, 1}(\tau_1)+h_{2, 2}(\tau_2)-2}dt
    \quad (\because h_{2, 1}(\tau_1)-1 < 0)                                                      \\
     & \geq u^{h_{2, 1}(\tau_1)+h_{2, 2}(\tau_2)-1}
    \int_0^1 (t+1)^{-2}dt                                                                        \\
     & \gtrsim u^{h_{2, 1}(\tau_1)+h_{2, 2}(\tau_2)-1}, \quad 0<u<2^{-1}.
  \end{align*}
  Therefore, we have $I(u; \tau_1, \tau_2)\lesssim u^{h_{2, 1}(\tau_1)+h_{2, 2}(\tau_2)-1-3^{-1}\epsilon}$ for $0<u<1$ and $1 + ap(u; \tau_1, \tau_2)\gtrsim u^{h_{2, 1}(\tau_1)+h_{2, 2}(\tau_2)-1}$ for $0<u<2^{-1}$, so that we have
  \begin{equation}\label{eq_ns_I_case1}
    \frac{I(u; \tau_1, \tau_2) }{1 + ap(u; \tau_1, \tau_2)}
    \lesssim
    |u|^{-\frac{1}{3}\epsilon_0} 1_{\{0<|u|\leq 2^{-1}\}} + 1,
    \quad 0<|u|<1.
  \end{equation}
  Especially we have (\ref{eq_obj_fB2_2}).

  Next, we will consider the other case $h_{2, 1}(\tau_1) + h_{2,2}(\tau_2) > 1$.
  For $0<u<1$, we observe
  \begin{align}\label{al_log_derivatives_nbnsp_ineq_nonsingular1}
     & \quad I(u; \tau_1, \tau_2) \nonumber                                                       \\
     & = u^{h_{2, 1}(\tau_1)+h_{2, 2}(\tau_2)-1}
    \int_0^{\frac{1}{u}-1}
    t^{h_{2, 1}(\tau_1)-1}|\log (ut)|^{k_1}
    (t+1)^{h_{2, 2}(\tau_2)-1}|\log (u(t+1))|^{k_2}dt          \nonumber                          \\
     & \leq
    \int_0^1 t^{h_{2, 1}(\tau_1)-1}|\log (u) +\log(t)|^{k_1}|\log(u)+\log(t+1)|^{k_2}dt \nonumber \\
     & \qquad + u^{h_{2, 1}(\tau_1)+h_{2, 2}(\tau_2)-1} \int_1^{(\frac{1}{u}-1)\vee 1}
    t^{h_{2, 1}(\tau_1)-1}|\log (u) +\log(t)|^{k_1}
    (t+1)^{h_{2, 2}(\tau_2)-1}|\log(u)+\log(t+1)|^{k_2}dt          \nonumber                      \\
     & \lesssim
    (\log(u^{-1}+2))^3
    \int_0^1 t^{\inf_{\tau_1}h_{2, 1}(\tau_1)-1}(\log(t^{-1}+2))^{k_1} dt          \nonumber      \\
     & \qquad + u^{h_{2, 1}(\tau_1)+h_{2, 2}(\tau_2)-1} (\log(u^{-1}+2))^3
    \int_1^{(\frac{1}{u}-1)\vee 1}
    t^{h_{2, 1}(\tau_1)-1}(t+1)^{h_{2, 2}(\tau_2)-1}dt                         \nonumber          \\
     & \qquad\quad \Bigl(
    \because |\log(u)|, |\log((u^{-1}-1)\vee 1)| \leq \log(u^{-1}+2)\text{ and } 1\leq \log(u^{-1}+2) \text{ , for $0<u<1$}
    \Bigr)     \nonumber                                                                          \\
     & \lesssim (\log(u^{-1}+2))^3
    \Biggl(
    1 + u^{h_{2, 1}(\tau_1)+h_{2, 2}(\tau_2)-1}
    \int_1^{(\frac{1}{u}-1)\vee 1}
    t^{h_{2, 1}(\tau_1)-1}(t+1)^{h_{2, 2}(\tau_2)-1}dt
    \Biggr).
  \end{align}
  Here we have
  \begin{align}\label{al_log_derivatives_nbnsp_ineq_nonsingular2}
     & \quad u^{h_{2, 1}(\tau_1)+h_{2, 2}(\tau_2)-1}
    \int_1^{(\frac{1}{u}-1)\vee 1}
    t^{h_{2, 1}(\tau_1)-1}(t+1)^{h_{2, 2}(\tau_2)-1}dt   \nonumber       \\
     & \lesssim u^{h_{2, 1}(\tau_1)+h_{2, 2}(\tau_2)-1}
    \int_1^{(\frac{1}{u}-1)\vee 1}
    t^{h_{2, 1}(\tau_1)+h_{2, 2}(\tau_2)-2}dt
    \quad (\because t\geq 1, \alpha>-1 \Rightarrow (t+1)^{\alpha}\leq (2^{\alpha\vee 0})t^{\alpha})\nonumber
    \\
     & \leq u^{h_{2, 1}(\tau_1)+h_{2, 2}(\tau_2)-1}
    \int_1^{(\frac{1}{u}-1)\vee 1}
    t^{h_{2, 1}(\tau_1)+h_{2, 2}(\tau_2)-2+6^{-1}\epsilon_0}dt \nonumber \\
     & \leq u^{-6^{-1}\epsilon_0}
    \frac{1}{h_{2, 1}(\tau_1)+h_{2, 2}(\tau_2)-1+6^{-1}\epsilon_0}
    \leq 6\epsilon_0^{-1}u^{-6^{-1}\epsilon_0}
    \quad (\because  h_{2, 1}(\tau_1) + h_{2, 2}(\tau_2) > 1) \nonumber  \\
     & \lesssim u^{-6^{-1}\epsilon_0}
  \end{align}
  Therefore, we obtain $I(u; \tau_1, \tau_2)\lesssim (\log(u^{-1}+2))^3 u^{-6^{-1}\epsilon_0}\lesssim u^{-3^{-1}\epsilon_0}$ for $0<u<1$. Together with the fact that we obviously have $1+ap(u; \tau_1, \tau_2)\geq 1$,
  we derive (\ref{eq_obj_fB2_2}) as well.
  Consequently, we can take $f_{B, 2}(u) = C''(|u|^{-\epsilon_0} + 1)$ for some large constant $C''>0$ and we have [RE](iv).
\end{proof}
\halflineskip

We also have parameter-dependent bounds of log-derivatives of the cross-correlation function.
\begin{lemma}\label{lem_log_ccf_derivative_estimate_parameter_wise}
  Let $\theta = (a, \tau_1, \tau_2)\in\Theta \text{[resp. $\in \ol{\Theta}$]}$ and $k=1, 2, 3$ [resp. $k=0$].
  Then, for any small $\epsilon'>0$, we have
  \begin{equation*}
    |\pt^k (\log g(u; \theta))| \leq C(\theta)(|u|^{-\epsilon'} 1_{\{0<|u|\leq 2^{-1}\}} + 1), \quad 0 < |u|\leq r
  \end{equation*}
  for some constant $C(\theta)>0$ not depending on $u$.
  Especially,
  we have $\pt^k (\log g(\cdot; \theta))\in \bigcap_{p\geq 1}L^p([-r, r])$.
\end{lemma}
\begin{proof}
  We will use the same notations as in the proof of Lemma \ref{lem_NS_RE}.
  For $k=0$, the estimate is obvious because $1\leq g(u; \theta)\lesssim f_{B,1}(u)$.
  Let $k=1, 2, 3$.
  For the case $h_{12}(\tau_1) + h_{22}(\tau_2) \leq 1$, we obtain the result from (\ref{eq_ns_I_case1}).

  For the other case $h_{12}(\tau_1) + h_{22}(\tau_2) > 1$,
  by (\ref{al_log_derivatives_nbnsp_ineq_nonsingular1}) and up to the first inequality in (\ref{al_log_derivatives_nbnsp_ineq_nonsingular2}), we have
  \begin{align*}
    I(u; \tau_1, \tau_2)
     & \lesssim
    (\log(u^{-1}+2))^3 u^{h_{2, 1}(\tau_1)+h_{2, 2}(\tau_2)-1}
    \int_1^{(\frac{1}{u}-1)\vee 1}
    t^{h_{2, 1}(\tau_1)+h_{2, 2}(\tau_2)-2}dt                                                      \\
     & \lesssim (\log(u^{-1}+2))^3 \frac{1}{h_{2, 1}(\tau_1) + h_{2, 2}(\tau_2) - 1}, \quad 0<u<1,
  \end{align*}
  hence we obtain the result. For $-1<u<0$, we have similar estimates.
\end{proof}
\halflineskip

Next, we will deal with the condition [MI] under [NS].
The $\alpha$-mixing coefficient of the Neyman-Scott process is evaluated by the tail probabilities of the dispersal kernels.
This lemma is proven by similar way as the proof of Lemma 1 in Prokešová \& Jensen~\cite{Palm2013}.
\begin{lemma}\label{lem_mixing_NS}
  For all $c_1\geq 0$ and $m\geq 2r+2$,
  \[
    \alpha_{c_1, \infty}^{N^S}(m; r)
    \leq
    8\lambda c_1(m+1+2r) \sum_{i=1}^2\sigma_i \int_{|z|\geq \frac{m}{2}-2r} dz f_i(z; \tau_i).
  \]
  Especially, together with the condition (\ref{eq_ns_cond_2}) in [NS](iii),
  we have
  \[
    \sum_{m=1}^{\infty}\alpha_{2, \infty}^{N^S}(m; r)^{\frac{\delta}{2+\delta}} < \infty,
  \]
  where $\delta>0$ is the one appearing in [NS](iii).
\end{lemma}
\halflineskip
\begin{proof}
      Suppose that $E_1 = \bigcup_{l\in M_1}C(l)\oplus r, E_2 = \bigcup_{l\in M_2}C(l)\oplus r,
            \#M_1\leq c_1,
            d(M_1, M_2)\geq m, M_1, M_2\subset \bbZ$, and
      \[
            N_i^1 = \sum_{c \in \calc\cap E_1\oplus\frac{m}{2}} \sum_{k = 1}^{M_i(c)} \delta_{c + d_i(c, k)}, \quad
            N_i^2 = \sum_{c \in \calc\cap (E_1\oplus\frac{m}{2})^c} \sum_{k = 1}^{M_i(c)} \delta_{c + d_i(c, k)}
      \]
      for $i=1, 2$.
      Let $C_1$, $C_2$ be arbitrary events from $\sigma(\{N_i\cap E_1\}_{i=1}^2)$, $\sigma(\{N_i\cap B\}_{i=1}^2)$.
      By the same argument as in the proof of Lemma 1 in Prokešová \& Jensen~\cite{Palm2013}, we have
      \begin{align*}
            |P(C_1\cap C_2) - P(C_1)P(C_2)|
             & \leq 4(P(\cup_{i=1}^2 \{N_i^1(E_2) \geq 1\}) + P(\cup_{i=1}^2 \{N_i^2(E_1) \geq 1\})) \\
             & \leq 4\sum_{i=1}^2(P(N_i^1(E_2) \geq 1) + P(N_i^2(E_1) \geq 1))                       \\
             & \leq 4\sum_{i=1}^2(E[N_i^1(E_2)]+ E[N_i^2(E_1)]).
      \end{align*}
      We observe that
      \begin{align*}
             & \quad E[N_i^1(E_2)]
            = \lambda \sigma_i \int_{E_1\oplus\frac{m}{2}}dc \int_{E_2} dz f_i(z-c; \tau_i)                          \\
             & \leq \lambda \sigma_i \int_{E_1\oplus\frac{m}{2}}dc \int_{|z|\geq \frac{m}{2}-2r} dz f_i(z; \tau_i)
            \leq \lambda \sigma_i \mathrm{Leb}(E_1\oplus\frac{m}{2}) \int_{|z|\geq \frac{m}{2}-2r} dz f_i(z; \tau_i) \\
             & \leq \lambda \sigma_i c_1(m+1+2r) \int_{|z|\geq \frac{m}{2}-2r} dz f_i(z; \tau_i).
      \end{align*}
      for $i=1, 2$
      because $d(E_1\oplus\frac{m}{2}, E_2)\geq \frac{m}{2} - 2r$.
      Similarly,
      \begin{align*}
             & \quad E[N_i^2(E_1)]
            = \lambda \sigma_i \int_{(E_1\oplus\frac{k}{2})^c}dc \int_{E_1} dz f_i(z-c; \tau_i) \\
             & \leq \lambda \sigma_i \int_{E_1} du \int_{|z|\geq \frac{m}{2}} dz f_i(z; \tau_i)
            \leq \lambda \sigma_i c_1 \int_{|z|\geq \frac{m}{2}} dz f_i(z; \tau_i).
      \end{align*}
      Thus,
      \[
            E[N_i^1(E_2)]+ E[N_i^2(E_1)]
            \leq 2\lambda \sigma_i c_1(m+1+2r) \int_{|z|\geq \frac{m}{2}-2r} dz f_i(z; \tau_i)
      \]
      holds.
      This concludes the proof.
\end{proof}
\halflineskip

\begin{lemma}\label{lem_ns_existence_of_moment}
  Under (\ref{gat_ns_cond_1}), the bivariate Neyman-Scott process $N^S$ has $\lceil 2+\delta \rceil$-th moment and locally finite factorial moment measures up to $(\lceil 2+\delta \rceil, \lceil 2+\delta \rceil)$-th order.
\end{lemma}
\begin{proof}
  For the second assertion, it is sufficient to show that
  \begin{equation}\label{eq_NS_moment_existence_obj}
    E[(N^S_1(A))^L (N^S_2(A))^L] < \infty
  \end{equation}
  for a bounded set $A\in \calb(\bbR)$ and $L=\lceil 2+\delta \rceil$.
  For a fixed configuration of the parent process $\calc$ and $i=1, 2$,
  we have
  \begin{align*}
    E[(N_{i, c}(A))^L|\calc]
     & = E\Bigl[ \Bigl(\sum_{j=1}^{M_i(c)}\delta_{d(c, j)}(A-c) \Bigr)^L|\calc\Bigr] \\
     & = E\Bigl[
      E\Bigl[
        \Bigl(\sum_{j=1}^{M_i(c)}\delta_{d(c, j)}(A-c)\Bigr)^L
        |\calc, M_i(c)\Bigr]
      |\calc\Bigr], \quad c\in\calc.
  \end{align*}
  Here, $\sum_{j=1}^{M_i(c)}\delta_{d(c, j)}(A-c)\sim \mathop{Binomial}\Bigl(M_i(c), \int_{A-c}f_i(u)du\Bigr)$ when conditioned by $M_i(c)$ and $\calc$.
  Thus, together with the fact that $g_i^{(j)}(1-)$ exists for $j=1, \ldots, L$ thanks to the assumption (\ref{gat_ns_cond_1}), we have
  \begin{align*}
    E[N_{i, c}(A)^L|\calc]
     & = E\Bigl[\sum_{j=1}^L S(L, j)
      M_i(c)(M_i(c)-1)\cdots(M_i(c)-j+1) \Bigl(\int_{A-c}f_i(u)du\Bigr)^j
    |\calc \Bigr]                                                       \\
     & = \sum_{j=1}^L S(L, j)g_i^{(j)}(1-)\Bigl(\int_{A-c}f_i(u)du\Bigr)^j \\
     & \lesssim \int_{A-c}f_i(u)du =: \phi_i(c)
    , \quad c\in\calc
  \end{align*}
  where $S(k, j)$ is the Stirling number of second kind.
  We note that $\int_{\bbR}\phi_i(c)dc = \mathrm{Leb}(A) < \infty$ and $\phi_i$ is bounded by $1$ because $f_i$ is a probability density function.
  Therefore, for $i=1, 2$, we have
  \begin{align*}
    E[(N^S_i(A))^L|\calc]
     & = E\Bigl[\Bigl(
      \sum_{c\in\calc} N_{i, c}(A)
    \Bigr)^L|\calc\Bigr]                                                                                             \\
     & = \sum_{c_1,\ldots, c_L\in\calc}
    E\Bigl[\prod_{l=1}^L N_{i, c_l}(A)|\calc\Bigr]                                                                   \\
     & = \sum_{a=1}^L \sum_{\pi\in\calp_a^L}
    \sum_{c_1, \ldots, c_a\in\calc}
    1_{\{c_v\neq c_w, v, w \in\{1, \ldots, a\}\}}
    E\Bigl[\prod_{l=1}^a N_{i, c_l}(A)^{|\pi^{-1}(l)|}|\calc\Bigr]                                                      \\
     & = \sum_{a=1}^L \sum_{\pi\in\calp_a^L}
    \sum_{c_1, \ldots, c_a\in\calc}
    1_{\{c_v\neq c_w, v, w \in\{1, \ldots, a\}\}}
    \prod_{l=1}^a E\Bigl[N_{i, c_l}(A)^{|\pi^{-1}(l)|}|\calc\Bigr]  \quad (\because \mathrm{conditional\ independence}) \\
     & \lesssim \sum_{a=1}^L \sum_{\pi\in\calp_a^L}
    \sum_{c_1, \ldots, c_a\in\calc}
    1_{\{c_v\neq c_w, v, w \in\{1, \ldots, a\}\}}
    \prod_{l=1}^a \phi_i(c_l)                                                                                        \\
  \end{align*}
  Finally, we obtain
  \begin{align}\label{al_existence_cross_moment1}
     & \quad E[(N^S_1(A))^L (N^S_2(A))^L]      \nonumber                                                             \\
     & = E[E[(N^S_1(A))^L (N^S_2(A))^L|\calc]]    \nonumber                                                          \\
     & = E[E[(N^S_1(A))^L|\calc] E[(N^S_2(A))^L|\calc]] \quad (\because \mathrm{conditional\ independence})\nonumber \\
     & = E[E[(N^S_1(A))^L|\calc] E[(N^S_2(A))^L|\calc]] \nonumber                                                    \\
     & \lesssim E\Bigl[
      \prod_{i=1}^2
      \sum_{a_i=1}^L \sum_{\pi_i\in\calp_a^L}
      \sum_{c_{i, 1}, \ldots, c_{i, a_i}\in\calc}
      1_{\{c_v\neq c_w, v, w \in\{1, \ldots, a_i\}\}}
      \prod_{l=1}^{a_i} \phi_i(c_l)
      \Bigr].
  \end{align}
  Because $\phi_i$ is bounded by $1$ and integrable on $\bbR$, it is sufficient to show that
  \begin{equation}\label{eq_existence_cross_moment1}
    E\Bigl[
      \sum_{c_{1}, \ldots, c_{b}\in\calc}
      1_{\{c_v\neq c_w, v, w \in\{1, \ldots, b\}\}}
      \prod_{l=1}^{b} \phi_{i_l}(c_l)
      \Bigr] < \infty
  \end{equation}
  for $b\in\bbZ_{\geq 1}$ and $i_l\in\{1, 2\}, l=1, \ldots, b$
  in order to show the finiteness of the rightmost side of (\ref{al_existence_cross_moment1}).
  However, the $b$-th order factorial moment measure of the stationary Poisson process on $\bbR$ is product the $b$-fold product measure of the Lebesgue measure with itself up to constant (see p.72 of Daley \& Vere-Jones~\cite{DaleyandVereJones2}).
  Considering the fact that $\phi_i$ is bounded by $1$ and integrable on $\bbR$, we derive (\ref{eq_existence_cross_moment1}).
  Consequently, we have the second assertion.

  The first assertion can be shown in the similar way.
\end{proof}
\halflineskip

\begin{lemma}\label{lem_singular_integral}
  Suppose that $M\in\bbZ_{\geq 1}$, $f\in L^q(\bbR)$ for some $q>1$,
  $h\in \bigcap_{p\geq1}L^p(\bbR)$, and $h$ has a compact support. Then,
  \begin{equation*}
    \sup_{x_1, \ldots, x_M\in\bbR} \int_{\bbR}
    \Bigl|f(y)\prod_{m=1}^M h(x_m + y)\Bigr|dy < \infty.
  \end{equation*}
\end{lemma}
\begin{proof}
  By Hölder's inequality, we obtain the result.
\end{proof}
\halflineskip

\begin{lemma}\label{lem_high_moment_NS}
  Suppose $h\in \bigcap_{p\geq1}L^p(\bbR)$, $h\geq 0$, and $h$ has a compact support.
  Then,
  for all $L\in\bbZ$ such that $1\leq L \leq \lceil 2+\delta \rceil$, we have
  \begin{equation*}
    E\Bigl[\Bigl(
      \sum_{\substack{x\in N_1^S, y\in N_2^S  \\ x\in (0, 1]}}
      h(y-x)
      \Bigr)^L\Bigr] < \infty.
  \end{equation*}
\end{lemma}
\begin{proof}
  Suppose
  $M_{(a, b)}$ is the moment measure,
  $M_{[a, b]}$ is the factorial moment measure, and
  $C_{[a, b]}$ is the factorial cumulant measure of $N^S$, where $a, b\in\bbZ_{\geq 0}$.
  By Lemma \ref{lem_ns_existence_of_moment}, the process $N^S$ has locally finite factorial moment measures up to $(\lceil 2+\delta \rceil, \lceil 2+\delta \rceil)$-th order.
  By the relation (\ref{eq_factorial_and_normal_moment}) and (\ref{fmomt2fcum}), we have
  \begin{align}\label{al_expansion_log_moment}
     & \quad E\Bigl[\Bigl|\sum_{\substack{x\in N_1^S, y\in N_2^S           \\ x\in (0, 1]}}h(y-x)\Bigr|^L\Bigr] \nonumber  \\
     & = \int_{\bbR^{2L}}
    \prod_{l=1}^L 1_{\{x_{l,1}\in (0, 1]\} }h(x_{2,l} - x_{1,l})
    M_{(L, L)}(dx_{1,1}\ldots dx_{2,L}) \nonumber                          \\
     & \lesssim \sum_{m_1=1}^L \sum_{m_2=1}^L
    \sum_{\pi_1\in\calp_{m_1}^L} \sum_{\pi_2\in\calp_{m_2}^L}
    \sum_{m=1}^{m_1+m_2} \sum_{\rho\in\calp_{m}^{m_1\sqcup m_2}} \nonumber \\
     & \qquad \int_{\bbR^{m_1+m_2}} \prod_{l=1}^L
    1_{\{y_{1, \pi_1(l)}\in (0, 1]\} }h(y_{2, \pi_2(l)} - y_{1, \pi_1(l)})
    \prod_{k=1}^m |C_{[|\rho^{-1}(k)_1|, |\rho^{-1}(k)_2||]}|\Bigl(\prod_{(i ,j)\in \rho^{-1}(k)}dy_{i,j}\Bigr).
  \end{align}

  We will evaluate each summand in (\ref{al_expansion_log_moment}).
  Let $m_1, m_2=1, \ldots, L$, $\pi_1\in\calp_{m_1}^L$, $\pi_2\in\calp_{m_2}^L$, $m=1, \ldots, m_1+m_2$, and $\rho\in\calp_{m}^{m_1\sqcup m_2}$.
  By the expression (\ref{eq_cumulant_NS}) of the factorial cumulant densities of the Neyman-Scott process, we have
  \begin{align}\label{al_integral_raw}
     & \quad \int_{\bbR^{m_1+m_2}} \Bigl(\prod_{l=1}^L
    1_{\{y_{1, \pi_1(l)}\in (0, 1]\} }h(y_{2, \pi_2(l)} - y_{1, \pi_1(l)})\Bigr)
    \prod_{k=1}^m |C_{[|\rho^{-1}(k)_1|, |\rho^{-1}(k)_2||]}|\Bigl(\prod_{(i ,j)\in \rho^{-1}(k)}dy_{i,j}\Bigr) \nonumber      \\
     & \lesssim \int_{\bbR^{m_1+m_2}} \Bigl(\prod_{l=1}^L
    1_{\{y_{1, \pi_1(l)}\in (0, 1]\} }h(y_{2, \pi_2(l)} - y_{1, \pi_1(l)})\Bigr)
    \prod_{k=1}^m \Bigl(\int_{\bbR}\prod_{(i ,j)\in \rho^{-1}(k)}f_i(y_{i,j}-c_k)dc_k\Bigr)dy_{1,1}\cdots dy_{2,m_2} \nonumber \\
     & \leq \int_{K^{m_1+m_2}} \Bigl(\prod_{l=1}^L
    h(y_{2,\pi_2(l)} - y_{1,\pi_1(l)})\Bigr)
    \prod_{k=1}^m \Bigl(\int_{\bbR}\prod_{(i ,j)\in \rho^{-1}(k)}f_i(y_{i,j}-c_k)dc_k\Bigr)dy_{1,1}\cdots dy_{2,m_2},
  \end{align}
  where $K=\mathop{\text{supp}}(h) \oplus 2$.

  The key idea is basically considering integrating over all $ y_{i,j} $ for each group divided by the partition $\rho$, but when faced with the $ dc_k $-integral over a non-compact region,
  we eliminate it by using the fact that $ f_i $ is a probability density function.


  We sort the order of integration by the partition $\rho$.
  That is, we see (\ref{al_integral_raw}) as
  \begin{equation}\label{eq_integral_transformed1}
    \int_{\bbR} \int_{K^{|\rho^{-1}(1)|}} \cdots \int_{\bbR} \int_{K^{|\rho^{-1}(m)|}}
    \Bigl(\prod_{l=1}^L
    h(y_{2, \pi_2(l)} - y_{1, \pi_1(l)})\Bigr)
    \prod_{k=1}^m \prod_{(i ,j)\in \rho^{-1}(k)}f_i(y_{i,j}-c_k)dy_{i,j}dc_k.
  \end{equation}

  First, we focus on the integrals associated with $\rho^{-1}(1)$, and
  eliminate the integral with respect to $\{y_{i, j}\}_{(i, j)\in \rho^{-1}(1)}$ sequentially by bounding from above by some constants until only one $y_{i, j}$ remains to be eliminated.
  This is possible because the integral
  \begin{equation*}
    \int_{K}f_i(y_{i,j}-c_1)\prod_{a=1}^A h(\pm(y_a-y_{i,j}))dy_{i,j},
    \quad A\in\bbZ_{\geq 0}, y_a\in\bbR, a=1\ldots, A
  \end{equation*}
  can be bounded by some constant not depending on $\{y_a\}_{a=1}^A$ and $c_1$ thanks to Lemma \ref{lem_singular_integral} when $A\geq 1$ and the fact that $f_i$ is a probability density when $A=0$.
  Now, we have only one $dy_{i, j}$-integral and the $dc_1$-integral in the block associated with $\rho^{-1}(1)$.
  We can eliminate $dc_1$ integral using $\int_{\bbR}f_i(y_{i, j}-c_1)dc_{1} = 1$ because we do not have any other components that depend on $c_1$ besides $f_{i}(y_{i,j}-c_1)$.
  Finally, we only have the following form of integral
  \[
    \int_{K}\prod_{a=1}^A h(\pm(y_a-y_{i,j}))dy_{i,j}, y_a\in\bbR, a=1\ldots, A,
  \]
  that also can be bounded by some constant thanks to the Hölder's inequality.

  After the evaluation step above for $k=1$, (\ref{eq_integral_transformed1}) is bounded up to constant by
  \begin{equation}\label{eq_integral_transformed2}
    \int_{\bbR} \int_{K^{|\rho^{-1}(2)|}} \cdots \int_{\bbR} \int_{K^{|\rho^{-1}(m)|}}
    \Bigl(\prod_{l\in\call}
    h(y_{2, \pi_2(l)} - y_{1, \pi_1(l)})\Bigr)
    \prod_{k=2}^m \prod_{(i ,j)\in \rho^{-1}(k)}f_i(y_{i,j}-c_k)dy_{i,j}dc_k,
  \end{equation}
  where $\call$ is some subset of $\{1, \ldots, L\}$ such that $\rho(\{(i, \pi_i(l)); l\in\call, i=1, 2\})\subset \{2, \ldots, m\}$.
  By repeating similar argument, the quantity (\ref{eq_integral_transformed1}) is bounded up to constant by
  \begin{equation}\label{eq_integral_transformed3}
    \int_{\bbR} \int_{K^{|\rho^{-1}(k')|}} \cdots \int_{\bbR} \int_{K^{|\rho^{-1}(m)|}}
    \Bigl(\prod_{l\in\call}
    h(y_{2, \pi_2(l)} - y_{1, \pi_1(l)})\Bigr)
    \prod_{k=k'}^m \prod_{(i ,j)\in \rho^{-1}(k)}f_i(y_{i,j}-c_k)dy_{i,j}dc_k,
  \end{equation}
  where $\call$ is some subset of $\{1, \ldots, L\}$ such that $\rho(\{(i, \pi_i(l)); l\in\call, i=1, 2\})\subset \{k', \ldots, m\}$.
  By repeating this evaluation up to $k=m$, we obtain the desired result.
\end{proof}
\halflineskip

\begin{lemma}\label{lem_high_moment_NS_2}
  Suppose $h\in \bigcap_{p\geq1}L^p(\bbR)$, $h\geq 0$, and $h$ has a compact support.
  Then, for all $L\in\bbZ$ such that $1\leq L \leq \lceil 2+\delta \rceil$, we have
  \begin{equation*}
    E\Bigl[\Bigl|
      \sum_{\substack{x\in N_1^S, y\in N_2^B  \\ x\in (0, 1]}}
      h(y-x)
      \Bigr|^L\Bigr], \quad
    E\Bigl[\Bigl|
      \sum_{\substack{x\in N_2^S, y\in N_1^B  \\ x\in (0, 1]}}
      h(y-x)
      \Bigr|^L\Bigr] < \infty.
  \end{equation*}
\end{lemma}
\begin{proof}
  We only consider the former because a similar argument can be applied to the latter.
  The proof proceeds exactly in the same manner as Proposition \ref{lem_high_moment_NS} up until (5.20),
  and then, we have
  \begin{align}\label{al_integral_raw_noise}
     & \quad \int_{\bbR^{m_1+m_2}} \Bigl(\prod_{l=1}^L
    1_{\{y_{\pi_1(l)}\in (0, 1]\} }h(y_{2, \pi_2(l)} - y_{1, \pi_1(l)})\Bigr)
    \prod_{k=1}^m |C_{[|\rho^{-1}(k)_1|, |\rho^{-1}(k)_2||]}|\Bigl(\prod_{(i ,j)\in \rho^{-1}(k)}dy_{i,j}\Bigr) \nonumber \\
     & \lesssim \int_{\bbR^{m_1+m_2}} \Bigl(\prod_{l=1}^L
    1_{\{y_{\pi_1(l)}\in (0, 1]\} }h(y_{2, \pi_2(l)} - y_{1, \pi_1(l)})\Bigr)     \nonumber                               \\
     & \qquad \prod_{k=1}^m \Bigl(
    |C^{N_1^S}_{[|\rho^{-1}(k)_1|]}|(dy_{1,1}\times \cdots \times dy_{1,m_1})
    + |C^{N_2^B}_{[|\rho^{-1}(k)_2|]}|(dy_{2,1}\times \cdots \times dy_{2,m_2})
    \Bigr)
  \end{align}
  using the independence $N_1^S\indep N_2^B$,
  where $C^{N_1^S}$ and $C^{N_2^B}$ are the factorial cumulant measures of $N_1^S$ and $N_2^B$, respectively.
  Since the density of $C^{N_2^B}$ is locally bounded because of the assumption [NS](iv),
  we apply a similar calculation as in Proposition \ref{lem_high_moment_NS} in this case too.
\end{proof}
\halflineskip

\begin{lemma}\label{lem_high_moment_NS_3}
  Suppose $h\in \bigcap_{p\geq1}L^p(\bbR)$, $h\geq 0$, and $h$ has a compact support.
  Then, for all $L\in\bbZ$ such that $1\leq L \leq \lceil 2+\delta \rceil$, we have
  \begin{equation*}
    E\Bigl[\Bigl|
      \sum_{\substack{x\in N_1^B, y\in N_2^B  \\ x\in (0, 1]}}
      h(y-x)
      \Bigr|^L\Bigr] < \infty.
  \end{equation*}
\end{lemma}
\begin{proof}
  By the same argument as around (\ref{al_integral_raw_noise}) and the assumption [NS](iv), we can obtain the result.
\end{proof}
\halflineskip

\begin{proposition}\label{lem_NS_MI}
  Under [NS](ii), (iii) and (iv), the noisy bivariate Neyman-Scott process model satisfies the condition [MI].
\end{proposition}
\begin{proof}
  Let $\delta>0$ be the one given in [NS](iii).
  The $\alpha$-mixing condition $\sum_{m=1}^{\infty}\alpha_{2, \infty}^N(m; r)^{\frac{\delta}{2+\delta}} < \infty$ is valid by Lemma \ref{lem_mixing_NS}.
  The first moment condition $\|N_i(C(0))\|_{2+\delta} < \infty, i=1, 2$ is a consequence of (\ref{gat_ns_cond_1}) and (\ref{gat_ns_cond_2}) (see Remark \ref{rem_mom_NS}).
  Finally, we will check the condition $\Bigl\|F_2^N\Bigl(|\pt^j\log g(\cdot; \theta)|; C(0)\Bigr)\Bigr\|_{2+\delta} < \infty, j=0, 1, 2, 3, \ \theta\in\Theta$.
  Fix $\theta\in\Theta$.
  Thanks to Lemma \ref{lem_log_ccf_derivative_estimate_parameter_wise}, $1_{|\cdot|\leq r}|\pt^j\log g(\cdot; \theta)|$ is in $\bigcap_{p\geq1}L^p(\bbR)$ and has a compact support.
  Therefore, by Lemmas \ref{lem_high_moment_NS}, \ref{lem_high_moment_NS_2} and \ref{lem_high_moment_NS_3}, we obtain
  \begin{align*}
     & \qquad \Bigl\|F_2^N\Bigl(|\pt^j\log g(\cdot; \theta)|; C(0)\Bigr)\Bigr\|_{2+\delta} \\
     & \leq
    \Bigl\|\sum_{\substack{x\in N_1^S, y\in N_2^S                                          \\ x\in (0, 1]}}
    1_{|y-x|\leq r}|\pt^j\log g(y-x; \theta)|
    \Bigr\|_{2+\delta}
    + \Bigl\|\sum_{\substack{x\in N_1^S, y\in N_2^B                                        \\ x\in (0, 1]}}
    1_{|y-x|\leq r}|\pt^j\log g(y-x; \theta)|
    \Bigr\|_{2+\delta}                                                                     \\
     & \qquad +
    \Bigl\|\sum_{\substack{x\in N_1^B, y\in N_2^S                                          \\ x\in (0, 1]}}
    1_{|y-x|\leq r}|\pt^j\log g(y-x; \theta)|
    \Bigr\|_{2+\delta}
    + \Bigl\|\sum_{\substack{x\in N_1^B, y\in N_2^B                                        \\ x\in (0, 1]}}
    1_{|y-x|\leq r}|\pt^j\log g(y-x; \theta)|
    \Bigr\|_{2+\delta}  <\infty.
  \end{align*}
\end{proof}
\halflineskip

\begin{remark}
  \rm
  Evaluating moments $\Bigl\|F_2^N\Bigl(|\pt^j\log g(\cdot; \theta)|; C(0)\Bigr)\Bigr\|_{2+\delta}$ in Proposition \ref{lem_NS_MI} corresponds to checking that the assumption (13) in p.398 in Prokešová \& Jensen~\cite{Palm2013} is satisfied, for instance.
  They assert that such assumption can be verified if the log derivatives of the moment density is bounded.
  However, in our case, the log derivatives of the moment density function $\pt^j\log g(\cdot; \theta)$ is not necessarily bounded so that we need the direct evaluations in the proof of Lemma \ref{lem_high_moment_NS}.
\end{remark}

\subsection{Proof of Theorems \ref{thm_NBNSP_G} and \ref{thm_NBNSP_E}}\label{subsec_prf_NBNSP}
Thanks to Theorem \ref{thm_noisy_NS_asymptotics}, the only thing we have to do is check the conditions [ID] and [ID2] for the models.

\begin{proposition}\label{prop_id_ns_gamma}
  The gamma kernel model satisfies the condition [ID].
\end{proposition}
\begin{proof}
  Let $\theta=(a, l_1, l_2, \alpha_1, \alpha_2), \theta^*=(a^*, l_1^*, l_2^*, \alpha_1^*, \alpha_2^*)\in\overline{\Theta}$ and
  assume
  \[
    g(\cdot; \theta) = g(\cdot; \theta^*) \quad \text{a.e. on $[-r, r]$}.
  \]
  By the analyticity of $g(\cdot; \theta)$ and $g(\cdot; \theta^*)$ on $\{z\in\bbC; \mathrm{Re}(z) \neq 0\}$ , we can assume
  \begin{equation}\label{eq_NS_id_proof_1}
    g(u; \theta) = g(u; \theta^*), \quad u\in\bbR\setminus \{0\}.
  \end{equation}
  because of the identity theorem for analytic functions.
  Then, by considering the limits $u\to \pm \infty$, we have
  $(l_1, l_2, \alpha_1, \alpha_2) = (l_1^*, l_2^*, \alpha_1^*, \alpha_2^*)$
  thanks to the asymptotic behavior of the bilateral gamma distribution shown in p.2483 of K{\"u}chler \& Tappe~\cite{kuchler2008shapes}.
  Thus, we also have $a = a^*$ by (\ref{eq_NS_id_proof_1}) and obtain $\theta = \theta^*$.
\end{proof}
\halflineskip

We will have to investigate the orders of Laplace transforms of regularly varying functions near zero to prove the positivity of the observed information.
For this sake, we introduce Lemma \ref{lem_abelian} below.

\begin{lemma}\label{lem_abelian}
  Suppose $L: \bbR_{>0} \to \bbR_{>0}$ is a slowly varying function at $0$, a continuous function $f: \bbR_{>0} \to \bbR_{>0}$ satisfies $f(t)\sim t^{\rho-1} L(t)$ as $t\to0$ for some $\rho>0$, and
  the Laplace transform $\call(f)(u) = \int_0^{\infty} f(t)e^{-ut}dt$ exists for all $u>0$.
  Then, we have \[\call(f)(u) \sim \Gamma(\rho)u^{-\rho} L(1/u)\] as $u\to\infty$.
\end{lemma}
\begin{proof}
  Because of Theorem XIII.5.3 and Theorem XIII.5.4 in Feller~\cite{feller1991introduction},
  we only have to prove
  \[
    F(t) = \int_0^t f(s)ds \sim \frac{1}{\rho} t^{\rho} L(t)
  \]
  as $t\to 0$.

  Let $\tilde{f}(s) = f(1/s) / s^2, s>0$.
  Then, we have $F(t) = \int_{1/t}^{\infty} \tilde{f}(s') ds'$ by the change of variable $s=1/s'$.
  By the assumption, we have $\tilde{f}(s) \sim s^{-(\rho+1)}L(1/s)$ as $s\to\infty$
  so that $\tilde{f}(s)$ is regularly varying with exponent $-(\rho+1)$.
  Thus, we can apply Theorem VIII.9.1 (a) in Feller~\cite{feller1991introduction} with $Z=\tilde{f}$ and $p=0$ to obtain
  \[
    \frac{t' \tilde{f}(t')}{\int_t'^{\infty}\tilde{f}(s) ds} \to -(-(\rho+1)+1) = \rho
  \]
  as $t'\to\infty$.
  Substituting $t'=1/t$ and letting $t\to 0$, we derive
  \[
    \rho F(t)
    =\rho\int_{t'}^{\infty}\tilde{f}(s) ds
    \sim
    t' \tilde{f}(t') = \tilde{f}(1/t) / t \sim tf(t)
    =t^{\rho} L(t)
  \]
  as $x\to 0$.
\end{proof}


\begin{proposition}\label{prop_id2_ns_gamma}
  The gamma kernel model satisfies the condition [ID2].
\end{proposition}
\begin{proof}
  It is sufficient to show that, for all $c\in\bbR^5$,
  \[
    (\forall u\in [-r, r]: \ c' \pt g(u; \theta^*) = 0 ) \quad \Rightarrow \quad  c = 0
  \]
  because
  \begin{gather*}
    \Gamma = \lambda_1^*\lambda_2^*\int_{|u|\leq r} \frac{\pt g(u; \theta^*)^{\otimes 2}}{g(u; \theta^*)}du \quad \text{is positive definite} \\
    \Leftrightarrow
    \forall c\in\bbR^5, c\neq 0:\quad
    \int_{|u|\leq r} \frac{c'\pt g(u; \theta^*)^{\otimes 2}c}{g(u; \theta^*)}du  > 0 \\
    \Leftrightarrow
    \forall c\in\bbR^5, c\neq 0, \exists u\in[-r, r]:\quad
    c'\pt g(u; \theta^*)^{\otimes 2}c > 0. \quad (\because \text{continuity of $\pt g(\cdot; \theta^*)$})
  \end{gather*}

  Let $q^+(\theta) = q^+(\alpha_1, \alpha_2, l_1, l_2) = \frac{l_1^{\alpha_1} l_2^{\alpha_2}}
    {(l_1+l_2)^{\alpha_1}\Gamma(\alpha_1)\Gamma(\alpha_2)}$,
  $q_-(\theta) = q^+(\alpha_2, \alpha_1, l_2, l_1)$.
  Then, by (\ref{eq_gamma_ns_ccf}) and (\ref{eq_bigamma_integral_expression}), we have
  \begin{gather*}
    g(u; \theta) = 1 + a q^+(\theta)e^{-l_1u}
    \int_0^{\infty} v^{\alpha_2-1}\Bigl(
    u + \frac{v}{l_1+l_2}
    \Bigr)^{\alpha_1-1}
    e^{-v}dv, \quad u>0, \\
    g(u; \theta) = 1 + a q^-(\theta)e^{-l_2|u|}
    \int_0^{\infty} v^{\alpha_1-1}\Bigl(
    |u| + \frac{v}{l_1+l_2}
    \Bigr)^{\alpha_2-1}
    e^{-v}dv, \quad u<0. \\
  \end{gather*}

  In the following, we will sometimes abbreviate $q^{\pm}(\theta)$ as $q^{\pm}$ for ease of notation.

  For $u>0$,
  we calculate the derivatives as
  \begin{align*}
    \partial_{a}g(u; \theta)
     & = q^+e^{-l_1u}
    \int_0^{\infty} v^{\alpha_2-1}\Bigl(
    u + \frac{v}{l_1+l_2}
    \Bigr)^{\alpha_1-1}e^{-v}dv                          \\
     & = q^+e^{-l_1u} u^{\alpha_1+\alpha_2-1}
    \int_0^{\infty} t^{\alpha_2-1}\Bigl(
    1 + \frac{t}{l_1+l_2}
    \Bigr)^{\alpha_1-1}e^{-ut}dt                         \\
     & =  q^+e^{-l_1u} u^{\alpha_1+\alpha_2-1} p_0^+(u),
  \end{align*}
  \begin{align*}
    \partial_{\alpha_1}g(u; \theta)
     & = a (\partial_{\alpha_1}q^+)e^{-l_1u}
    \int_0^{\infty} v^{\alpha_2-1}\Bigl(
    u + \frac{v}{l_1+l_2}
    \Bigr)^{\alpha_1-1}e^{-v}dv                                              \\
     & \quad + aq^+e^{-l_1u}
    \int_0^{\infty} v^{\alpha_2-1}\Bigl(
    u + \frac{v}{l_1+l_2}
    \Bigr)^{\alpha_1-1}
    \log\Bigl(u+\frac{v}{l_1+l_2}\Bigr)e^{-v}dv                              \\
     & = a (\partial_{\alpha_1}q^+)e^{-l_1u}u^{\alpha_1+\alpha_2-1} p_0^+(u) \\
     & \quad + a q^+ e^{-l_1u} u^{\alpha_1+\alpha_2-1}
    \int_0^{\infty} t^{\alpha_2-1}\Bigl(
    1 + \frac{t}{l_1+l_2}
    \Bigr)^{\alpha_1-1}
    \Bigl(
    \log\Bigl(1 + \frac{t}{l_1+l_2}\Bigr) + \log(u)
    \Bigr)e^{-ut}dt                                                          \\
     & = a e^{-l_1u} u^{\alpha_1+\alpha_2-1}\Bigl(
    (\partial_{\alpha_1}q^+) p_0^+(u)
    + q^+\log(u) p_0^+(u)
    + q^+p_1^+(u)
    \Bigr),
  \end{align*}
  \begin{align*}
    \partial_{\alpha_2}g(u; \theta)
     & = a (\partial_{\alpha_2}q^+)e^{-l_1u}
    \int_0^{\infty} v^{\alpha_2-1}\Bigl(
    u + \frac{v}{l_1+l_2}
    \Bigr)^{\alpha_1-1}e^{-v}dv                                              \\
     & \quad + aq^+e^{-l_1u}
    \int_0^{\infty} v^{\alpha_2-1}\Bigl(
    u + \frac{v}{l_1+l_2}
    \Bigr)^{\alpha_1-1}\log(v) e^{-v}dv                                      \\
     & = a (\partial_{\alpha_2}q^+)e^{-l_1u}u^{\alpha_1+\alpha_2-1} p_0^+(u) \\
     & \quad + a q^+ e^{-l_1u} u^{\alpha_1+\alpha_2-1}
    \int_0^{\infty} t^{\alpha_2-1}\Bigl(
    1 + \frac{t}{l_1+l_2}
    \Bigr)^{\alpha_1-1}
    \Bigl(
    \log(t) + \log(u)
    \Bigr)e^{-ut}dt                                                          \\
     & = a e^{-l_1u} u^{\alpha_1+\alpha_2-1}\Bigl(
    (\partial_{\alpha_2}q^+)p_0^+(u)
    + q^+\log(u) p_0^+(u)
    + q^+ p_2^+(u)
    \Bigr),
  \end{align*}
  \begin{align*}
    \partial_{l_1}g(u; \theta)
     & = a (\partial_{l_1}q^+) e^{-l_1u}
    \int_0^{\infty} v^{\alpha_2-1}\Bigl(
    u + \frac{v}{l_1+l_2}
    \Bigr)^{\alpha_1-1}e^{-v}dv                                                  \\
     & \quad + aq^+ \times (-u)e^{-l_1u}
    \int_0^{\infty} v^{\alpha_2-1}\Bigl(
    u + \frac{v}{l_1+l_2}
    \Bigr)^{\alpha_1-1}e^{-v}dv                                                  \\
     & \quad + aq^+e^{-l_1u}
    \int_0^{\infty} v^{\alpha_2-1}\Bigl(
    u + \frac{v}{l_1+l_2}
    \Bigr)^{\alpha_1-2}\frac{1-\alpha_1}{(l_1+l_2)^2}ve^{-v}dv                   \\
     & = a e^{-l_1u} u^{\alpha_1+\alpha_2-1}
    ((\partial_{l_1}q^+) p_0^+(u) - q^+  up_0^+(u))                              \\
     & \quad + (1-\alpha_1) (l_1+l_2)^{-2}a q^+e^{-l_1u} u^{\alpha_1+\alpha_2-1}
    \int_0^{\infty} t^{\alpha_2}\Bigl(
    1 + \frac{t}{l_1+l_2}
    \Bigr)^{\alpha_1-2}e^{-ut}dt                                                 \\
     & =a e^{-l_1u} u^{\alpha_1+\alpha_2-1}
    \Bigl((\partial_{l_1}q^+) p_0^+(u) - q^+  up_0^+(u) + (1-\alpha_1)(l_1+l_2)^{-2}q^+ p_3^+(u)\Bigr),
  \end{align*}
  and
  \begin{align*}
    \partial_{l_2}g(u; \theta)
     & = a (\partial_{l_2}q^+)e^{-l_1u}
    \int_0^{\infty} v^{\alpha_2-1}\Bigl(
    u + \frac{v}{l_1+l_2}
    \Bigr)^{\alpha_1-1}e^{-v}dv                                \\
     & \quad + aq^+e^{-l_1u}
    \int_0^{\infty} v^{\alpha_2-1}\Bigl(
    u + \frac{v}{l_1+l_2}
    \Bigr)^{\alpha_1-2}\frac{1-\alpha_1}{(l_1+l_2)^2}ve^{-v}dv \\
     & =a e^{-l_1u} u^{\alpha_1+\alpha_2-1}
    \Bigl((\partial_{l_2}q^+) p_0^+(u) + (1-\alpha_1)(l_1+l_2)^{-2}q^+ p_3^+(u)\Bigr),
  \end{align*}
  where
  \begin{align*}
    p_0^+(u) = \int_0^{\infty} t^{\alpha_2-1}\Bigl(
    1 + \frac{t}{l_1+l_2}
    \Bigr)^{\alpha_1-1}e^{-ut}dt,
  \end{align*}
  \begin{align*}
    p_1^+(u) = \int_0^{\infty} t^{\alpha_2-1}\Bigl(
    1 + \frac{t}{l_1+l_2}
    \Bigr)^{\alpha_1-1} \log\Bigl(
    1 + \frac{t}{l_1+l_2}
    \Bigr)e^{-ut}dt,
  \end{align*}
  \begin{align*}
    p_2^+(u) = \int_0^{\infty} t^{\alpha_2-1}\Bigl(
    1 + \frac{t}{l_1+l_2}
    \Bigr)^{\alpha_1-1} \log(t)e^{-ut}dt,
  \end{align*}
  and
  \begin{align*}
    p_3^+(u) =\int_0^{\infty} t^{\alpha_2}\Bigl(
    1 + \frac{t}{l_1+l_2}
    \Bigr)^{\alpha_1-2} e^{-ut}dt.
  \end{align*}

  For $u<0$,
  \begin{align*}
    \partial_{a}g(u; \theta)
     & = q^-e^{-l_2|u|}
    \int_0^{\infty} v^{\alpha_1-1}\Bigl(
    |u| + \frac{v}{l_1+l_2}
    \Bigr)^{\alpha_2-1}e^{-v}dv                                \\
     & = q^-e^{-l_2|u|} |u|^{\alpha_1+\alpha_2-1}
    \int_0^{\infty} t^{\alpha_1-1}\Bigl(
    1 + \frac{t}{l_1+l_2}
    \Bigr)^{\alpha_2-1}e^{-|u|t}dt                             \\
     & =  q^-e^{-l_2|u|} |u|^{\alpha_1+\alpha_2-1} p_0^-(|u|),
  \end{align*}
  \begin{align*}
    \partial_{\alpha_1}g(u; \theta)
     & = a (\partial_{\alpha_1}q^-)e^{-l_2|u|}
    \int_0^{\infty} v^{\alpha_1-1}\Bigl(
    |u| + \frac{v}{l_1+l_2}
    \Bigr)^{\alpha_2-1}e^{-v}dv                                                    \\
     & \quad + aq^-e^{-l_2|u|}
    \int_0^{\infty} v^{\alpha_1-1}\Bigl(
    |u| + \frac{v}{l_1+l_2}
    \Bigr)^{\alpha_2-1}\log(v) e^{-v}dv                                            \\
     & = a (\partial_{\alpha_1}q^-)e^{-l_2|u|}|u|^{\alpha_1+\alpha_2-1} p_0^-(|u|) \\
     & \quad + a q^- e^{-l_2|u|} |u|^{\alpha_1+\alpha_2-1}
    \int_0^{\infty} t^{\alpha_1-1}\Bigl(
    1 + \frac{t}{l_1+l_2}
    \Bigr)^{\alpha_2-1}
    \Bigl(
    \log(t) + \log(|u|)
    \Bigr)e^{-|u|t}dt                                                              \\
     & = a e^{-l_2|u|} |u|^{\alpha_1+\alpha_2-1}\Bigl(
    (\partial_{\alpha_1}q^-)p_0^-(|u|)
    + q^-\log(|u|) p_0^-(|u|)
    + q^- p_2^-(|u|)
    \Bigr),
  \end{align*}
  \begin{align*}
    \partial_{\alpha_2}g(u; \theta)
     & = a (\partial_{\alpha_2}q^-)e^{-l_2|u|}
    \int_0^{\infty} v^{\alpha_1-1}\Bigl(
    |u| + \frac{v}{l_1+l_2}
    \Bigr)^{\alpha_2-1}e^{-v}dv                                                    \\
     & \quad + aq^-e^{-l_2|u|}
    \int_0^{\infty} v^{\alpha_1-1}\Bigl(
    |u| + \frac{v}{l_1+l_2}
    \Bigr)^{\alpha_2-1}
    \log\Bigl(|u|+\frac{v}{l_1+l_2}\Bigr)e^{-v}dv                                  \\
     & = a (\partial_{\alpha_2}q^-)e^{-l_2|u|}|u|^{\alpha_2+\alpha_1-1} p_0^+(|u|) \\
     & \quad + a q^- e^{-l_2|u|} |u|^{\alpha_2+\alpha_1-1}
    \int_0^{\infty} t^{\alpha_1-1}\Bigl(
    1 + \frac{t}{l_1+l_2}
    \Bigr)^{\alpha_2-1}
    \Bigl(
    \log\Bigl(1 + \frac{t}{l_1+l_2}\Bigr) + \log(|u|)
    \Bigr)e^{-|u|t}dt                                                              \\
     & = a e^{-l_2|u|} |u|^{\alpha_2+\alpha_1-1}\Bigl(
    (\partial_{\alpha_2}q^-) p_0^+(|u|)
    + q^-\log(|u|) p_0^-(|u|)
    + q^-p_1^-(|u|)
    \Bigr),
  \end{align*}
  \begin{align*}
    \partial_{l_1}g(u; \theta)
     & = a (\partial_{l_1}q^-)e^{-l_2|u|}
    \int_0^{\infty} v^{\alpha_1-1}\Bigl(
    |u| + \frac{v}{l_1+l_2}
    \Bigr)^{\alpha_2-1}e^{-v}dv                                \\
     & \quad + aq^-e^{-l_2|u|}
    \int_0^{\infty} v^{\alpha_1-1}\Bigl(
    |u| + \frac{v}{l_1+l_2}
    \Bigr)^{\alpha_2-2}\frac{1-\alpha_2}{(l_1+l_2)^2}ve^{-v}dv \\
     & =a e^{-l_2|u|} |u|^{\alpha_1+\alpha_1-1}
    \Bigl((\partial_{l_1}q^-) p_0^-(|u|) + (1-\alpha_2)(l_1+l_2)^{-2}q^- p_3^-(|u|)\Bigr),
  \end{align*}
  and
  \begin{align*}
    \partial_{l_2}g(|u|; \theta)
     & = a (\partial_{l_2}q^-) e^{-l_2|u|}
    \int_0^{\infty} v^{\alpha_1-1}\Bigl(
    |u| + \frac{v}{l_1+l_2}
    \Bigr)^{\alpha_2-1}e^{-v}dv                                                      \\
     & \quad + aq^- \times (-|u|)e^{-l_2|u|}
    \int_0^{\infty} v^{\alpha_1-1}\Bigl(
    |u| + \frac{v}{l_1+l_2}
    \Bigr)^{\alpha_2-1}e^{-v}dv                                                      \\
     & \quad + aq^-e^{-l_2|u|}
    \int_0^{\infty} v^{\alpha_1-1}\Bigl(
    |u| + \frac{v}{l_1+l_2}
    \Bigr)^{\alpha_2-2}\frac{1-\alpha_2}{(l_1+l_2)^2}ve^{-v}dv                       \\
     & = a e^{-l_2|u|} |u|^{\alpha_2+\alpha_1-1}
    ((\partial_{l_2}q^-) p_0^+(|u|) - q^-  |u|p_0^+(|u|))                            \\
     & \quad + (1-\alpha_2) (l_1+l_2)^{-2}a q^-e^{-l_2|u|} |u|^{\alpha_2+\alpha_1-1}
    \int_0^{\infty} t^{\alpha_1}\Bigl(
    1 + \frac{t}{l_1+l_2}
    \Bigr)^{\alpha_2-2}e^{-|u|t}dt                                                   \\
     & =a e^{-l_2|u|} |u|^{\alpha_2+\alpha_1-1}
    \Bigl((\partial_{l_2}q^-) p_0^-(|u|) - q^-  |u|p_0^-(|u|) + (1-\alpha_2)(l_1+l_2)^{-2}q^- p_3^-(|u|)\Bigr),
  \end{align*}
  where
  \begin{align*}
    p_0^-(u) = \int_0^{\infty} t^{\alpha_1-1}\Bigl(
    1 + \frac{t}{l_1+l_2}
    \Bigr)^{\alpha_2-1}e^{-ut}dt,
  \end{align*}
  \begin{align*}
    p_1^-(u) = \int_0^{\infty} t^{\alpha_1-1}\Bigl(
    1 + \frac{t}{l_1+l_2}
    \Bigr)^{\alpha_2-1} \log\Bigl(
    1 + \frac{t}{l_1+l_2}
    \Bigr)e^{-ut}dt,
  \end{align*}
  \begin{align*}
    p_2^-(u) = \int_0^{\infty} t^{\alpha_1-1}\Bigl(
    1 + \frac{t}{l_1+l_2}
    \Bigr)^{\alpha_2-1} \log(t)e^{-ut}dt,
  \end{align*}
  and
  \begin{align*}
    p_3^-(u) =\int_0^{\infty} t^{\alpha_1}\Bigl(
    1 + \frac{t}{l_1+l_2}
    \Bigr)^{\alpha_2-2} e^{-ut}dt.
  \end{align*}

  By Lemma \ref{lem_abelian}, we have
  \begin{gather*}
    p_0^+(u) \sim \Gamma(\alpha_2) u^{-\alpha_2}, \quad
    p_1^+(u) \sim \Gamma(\alpha_2+1) \frac{1}{l_1+l_2}u^{-(\alpha_2+1)} , \\
    p_2^+(u) \sim -\Gamma(\alpha_2) u^{-\alpha_2} \log(u), \quad
    p_3^+(u) \sim \Gamma(\alpha_2+1) u^{-(\alpha_2+1)}
  \end{gather*}
  as $u\to\infty$, and
  \begin{gather*}
    p_0^-(|u|) \sim \Gamma(\alpha_1) |u|^{-\alpha_1}, \quad
    p_1^-(|u|) \sim \Gamma(\alpha_1+1) \frac{1}{l_1+l_2}|u|^{-(\alpha_1+1)} , \\
    p_2^-(|u|) \sim -\Gamma(\alpha_1) |u|^{-\alpha_1} \log(|u|), \quad
    p_3^-(|u|) \sim \Gamma(\alpha_1) |u|^{-(\alpha_1+1)}
  \end{gather*}
  as $u\to-\infty$.

  Suppose $c_1, \ldots, c_5\in\bbR$ and
  \begin{equation}\label{eq_gamma_positivity_prf_1}
    c_1\partial_{a}g(u; \theta^*)
    + c_2\partial_{\alpha_2}g(u; \theta^*)
    + c_3\partial_{\alpha_2}g(u; \theta^*)
    + c_4\partial_{l_1}g(u; \theta^*)
    + c_5\partial_{l_2}g(u; \theta^*) = 0, \quad u\in[-r, r].
  \end{equation}
  Then, by the identity theorem for analytic functions, the equation (\ref{eq_gamma_positivity_prf_1}) holds for all $u\neq 0$.
  Writing down the derivatives calculated above, we observe
  \begin{align}\label{al_linear_conbination_positive}
    0
     & =  c_1 d_{11} p_0^+(u)\nonumber                                        \\
     & \quad + c_2 (d_{21}p_0^+(u)
    + d_{22}\log(u) p_0^+(u)
    + d_{23} p_1^+(u))          \nonumber                                     \\
     & \quad
    + c_3(d_{31}p_0^+(u) + d_{32}\log(u) p_0^+(u) + d_{33}p_2^+(u)) \nonumber \\
     & \quad
    + c_4(d_{41}p_0^+(u) + d_{42} u p_0^+(u) + d_{43} p_3^+(u)) \nonumber     \\
     & \quad
    + c_5(d_{51}p_0^+(u) + d_{52} p_3^+(u)), \quad u>0
  \end{align}
  and
  \begin{align}\label{al_linear_conbination_negative}
    0
     & =  c_1 e_{11} p_0^-(|u|)\nonumber                                              \\
     & \quad + c_2 (e_{21}p_0^-(|u|)
    + e_{22}\log(|u|) p_0^-(|u|)
    + e_{23} p_2^-(|u|))          \nonumber                                           \\
     & \quad
    + c_3(e_{31}p_0^-(|u|) + e_{32}\log(|u|) p_0^-(|u|) + e_{33}p_1^-(|u|)) \nonumber \\
     & \quad
    + c_4(e_{41}p_0^-(|u|) + e_{42} p_3^-(|u|)) \nonumber                             \\
     & \quad
    + c_5(e_{51}p_0^-(|u|) + e_{52} |u| p_0^-(|u|) + e_{53}p_3^-(|u|)), \quad u<0
  \end{align}
  where
  $d_{11}, d_{22}, d_{23}, d_{32}, d_{33}, d_{42},
    e_{11}, e_{22}, e_{23}, e_{32}, e_{33}, e_{52}
    \in \bbR\setminus{\{0\}}$ \\ and
  $d_{21}, d_{31}, d_{41}, d_{43}, d_{51}, d_{52}, e_{21}, e_{31}, e_{41}, e_{42}, e_{51}, e_{53} \in \bbR$.


  First, since the function $u p_0^+(u)$ has the highest order as $u\to\infty$ in (\ref{al_linear_conbination_positive}) and $d_{42}\neq 0$, we have $c_4=0$.
  Similarly, because the function $|u| p_0^-(|u|)$ has the highest order as $u\to-\infty$ in (\ref{al_linear_conbination_negative}) and $e_{52}\neq 0$, we also have $c_5=0$.
  Next, we observe that the functions $\log(u) p_0^+(u)$ and $p_2^+(u)$ have the highest order
  among the remaining functions in (\ref{al_linear_conbination_positive}) as $u\to\infty$.
  Together with the fact that $d_{22}=d_{32}=d_{33}=a^*q^+(\theta^*)$,
  we obtain $c_2=0$.
  Repeating a similar argument for $u<0$, we also get $c_3=0$. (Notice that $e_{22}=e_{23}=e_{33}=a^*q^-(\theta^*)$.)
  Finally, because $p_0^+(u)$ is not constant, we obtain $c_1=0$ as well.
\end{proof}

\begin{proposition}
  The exponential kernel model satisfies the condition [ID].
\end{proposition}
\begin{proof}
  This is just a special case of Proposition \ref{prop_id_ns_gamma} (take $\alpha_1=\alpha_2=\alpha_1^*=\alpha_2^*=1$).
\end{proof}

\begin{proposition}
  The exponential kernel model satisfies the condition [ID2].
\end{proposition}
\begin{proof}
  This assertion follows from Proposition \ref{prop_id_ns_gamma} and the fact that the principal submatrix of a symmetric positive definite matrix is also a positive definite matrix.
\end{proof}

\bibliographystyle{spmpsci}      
\bibliography{ts}   

\newpage
\appendix
\section*{Appendix}
\section{Boundedness of the factorial cumulant densities for stationary Hawkes processes with bounded kernel}\label{sec:appendix_a}

In this section, we verify that a stationary univariate Hawkes process with a bounded kernel has bounded factorial cumulant densities up to any finite order. 

For $\psi:\mathbb{R}\to \mathbb{R}$, let $\|\psi\|_1$ and $\|\psi\|_\infty$  is the norm of $L^1(\mathbb{R})$ and $L^{\infty}(\mathbb{R})$, respectively.
Let \(N\) be a stationary univariate Hawkes process on \(\mathbb{R}\) with the baseline intensity $\mu > 0$ and the kernel function $\phi: \mathbb{R}\to\mathbb{R}_{\geq 0}$ satisfying $\|\phi\|_1<1$ and $\|\phi\|_{\infty} < \infty$.
The exponential kernel is such an example:
\(
\phi(t) = \nu \beta e^{-\beta t}1_{(0,\infty)}(t), t\in\mathbb{R},  0<\nu<1,\beta>0.
\)

For $N$, we write \(C_{[n]}\) for the
factorial cumulant measure of order \(n\),
and denote by \(\gamma_{[n]}\) its density with respect to the Lebesgue measure on \(\mathbb{R}^n\), when it exists:
\[
C_{[n]}(dt_1\cdots dt_n) = \gamma_{[n]}(t_1,\ldots,t_n)\,dt_1\cdots dt_n.
\]

\begin{proposition}\label{prop:Hawkes_NSiv}
Let \(N\) be the stationary univariate Hawkes process on $\mathbb{R}$ with a bounded kernel $\phi$. 
Then, for every integer \(n\ge2\), the factorial cumulant measure \(C_{[n]}\) of \(N\) admits a density
\(\gamma_{[n]}\) on \(\mathbb{R}^n\), and \(\gamma_{[n]}\) is bounded.
\end{proposition}
\begin{proof}
The proof is divided into three steps. We also illustrate the proof strategy in Figure \ref{fig:hawkes_proof_vis_fix}.
\paragraph{Step 1: Tree expansion and notation.}
Let \(n\ge2\) and \(\mathcal{T}_n\) be the set of all possible rooted trees in which every internal node that is not connected to any leaves has at least two children and with \(n\) leaves (``family tree'' in Jovanovi\'c et al.~\cite{jovanovic2015cumulants}).
Specializing the calculation in Jovanovi\'c et al.~\cite{jovanovic2015cumulants} to the univariate case, we have
\begin{equation}\label{eq_tree-expansion}
\gamma_{[n]}(t_1,\ldots,t_n)
 = \sum_{T\in\mathcal{T}_n} I_T(t_1,\ldots,t_n),
 \qquad (t_1,\ldots,t_n)\in\mathbb{R}^n,
\end{equation}
where the function \(I_T\) will be defined below.
Note that their definition of the cumulant density agrees with our definition of the factorial cumulant density when all $t_i, \ i=1, \ldots, n$ are mutually distinct.

Before presenting the definition of $I_T$, we first introduce several quantities following \cite{jovanovic2015cumulants}.
Let
\[
R(t)
 = \delta_0(t) + \sum_{m\ge1} \phi^{*m}(t),
 \qquad t\in\mathbb{R},
\]
and
\[
\Psi(t) = R(t) - \delta_0(t)
        = \sum_{m\ge1} \phi^{*m}(t),
\]
where \(\phi^{*m}\) denotes the \(m\)-fold self-convolution of \(\phi\) and
\(\delta_0\) is the Dirac mass at the origin. 
Since we assume $\phi$ is bounded and $\|\phi\|_1 < 1$, we have $\Psi\in  L^1(\mathbb{R})\cap L^{\infty}(\mathbb{R})$ thanks to Lemma \ref{lem_phi_l1_linfty}. 

We will also need some notations concerning the tree.
Fix \(n\ge2\) and \(T\in\mathcal{T}_n\).
Suppose $|T|$ is the total number of all nodes,
$M:=|T| - n$ is the number of all internal nodes (containing the root), 
$\mathcal{S}=\{s_1, \ldots, s_n\}$ is all of the leaf nodes, 
$\mathcal{V}=\{v_1, \ldots, v_M\}$ is all of the internal nodes,
$\mathcal{E} \subset \mathcal{V}\times (\mathcal{V}\sqcup \mathcal{S})$ is all of the edges (notice that each leaf is connected to an internal node), 
and $\mathcal{E}_{\text{int}} = \mathcal{E} \cap ( \mathcal{V}\times  \mathcal{V})$ is all of the internal edges.
We label the internal nodes $v_1, \ldots, v_M$ using a depth–first search, assigning smaller indices to nodes closer to the root. The induced total order is used purely for notation; for nodes that are incomparable in the tree order, their relative order has no structural significance.
For each $i = 1, \ldots, n$, let $m_i \in \{1,\ldots,M\}$ denote the index of the parent of the leaf $s_i$, that is, $v_{m_i}$ is the unique internal node adjacent to $s_i$.
For each internal edge $e \in \mathcal{E}_{\text{int}}$, we write
\(
  e = (v_{m_p(e)}, v_{m_d(e)})
\)
so that $m_p(e), m_d(e) \in \{1,\ldots,M\}$ are the indices of the two internal nodes incident to $e$, and we impose the convention
\(
  m_p(e) < m_d(e)
\)
Similarly, for edges incident to leaves $e\in E\setminus \mathcal{E}_{\text{int}}$, we write $e = (v_{m_p(e)}, s_{i_d(e)})$.

Then, we define $I_T(t_1, \ldots, t_n)$ as 
\begin{gather*}
    I_T(t_1, \ldots, t_n) = 
    \lambda_H
    \int_{\mathbb{R}^M} F_T(u_1, \ldots, u_M) 
    \Bigl\{\prod_{i=1}^n R(t_i - u_{m_i})\Bigr\}du_1\cdots du_M, \\
    F_T(u_1, \ldots, u_M) = \prod_{e\in \mathcal{E}_{\text{int}}}\Psi(u_{m_p(e)} - u_{m_d(e)}), \\
    \lambda_H = \frac{\mu}{1-\|\phi\|_1},
\end{gather*}
for $(t_1, \ldots, t_n)\in\mathbb{R^n}$ and have the relation (\ref{eq_tree-expansion}) by the algorithm in Section III in Jovanovi\'c et al.~\cite{jovanovic2015cumulants}.

Thus, to establish the proposition, it suffices to show that $I_T$ is bounded by a constant not depending on pairwise distinct $t_1, \ldots, t_n$.

\paragraph{Step 2: Decomposition of $I_T$ and identification of free integration variables.}
We observe
\begin{gather}\label{eq_IT_expansion}
    I_T(t_1, \ldots, t_n) = 
    \lambda_H\sum_{\mathcal{I}\subset \{1, \ldots, n\}}
    J_T(t_1, \ldots, t_n; \mathcal{I}), \\
    J_T(t_1, \ldots, t_n; \mathcal{I}) = \int_{\mathbb{R}^M} F_T(u_1, \ldots, u_M) 
    \Bigl\{\prod_{i\in \mathcal{I}} \Psi(t_i - u_{m_i})\Bigr\}
     \Bigl\{\prod_{i\in \mathcal{I}^c} \delta_0(t_i - u_{m_i})\Bigr\}
    du_1\cdots du_M, \nonumber
\end{gather}
where $\mathcal{I}^c = \{1, \ldots, n\} \setminus \mathcal{I}$ for $\mathcal{I}\subset \{1, \ldots, n\}$, by expanding $\prod_{i=1}^n R(t_i - u_{m_i})= \prod_{i=1}^n \Bigl(\delta_0(t_i - u_{m_i}) + \Psi(t_i - u_{m_i})\Bigr)$.
Therefore, our task reduces to proving, for each $\mathcal{I}$,
$J_T(t_1, \ldots, t_n; \mathcal{I})$ is bounded by a constant not depending on $(t_1, \ldots, t_n)\in\mathbb{R}^n$.

Fix $\mathcal{I}\subset\{1,\ldots,n\}$ and let $\mathcal{M}$ be the set of indices of the internal nodes that remain as free integration variables, that is,
\[
    \mathcal{M}
    = \{1,\ldots,M\} \setminus \{m_i : i \in \mathcal{I}^c\}.
\]

To keep track of which factors in the integrand of $J_T$ are of $\Psi$–type, we set
\[
  \mathcal E_\Psi
    = \mathcal E_{\mathrm{int}}
       \cup \{(v_{m_i},s_i) : i \in \mathcal I\}.
\]
Thus $\mathcal E_\Psi$ consists of all internal edges (coming from $F_T$) together with those leaf–edges whose factor in $\prod_{i=1}^n R(t_i-u_{m_i})$ is a
$\Psi$–term (corresponding to $i\in\mathcal I$). The Dirac factors
$\delta_0(t_i-u_{m_i})$ with $i\in\mathcal I^c$ do not contribute any $\Psi$–term and are only used to eliminate integration variables.

After using the Dirac masses $\{\delta_0(t_i-u_{m_i}): i\in\mathcal I^c\}$ to integrate out the variables $u_m$ with $m\notin\mathcal M$, each remaining $\Psi$–factor is associated with some edge $e\in\mathcal E_\Psi$, whose argument is a difference of two node variables (some of $u_k$ may have been replaced by the corresponding times $t_i$ by the Dirac delta). We then split
\[
  \mathcal E_{\mathrm{free}}
    := \{ e\in\mathcal E_\Psi :
          \text{at least one endpoint of $e$ is } v_m
          \text{ with } m\in\mathcal M \},
  \qquad
  \mathcal E_{\mathrm{fixed}}
    := \mathcal E_\Psi \setminus \mathcal E_{\mathrm{free}}.
\]
For each $e\in\mathcal E_{\mathrm{free}}$ we write
\[
  \rho_e(u_{\mathcal M};t_1,\ldots,t_n)
\]
for the corresponding argument of $\Psi$, obtained by replacing all variables
$u_m$ with $m\notin\mathcal M$ by the appropriate times $t_i$ and keeping
$u_m$ with $m\in\mathcal M$ as free variables. In particular,
$\rho_e(\,\cdot\,;t_1,\ldots,t_n)$ genuinely depends on $u_{\mathcal M}$.
For $e\in\mathcal E_{\mathrm{fixed}}$ the argument depends only on
$(t_1,\ldots,t_n)$, and we denote it by $\rho_e(t_1,\ldots,t_n)$.

With this notation, we can rewrite $J_T$ exactly as
\begin{equation}\label{eq_JT_decomposed}
  J_T(t_1,\ldots,t_n;\mathcal I)
  = \int_{\mathbb R^{\#\mathcal M}}
      G_T(u_{\mathcal M};t_1,\ldots,t_n)\,du_{\mathcal M},
\end{equation}
where
\[
  G_T(u_{\mathcal M};t_1,\ldots,t_n)
  =
    \prod_{e\in\mathcal E_{\mathrm{free}}}
      \Psi\bigl(\rho_e(u_{\mathcal M};t_1,\ldots,t_n)\bigr)
    \prod_{e\in\mathcal E_{\mathrm{fixed}}}
      \Psi\bigl(\rho_e(t_1,\ldots,t_n)\bigr).
\]
Here, we have used the fact that all factors associated with edges carry only a Dirac term have already been accounted for when integrating out the corresponding
variables, so that only edges in $\mathcal E_\Psi$ contribute a $\Psi$–factor.
Since $\Psi$ is bounded, we have
\begin{equation}\label{eq_ap_hawkes_what_to_show}
    J_T(t_1,\ldots,t_n;\mathcal I) \leq \|\Psi\|_{\infty}^{\#\mathcal{E}_{\text{fixed}}}
    \int_{\mathbb{R}^{\#\mathcal{M}}} 
    \prod_{e \in \mathcal{E}_{\text{free}}} \Psi(\rho_e(u_{\mathcal M};t_1,\ldots,t_n)) \prod_{k \in \mathcal{M}} du_k.
\end{equation}
from \eqref{eq_JT_decomposed}.

\paragraph{Step 3: Construction of dedicated edges and uniform bound.}
The key idea of the proof lies in the strategy for handling the integral in (\ref{eq_ap_hawkes_what_to_show}).
To integrate out a specific variable $u_k \in \mathcal{M}$, we rely on the $L^1$-norm estimate $\int |\Psi(u)|du = \|\Psi\|_1$, which requires consuming at least one factor of $\Psi$ that involves $u_k$.
However, the integrand contains multiple edges. The central question is: can we assign a unique, dedicated edge to each integration variable $u_k$ to facilitate this process without overlap?

To identify the ``dedicated edge'' for each $u_k, k\in\mathcal{M}$,
we construct an injective mapping $\chi: \mathcal{M} \to \mathcal{E}_{\text{free}}$.
By the definition of $\mathcal{E}_{\text{free}}$, any edge $e$ with a parent $m_p(e) \in \mathcal{M}$ automatically belongs to $\mathcal{E}_{\text{free}}$, regardless of whether its child node is free or fixed.
Consequently, even though the set of edges $\mathcal{E}_{\text{free}}$ may not form a single connected tree (it may be a forest due to the fixed nodes), the local branching property guarantees that for every $k \in \mathcal{M}$, we can select exactly one outgoing edge $e_k \in \mathcal{E}_{\text{free}}$ such that $m_p(e_k) = k$, and then we define the injection as $\chi(k) = e_k$.

We now continue the evaluation of $J_T$ from \eqref{eq_ap_hawkes_what_to_show}.
Let $\mathcal{E}_{\text{ded}} = \chi(\mathcal{M}) \subset \mathcal{E}_{\text{free}}$ be the set of these selected edges, and let $\mathcal{E}_{\text{rem}} = \mathcal{E}_{\text{free}} \setminus \mathcal{E}_{\text{ded}}$ be the remaining free edges.

First, we bound the terms associated with $\mathcal{E}_{\text{rem}}$ by the sup-norm $\|\Psi\|_\infty$.
Then, together with \eqref{eq_ap_hawkes_what_to_show}, we have
\[
     J_T(t_1,\ldots,t_n;\mathcal I) \leq
    \|\Psi\|_\infty^{\#\mathcal{E}_{\text{rem}} + \#\mathcal{E}_{\text{fixed}}}
    \int_{\mathbb{R}^{\#\mathcal{M}}} 
    \prod_{k \in \mathcal{M}} |\Psi(\rho_{\chi(k)}(u_{\mathcal M};t_1,\ldots,t_n))| \prod_{k \in \mathcal{M}} du_k.
\]

Next, we evaluate this integral sequentially from the root to the leaves in $\mathcal{E}$.
Let us integrate with respect to $u_k \in \mathcal{M}$. The term associated with the dedicated edge $\chi(k)$ takes the form $\Psi(u_k - z_k)$, where $z_k$ corresponds to a child node (either a leaf $t_i$ or another internal node).
Crucially, because we proceed from parent to child, at the step where we integrate $u_k$, the variable $z_k$ acts as a constant (either it is a fixed time $t_i$, or a variable deeper in the tree that has not yet been integrated).
By the translation invariance of the Lebesgue measure, we have
\[
    \int_{\mathbb{R}} |\Psi(u_k - z_k)| du_k = \|\Psi\|_1.
\]
Repeating this procedure for all $k \in \mathcal{M}$, we integrate out all variables and obtain the factor $(\|\Psi\|_1)^{\#\mathcal{M}}$.
Combining all estimates, we conclude that \[
    |J_T(t_1, \ldots, t_n; \mathcal{I})| \le \|\Psi\|_\infty^{\#\mathcal{E}_{\text{fixed}} + \#\mathcal{E}_{\text{rem}}} \|\Psi\|_1^{\#\mathcal{M}},
\] so that $\gamma_{[n]}$ is bounded.
\end{proof}
\halflineskip

\begin{lemma}\label{lem_phi_l1_linfty}
Suppose that $\phi\in L^1(\mathbb{R})\cap L^\infty(\mathbb{R})$ with $\|\phi\|_1<1$.
Then the function
\[
\Psi(t) = \sum_{m\ge1} \phi^{*m}(t), \qquad t\in\mathbb{R},
\]
belongs to $L^1(\mathbb{R})\cap L^\infty(\mathbb{R})$.
\end{lemma}

\begin{proof}
Let $a := \|\phi\|_1<1$ and $M := \|\phi\|_\infty<\infty$.
By Young's inequality for convolutions, we have
\[
\|\phi^{*m}\|_1 \le \|\phi\|_1^m = a^m,
\qquad m\ge1.
\]
Hence
\[
\sum_{m=1}^\infty \|\phi^{*m}\|_1
 \le \sum_{m=1}^\infty a^m
 = \frac{a}{1-a}
 < \infty.
\]
Therefore the series $\sum_{m\ge1}\phi^{*m}$ converges in $L^1(\mathbb{R})$, and in particular
$\Psi\in L^1(\mathbb{R})$.

Next we show that $\Psi\in L^\infty(\mathbb{R})$.
We observe
\begin{equation}\label{eq:phi_m_Linf_bound}
\|\phi^{*m}\|_\infty \le M\,a^{m-1},
 \qquad m\ge1.
\end{equation}
using Young's inequality recursively.
Then, we have
\[
\|\Psi\|_\infty 
\le \sum_{m=1}^\infty \|\phi^{*m}\|_\infty
 \le \sum_{m=1}^\infty M\,a^{m-1}
 = \frac{M}{1-a}
 <\infty.
\]
\end{proof}
\halflineskip

\begin{figure}[htbp]
    \centering
    \tikzset{
        internal/.style={circle, draw=black, thick, minimum size=7mm, inner sep=0pt, font=\small},
        leaf/.style={rectangle, draw=black, thick, minimum size=6mm, inner sep=2pt, font=\small},
        fixed/.style={fill=gray!30},
        dedicated/.style={->, >=Stealth, line width=2.5pt, draw=black},
        remaining/.style={->, >=Stealth, thick, draw=black},
        fixededge/.style={->, >=Stealth, thick, draw=gray!80},
        level 1/.style={sibling distance=3.0cm, level distance=1.5cm},
        level 2/.style={sibling distance=1.2cm, level distance=1.5cm},
    }

    \begin{subfigure}[b]{0.32\textwidth}
        \centering
        \begin{tikzpicture}[scale=0.65, transform shape] 
            \node[internal] (u1) {$u_1$}
                child {node[internal] (u2) {$u_2$}
                    child {node[leaf] (t1) {$t_1$}}
                    child {node[leaf] (t2) {$t_2$}}
                }
                child {node[internal] (u3) {$u_3$}
                    child {node[leaf] (t3) {$t_3$}}
                    child {node[leaf] (t4) {$t_4$}}
                };
        \end{tikzpicture}
        \caption{an example of $T \in \mathcal{T}_4$}
        \label{fig:tree_setup_fix}
    \end{subfigure}
    \hfill
    \begin{subfigure}[b]{0.32\textwidth}
        \centering
        \begin{tikzpicture}[scale=0.65, transform shape]
            \node[internal] (u1) {$u_1$}
                child {node[internal] (u2) {$u_2$}
                    child {node[leaf] (t1) {$t_1$}}
                    child {node[leaf] (t2) {$t_2$}}
                }
                child {node[internal, fixed] (u3) {$u_3$}
                    child {node[leaf, fixed] (t3) {$t_3$}}
                    child {node[leaf] (t4) {$t_4$}}
                };
            
            \node[draw=none, anchor=north, align=left, font=\scriptsize, yshift=-1.0cm] at (u1) {
                \textbf{Legend:}\\
                \tikz[baseline=-0.5ex]\node[internal, scale=0.6]{}; Free ($\mathcal{M}$)\\
                \tikz[baseline=-0.5ex]\node[internal, fixed, scale=0.6]{}; Fixed ($\mathcal{M}^c$)
            };
        \end{tikzpicture}
        \caption{Fix the variables by $\mathcal{I}$}
        \label{fig:tree_decomp_fix}
    \end{subfigure}
    \hfill
    \begin{subfigure}[b]{0.32\textwidth}
        \centering
        \begin{tikzpicture}[scale=0.65, transform shape]
            \node[internal] (u1) {$u_1$}
                child {node[internal] (u2) {$u_2$}
                    child {node[leaf] (t1) {$t_1$}}
                    child {node[leaf] (t2) {$t_2$}}
                }
                child {node[internal, fixed] (u3) {$u_3$}
                    child {node[leaf] (t4) {$t_4$}}
                };

            \draw[dedicated] (u1) -- (u2);
            \draw[dedicated] (u2) -- (t1);
            \draw[remaining] (u1) -- (u3);
            \draw[remaining] (u2) -- (t2);
            \draw[fixededge] (u3) -- (t4);
            
            \node[draw=none, anchor=north, align=left, font=\scriptsize, yshift=-1.0cm] at (u1) {
                \textbf{Edge Types:}\\
                \tikz[baseline=-0.5ex]\draw[dedicated] (0,0) -- (0.5,0); Dedicated\\
                \tikz[baseline=-0.5ex]\draw[remaining] (0,0) -- (0.5,0); Remaining\\
                \tikz[baseline=-0.5ex]\draw[fixededge] (0,0) -- (0.5,0); Fixed
            };
        \end{tikzpicture}
        \caption{Integration Strategy $\chi$}
        \label{fig:tree_proof_fix}
    \end{subfigure}
    
    \caption{Visual illustration of the proof strategy for $n=4$ with the tree.
    (To keep the figure readable, we label internal nodes by $(u_k)$ and leaves 
    by $(t_i)$ respectively to $(v_k)$ and $(s_i)$.)
    In this case, $T\in\mathcal{T}_4$ and the choice $\mathcal{I}=\{1,2,4\}$. 
    The internal nodes are labeled by the integration variables $u_1, u_2, u_3$, and the leaves by the times $t_1, \ldots, t_4$. 
    For this choice of $\mathcal{I}$ we have $\mathcal{M}=\{1,2\}$ so that $u_3$ is fixed by the constraint $\delta_0(t_3-u_3)$, and the edge sets are
    $\mathcal{E}=\{(u_1,u_2),(u_1,u_3),(u_2,t_1),(u_2,t_2),(u_3,t_3),(u_3,t_4)\}$,
    $\mathcal{E}_{\mathrm{int}}=\{(u_1,u_2),(u_1,u_3)\}$,
    $\mathcal{E}_\Psi=\mathcal{E}_{\mathrm{int}}\cup\{(u_2,t_1),(u_2,t_2),(u_3,t_4)\}$,
    $\mathcal{E}_{\mathrm{free}}=\{(u_1,u_2),(u_1,u_3),(u_2,t_1),(u_2,t_2)\}$,
    and $\mathcal{E}_{\mathrm{fixed}}=\{(u_3,t_4)\}$. In panel~(c) we choose the injection $\chi:\mathcal{M}\to\mathcal{E}_{\mathrm{free}}$ given by $\chi(1)=(u_1,u_2)$ and $\chi(2)=(u_2,t_1)$; the corresponding dedicated edges $\mathcal{E}_{\mathrm{ded}}=\chi(\mathcal{M})$ are drawn as thick arrows, the remaining free edges $\mathcal{E}_{\mathrm{rem}}=\mathcal{E}_{\mathrm{free}}\setminus\mathcal{E}_{\mathrm{ded}}$ as thin solid arrows, and the fixed edges as gray arrows.}

    \label{fig:hawkes_proof_vis_fix}
\end{figure}

\section{The estimated parameters of the models in Section \ref{sec_realdata}}

\begin{table}[htbp]
      \centering
      \begin{tabular}{rrrrrrr}
            \toprule
            code1 & code2 & $a$  & $\alpha_1$ & $\alpha_2$ & $l_1$ & $l_2$ \\
            \midrule
            7201  & 7203  & 3.15 & 0.281      & 0.291      & 0.569 & 0.712 \\
            8306  & 8411  & 3.89 & 0.352      & 0.28       & 1.48  & 0.627 \\
            8031  & 8058  & 4.63 & 0.258      & 0.3        & 1.15  & 1.69  \\
            \bottomrule
      \end{tabular}
      \caption{NBNSP-G, buy orders}
      \label{table_est_params_real_start}
\end{table}

\begin{table}[htbp]
      \centering
      \begin{tabular}{rrrrr}
            \toprule
            code1 & code2 & $a$  & $l_1$ & $l_2$ \\
            \midrule
            7201  & 7203  & 2.09 & 5.37  & 5.65  \\
            8306  & 8411  & 2.71 & 6.91  & 6.55  \\
            8031  & 8058  & 3.41 & 10.8  & 10    \\
            \bottomrule
      \end{tabular}
      \caption{NBNSP-E, buy orders}
\end{table}

\begin{table}[htbp]
      \centering
      \begin{tabular}{rrrrrrrrrr}
            \toprule
            code1 & code2 & $\mu_1$ & $\mu_2$ & $\alpha_{11}$ & $\alpha_{12}$ & $\alpha_{21}$ & $\alpha_{22}$ & $\beta_1$ & $\beta_2$ \\
            \midrule
            7201  & 7203  & 0.161   & 0.124   & 27            & 2.05          & 1.46          & 15.4          & 80.7      & 55.3      \\
            8306  & 8411  & 0.159   & 0.0714  & 25.9          & 1.34          & 0.409         & 23.1          & 89.1      & 95.7      \\
            8031  & 8058  & 0.0566  & 0.095   & 35.1          & 2.33          & 3.68          & 27.5          & 127       & 99.1      \\
            \bottomrule
      \end{tabular}
      \caption{BHP-E, buy orders}
\end{table}

\begin{table}[htbp]
      \centering
      \begin{tabular}{rrrrrrr}
            \toprule
            code1 & code2 & $a$  & $\alpha_1$ & $\alpha_2$ & $l_1$ & $l_2$ \\
            \midrule
            7201  & 7203  & 4.42 & 0.258      & 0.309      & 0.234 & 0.842 \\
            8306  & 8411  & 4.13 & 0.242      & 0.313      & 0.481 & 0.829 \\
            8031  & 8058  & 3.66 & 0.309      & 0.269      & 2.54  & 1.74  \\
            \bottomrule
      \end{tabular}
      \caption{NBNSP-G, sell orders}
\end{table}

\begin{table}[htbp]
      \centering
      \begin{tabular}{rrrrr}
            \toprule
            code1 & code2 & $a$  & $l_1$ & $l_2$ \\
            \midrule
            7201  & 7203  & 2.79 & 4.11  & 4.73  \\
            8306  & 8411  & 2.65 & 7.26  & 5.99  \\
            8031  & 8058  & 2.71 & 13.7  & 17.1  \\
            \bottomrule
      \end{tabular}
      \caption{NBNSP-E, sell orders}
\end{table}

\begin{table}[htbp]
      \centering
      \begin{tabular}{rrrrrrrrrr}
            \toprule
            code1 & code2 & $\mu_1$ & $\mu_2$ & $\alpha_{11}$ & $\alpha_{12}$ & $\alpha_{21}$ & $\alpha_{22}$ & $\beta_1$ & $\beta_2$ \\
            \midrule
            7201  & 7203  & 0.156   & 0.118   & 25.7          & 2.57          & 2.3           & 20.3          & 81.9      & 72.6      \\
            8306  & 8411  & 0.156   & 0.0799  & 25.9          & 1.67          & 1.47          & 17.2          & 81        & 81.6      \\
            8031  & 8058  & 0.0613  & 0.111   & 27.3          & 1.53          & 3.8           & 31.2          & 88.8      & 126       \\
            \bottomrule
      \end{tabular}
      \caption{BHP-E, sell orders}
      \label{table_est_params_real_end}
\end{table}


\end{document}